\documentclass{amsart}

\usepackage{amssymb}
\usepackage{amsfonts}
\usepackage{amsmath}
\usepackage{graphicx}
\usepackage{mathrsfs}
\usepackage{dsfont}
\usepackage{amscd}
\usepackage{multirow}
\usepackage{mathpazo}
\usepackage[all]{xy}
\usepackage[scaled]{helvet} 
\usepackage{courier} 
\normalfont
\usepackage{calligra}
\usepackage[T1]{fontenc}
\usepackage{verbatim}
\usepackage{hyperref}

\newcommand{\N}{{\mathds{N}}}
\newcommand{\Z}{{\mathds{Z}}}

\newcommand{\R}{{\mathds{R}}}
\newcommand{\C}{{\mathds{C}}}

\newcommand{\D}{{\mathfrak{D}}}
\newcommand{\A}{{\mathfrak{A}}}
\newcommand{\B}{{\mathfrak{B}}}

\newcommand{\Lip}{{\mathsf{L}}}
\newcommand{\Hilbert}{{\mathscr{H}}}

\newcommand{\dist}{{\mathrm{dist}}}
\newcommand{\propinquity}{{\mathsf{\Lambda}}}

\newcommand{\Kantorovich}[1]{{\mathsf{mk}_{#1}}}

\newcommand{\Haus}[1]{{\mathsf{Haus}_{#1}}}

\newcommand{\StateSpace}{{\mathscr{S}}}

\newcommand{\mongekant}{{Mon\-ge-Kan\-to\-ro\-vich metric}}

\newcommand{\Lqcms}{{\JLL} quantum compact metric space}
\newcommand{\qcms}{quantum compact metric space}
\newcommand{\lcqms}{quantum locally compact metric space}

\newcommand{\unit}{1}

\newcommand{\sa}[1]{{\mathfrak{sa}({#1})}}
\newcommand{\mulip}{{\mathfrak{Lip}_1}}

\newcommand{\Adm}{{\mathrm{Adm}}}

\newcommand{\JLL}{Lei\-bniz}

\newcommand{\dom}[1]{{\operatorname*{dom}({#1})}}
\newcommand{\diam}[2]{{\mathrm{diam}\left({#1},{#2}\right)}}

\newcommand{\bridgeset}[2]{{\text{\calligra Bridges}\left({#1}\rightarrow{#2}\right)}}
\newcommand{\Lipball}[2]{{\mathfrak{Lip}_{#1}\left({#2}\right)}}
\newcommand{\bridgereach}[2]{{\varrho\left({#1}\middle|{#2}\right)}}
\newcommand{\bridgeheight}[2]{{\varsigma\left({#1}\middle|{#2}\right)}}
\newcommand{\bridgelength}[2]{{\lambda\left({#1}\middle|{#2}\right)}}
\newcommand{\treklength}[1]{{\lambda\left({#1}\right)}}

\newcommand{\startrekset}[3]{{\text{\calligra Treks}\left(\left({#2}\right)\stackrel{#1}{\longrightarrow}\left({#3}\right)\right)}}
\newcommand{\bridgenorm}[2]{{\mathsf{bn}_{ {#1}  }\left({#2}\right)}}
\newcommand{\treknorm}[2]{{\mathsf{tn}_{#1}\left({#2}\right)}}
\newcommand{\itineraries}[3]{{\text{\calligra Itineraries}\left({#2}\stackrel{#1}{\longrightarrow} {#3}  \right)}}
\newcommand{\lambdaitineraries}[4]{{\text{\calligra Itineraries}\left( {#2}\stackrel{{#1} }{\longrightarrow}{#3}\middle|{#4}  \right)}}
\newcommand{\Jordan}[2]{{{#1}\circ{#2}}} 
\newcommand{\Lie}[2]{{\left\{{#1},{#2}\right\}}} 
\newcommand{\targetsetbridge}[3]{{\mathfrak{t}_{#1}\left({#2}\middle\vert{#3}\right)}}
\newcommand{\targetsettrek}[3]{{\mathfrak{T}_{#1}\left({#2}\middle\vert{#3}\right)}}
\newcommand{\LCQMS}{{\mathcal{L}^\ast}}

\theoremstyle{plain}
\newtheorem{theorem}{Theorem}[section]

\newtheorem{claim}[theorem]{Claim}

\newtheorem{corollary}[theorem]{Corollary}

\newtheorem{lemma}[theorem]{Lemma}
\newtheorem{proposition}[theorem]{Proposition}

\newtheorem{theorem-definition}[theorem]{Theorem-Definition}

\theoremstyle{definition}
\newtheorem{definition}[theorem]{Definition}

\newtheorem{notation}[theorem]{Notation}

\newtheorem{convention}[theorem]{Convention}

\theoremstyle{remark}

\newtheorem{example}[theorem]{Example}

\newtheorem{remark}[theorem]{Remark}

\renewcommand{\geq}{\geqslant}
\renewcommand{\leq}{\leqslant}

\numberwithin{equation}{section}

\begin{document}


\title[The Quantum Gromov-Haus\-dorff Propinquity]{The Quantum Gromov-Haus\-dorff Propinquity}
\author{Fr\'{e}d\'{e}ric Latr\'{e}moli\`{e}re}
\email{frederic@math.du.edu}
\urladdr{http://www.math.du.edu/\symbol{126}frederic}
\address{Department of Mathematics \\ University of Denver \\ Denver CO 80208}

\date{\today}

\begin{abstract}
We introduce the quantum Gromov-Haus\-dorff pro\-pin\-qui\-ty, a new distance between quantum compact metric spaces, which extends the Gromov-Haus\-dorff distance to noncommutative geometry and strengthens Rieffel's quantum Gro\-mov-Haus\-dorff distance and Rieffel's proximity by making *-iso\-morphism a necessary condition for distance zero, while being well adapted to Leibniz seminorms. This work offers a natural solution to the long-standing problem of finding a framework for the development of a theory of Leibniz Lip-norms over C*-algebras. 
\end{abstract}

\maketitle

\newcommand{\physics}[1]{#1}



\section{Introduction}

Noncommutative geometry is founded on the allegory that noncommutative algebras should be studied as generalizations of algebras of functions over spaces endowed with some form of geometry, and finds its roots in the work of Connes \cite{Connes80,Connes}. Noncommutative metric geometry, in particular, is the study of noncommutative generalizations of algebras of Lipschitz functions over metric spaces \cite{Connes89,Rieffel99,Weaver99}. A fascinating consequence of the introduction of noncommutative metric data on C*-algebras, encoded in special seminorms called Lip-norms, is the possibility to define topologies on the class of all quantum metric spaces. In turn, such topologies establish a framework for new forms of approximations and the development of new techniques inspired, in large part, by the mathematical physics literature. Moreover, this new framework allows for generalizations of techniques from metric geometry \cite{Gromov} to the realm of quantum spaces. A noncommutative generalization of the flexible theory of geometry on metric spaces may indeed be well-suited to the pathologies of quantum spaces. 

We introduce in this paper the quantum Gromov-Haus\-dorff pro\-pin\-qui\-ty, a new distance between quantum compact metric spaces, which extends the Gro\-mov-Haus\-dorff distance to noncommutative geometry and strengthens Rieffel's quantum Gro\-mov-Haus\-dorff distance and Rieffel's proximity by making *-iso\-mor\-phism a necessary condition for distance zero, while being well adapted to Leibniz seminorms. Our new metric thus meets several challenges raised by the current research in noncommutative metric geometry \cite{Rieffel09,Rieffel10,Rieffel10c,Rieffel11,Rieffel12}. Indeed, the original notion of convergence for {\qcms s} was introduced in \cite{Rieffel00} by Rieffel, who constructed the quantum Gro\-mov-Haus\-dorff distance on the class of all compact quantum metric spaces. This distance, informally, did not see the multiplicative structure of the underlying C*-algebras.

Now, once convergence of some sequence of {\qcms s} has been established for the quantum Gro\-mov-Haus\-dorff distance, it then becomes natural to investigate the behavior of sequences of associated structures, such as noncommutative vector bundles (i.e. finitely generated projective modules). However, recent research on this topic \cite{Rieffel09,Rieffel10,Rieffel10c,Rieffel11,Rieffel12} reveals that both the C*-algebraic structure of the underlying space defining a {\qcms s}, as well as the Leibniz property for the Lip-norm, play an important role. We shall call a {\qcms} defined as pairs of a C*-algebra and a Leibniz Lip-norm a \emph{\Lqcms} in this paper, and we note that the class $\LCQMS$ of {\Lqcms s} includes, for instance, the class of compact C*-metric spaces introduced by Rieffel. 

By construction, the computation of the quantum Gro\-mov-Haus\-dorff distance, even between two {\Lqcms s}, involve {\qcms s} which are not Leibniz. This observation raises the challenge of defining a metric on $\LCQMS$ whose computation only involves spaces in $\LCQMS$, and which extend, in some sense, the Gro\-mov-Haus\-dorff distance, while dominating Rieffel's quantum Gro\-mov-Haus\-dorff distance --- as the latter has a clear interpretation in terms of states of the quantum system described by the {\qcms s}. Moreover, we consider that such a distance should remember the C*-algebraic distance, which should mean at least that distance zero only occurs between *-isomorphic C*-algebras. As mentioned in \cite{Rieffel10,Rieffel10c}, for instance, the construction of such a metric proved elusive for some time. In \cite{Rieffel10c}, Rieffel introduced a pseudo-semi metric, called the quantum proximity, as a step toward answering this challenge. Yet the quantum proximity is unlikely to satisfy the triangle inequality.

We demonstrate that the quantum Gro\-mov-Haus\-dorff propinquity which we introduce in this paper indeed dominates Rieffel's quantum Gro\-mov-Haus\-dorff distance, is dominated by the Gro\-mov-Haus\-dorff distance when restricted to the class of classical metric spaces, and can be seen as a well-behaved form of Rieffel's proximity, from which it takes its name. Moreover, the quantum propinquity possesses the desired coincidence property. Thus, we propose our new metric as a potential answer to the challenge of extending the Gro\-mov-Haus\-dorff distance to the class $\LCQMS$ of {\Lqcms s}. We then note that, for our quantum propinquity, the quantum tori form a continuous family, and can be approximated by finite dimensional C*-algebras, following our work in \cite{Latremoliere05} and a remark in \cite{li03}.


A \emph{compact quantum metric space} $(\A,\Lip)$, as defined by Rieffel \cite{Rieffel98a,Rieffel99} is a pair of an order unit-space $\A$ and a densely defined seminorm on $\A$, called a Lip-norm,  whose dual seminorm induces a metric on the state space $\StateSpace(\A)$ of $\A$ whose topology is the weak* topology. The classical picture is given by a pair $(C_\R(X),\mathsf{Lip})$ of the algebra of $\R$-valued continuous functions on a compact metric space $(X,\mathsf{d})$ equipped with the Lipschitz seminorm $\mathsf{Lip}$ associated with the metric $\mathsf{d}$. Noncommutative examples include the (self-adjoint part of) quantum tori and other quantum homogeneous spaces \cite{Rieffel98a}, hyperbolic group C*-algebras \cite{Ozawa05,Rieffel02}, Connes-Landi-Dubois-Violette quantum spheres \cite{Landi01,violette02,Li05}, among others. We generalized this structure to {\lcqms s} in \cite{Latremoliere12b}.

Rieffel introduces his \emph{quantum Gro\-mov-Haus\-dorff distance} in \cite{Rieffel00}, as a generalization of the Gro\-mov-Haus\-dorff distance \cite{Gromov81,Gromov} to compact quantum metric spaces. He showed, as a first example, that quantum tori formed a continuous family for this metric. The construction of the quantum Gro\-mov-Haus\-dorff distance proceeds as follows: given two compact quantum metric spaces $(\A,\Lip_\A)$ and $(\B,\Lip_\B)$, we consider the set $\Adm(\Lip_\A,\Lip_\B)$ of all Lip-norms on $\A\oplus\B$ whose quotients to $\A$ and $\B$ are, respectively, $\Lip_\A$ and $\Lip_\B$. For any such admissible Lip-norm $\Lip$, one may consider the Haus\-dorff distance between $\StateSpace(\A)$ and $\StateSpace(\B)$ identified with their isometric copies in $\StateSpace(\A\oplus\B)$, where the state spaces $\StateSpace(\A)$, $\StateSpace(\B)$ and $\StateSpace(\A\oplus\B)$ are equipped with the metrics dual to, respectively, $\Lip_\A$, $\Lip_\B$ and $\Lip$. The infimum of these Haus\-dorff distances over all possible choices of $\Lip \in \Adm(\Lip_\A,\Lip_\B)$ is the quantum Gro\-mov-Haus\-dorff distance between $(\A,\Lip_\A)$ and $(\B,\Lip_\B)$.

As an example of an interesting application of this distance, we proved in \cite{Latremoliere05} that we can also use Rieffel's metric to approximate quantum tori by finite dimensional C*-algebras. Our work was motivated by such papers in mathematical physics as \cite{Connes97,Seiberg99,tHooft02,Zumino98}, where such finite approximations of quantum tori surfaced without any formal framework. Of course, quantum tori are not AF algebras --- rather, the closest form of approximation of this nature, due to Elliott and Evans, is that irrational rotation algebras, i.e. two dimensional simple quantum tori, are AT algebras, or inductive limits of circle algebras \cite{Elliott93b}. Thus our work was one of the first construction to provide some evidence that the intuition often encountered in mathematical physics could be formalized. Other examples of convergence results for Rieffel's metric were established regarding Dubois-Violette quantum spheres \cite{Li06} and finite dimensional approximations of coadjoint orbits of compact Lie groups \cite{Rieffel01,Rieffel10c}.

\physics{More generally, the study of noncommutative space-times finds its roots in a paradox which emerges from attempting to bring quantum physics and general relativity together: informally, the uncertainty principle of quantum mechanics suggests that measurements at very small scale require very large momenta, and in turn, Einstein's field equation implies that large localized momenta would curve space-time, so that measurement of an arbitrary small volume would actually create a small black hole whose event horizon would hide the volume one attempts to measure. A discussion of this paradox and a first, simple noncommutative space-time model to address it, was devised in  \cite{Doplicher95}. In essence, a noncommutative space-time is a structure where space-time measurements are themselves subject to uncertainty, thus preventing arbitrary precision and offering a way out of the stated paradox --- while reflecting the fact that the quantum uncertainty between momentum and position paired with the relativistic connection between momentum and geometry would naturally lead one to postulate an uncertainty principle at the level of geometry. In a different direction, Connes introduced a noncommutative standard model, which involves metric considerations as well \cite{Connes}.

On the other hand, Edwards \cite{Edwards75} introduced what is now known as the Gro\-mov-Haus\-dorff distance for compact metric spaces, based upon consideration ab\-out quantum gravity, suggesting that one would have to work, not on one space-time, but rather on a ``super-space'' of possible space-times, metrized by a generalized Haus\-dorff distance, thus allowing fluctuations in the gravitational field. Edwards was following upon the suggestion to study quantum gravity by means of a superspace by Wheeler \cite{Wheeler68}. It should be noted that Gromov introduces his metric \cite{Gromov81} for more general locally compact metric spaces, and used his metric toward major advances in group theory. 

Thus, noncommutative metric geometry finds some of its roots in mathematical physics, and may well play a role in noncommutative theories of space-time. On the other hand, the quantum Gro\-mov-Haus\-dorff distance opens a new realm of mathematical possibilities: in general, the drought of sub-objects in the category of quantum spaces (i.e. of quotients in the category of C*-algebras, where some of the most studied examples are simple C*-algebras, such as the irrational rotation algebras) makes the study of quantum spaces difficult; yet the quantum Gro\-mov-Haus\-dorff topology allows for approximations of quantum spaces ``from the outside'' by other spaces, and opens new avenues for the study of C*-algebras. In the spirit of noncommutative geometry, this distance also opens the possibility to bring into the noncommutative realm the techniques of metric Geometry \cite{Gromov}.}

When introducing the quantum Gro\-mov-Haus\-dorff distance, the choice to forget the multiplicative structure (or at least, some of it, as the Jordan products can still be recovered \cite{Alfsen01}) by working with order-unit spaces rather than C*-algebras, allowed for very useful constructions on which a lot of convergence results mentioned above relied, including our own \cite{Latremoliere05}. Yet, the price for this choice is that Rieffel's distance may be null between two non-isomorphic C*-algebras, as long as they are both isomorphic as order-unit spaces and their quantum metric space structures are isometric in a well-defined sense \cite{Rieffel00}. This weakened form of the desirable coincidence property has sparked quite some efforts geared toward strengthening Rieffel's metric \cite{Kerr02,Kerr09,li03,Li06}. In all of these approaches, the principle is to replace the state space with an object more connected to the C*-algebraic structure, such as unital matrix-valued completely positive maps \cite{Kerr02,Kerr09}, or the restriction of the multiplication graph to the Lipschitz ball \cite{li03}. In general, new techniques need to be devised to show that sequences which converge in the quantum Gro\-mov-Haus\-dorff distance converge in these metrics, as they are stronger. However, it is not always clear what these techniques may be, as discussed for instance at the end of \cite[Section 2]{Rieffel09}. It was also not immediately apparent which construction may fit the future development of the theory best.

As a parallel development within the field, the importance of the C*-algebraic structure became evident as the field of noncommutative metric geometry grew. Rieffel's work on convergence for vector bundles \cite{Rieffel10}, and our own generalization of the notion of quantum metric spaces to the {\lcqms s} setup \cite{Latremoliere05b,Latremoliere12b}, are two such pieces of evidence. Concretely, the requirement that Lip-norms satisfy a strong form of the Leibniz identity has emerged and generated a large effort to recast previous convergence results for the quantum Gro\-mov-Haus\-dorff distance within this more stringent framework. Rieffel introduced the class of \emph{compact C*-metric spaces} $(\A,\Lip)$, as pairs of a unital C*-algebra $\A$ and a densely defined seminorm on $\A$ whose restriction to the self-adjoint part $\sa{\A}$ of $\A$ is a Lip-norm, and which satisfies, among various technical requirements, that for all $a,b\in \A$:
\begin{equation}\label{intro-leibniz-eq}
\Lip(ab) \leq \|a\|_\A\Lip(b) + \|b\|_\A\Lip(b)
\end{equation}
while, if $a\in\A$ is invertible:
\begin{equation}\label{intro-strong-leibniz-eq}
\Lip\left(a^{-1}\right) \leq \left\|a^{-1}\right\|_\A^2\Lip(a)\text{,}
\end{equation}
where $\|\cdot\|_\A$ is the norm of $\A$. The property described by Equation (\ref{intro-leibniz-eq}) is the \emph{Leibniz condition}, whereas together, Equations (\ref{intro-leibniz-eq}) and (\ref{intro-strong-leibniz-eq}) define what Rieffel calls the \emph{strong Leibniz condition} \cite{Rieffel10c}. These additional requirements define a natural class of C*-metric spaces, and by keeping the multiplicative structure encoded, they provide a natural framework for noncommutative metric geometry. Much effort has since been devoted to understanding Leibniz seminorms and their metric implications (for instance, \cite{Rieffel10,Rieffel10c,Rieffel11,Rieffel12}). We expect that some form of the Leibniz property will play a role in the theory of {\lcqms s}, which started our own investigations in these matters, resulting in this paper.

There are two fascinating difficulties with this new approach. First, most original proofs of convergence for the quantum Gro\-mov-Haus\-dorff distance involved, at some point, the use of a Lip-norm which failed to be a strong Leibniz seminorm, or even a Leibniz seminorm --- in fact, in the case of quantum tori, much work is done on order-unit spaces which are not the full self-adjoint part of a C*-algebra. Indeed, to establish the continuity of the family of quantum tori for the quantum Gro\-mov-Haus\-dorff distance, a particular construction, involving continuous fields of states, leads to a Lip-norm which fails to be Leibniz \cite[Proposition 11.1]{Rieffel00}. These fields are constructed to deal with the varying norms between different quantum tori. The same construction played a similar role in our own work on finite dimensional approximations of quantum tori \cite{Latremoliere05}. Thus, if one wishes to work within the context of C*-algebras and (strong) Leibniz seminorms, then some very non-trivial changes must be made. An example of the difficulties involved can be found in \cite{Rieffel10c}, which replaces all non-Leibniz Lip-norms in \cite{Rieffel01} involved in the proof of convergence of matrix algebras to co-adjoint orbits of Lie groups for the quantum Gro\-mov-Haus\-dorff distance, with strong Leibniz Lip-norms built using bimodules.


A second, deeper difficulty, regards the proper theoretical framework to work with C*-metric spaces, or more generally, with quantum compact metric spaces whose Lip-norm possess a form of the Leibniz property. A natural idea could be to require, in the definition of the quantum Gro\-mov-Haus\-dorff distance which we discussed above, that all admissible Lip-norms should be strong Leibniz themselves, and proceed with the construction without changes. Such an approach led Rieffel to define the \emph{proximity} between compact C*-metric spaces in \cite{Rieffel10c}. However, this quantity is not proven to be metric, and in fact, is unlikely to be so. The main difficulty is to prove the triangle inequality. Consequently, even if one proves convergence of a sequence of compact C*-metric spaces to some limit for Rieffel's proximity, it is unclear whether the sequence is Cauchy, and this leads to intriguing questions \cite{Rieffel09}. This is not the first time such problems occur: Kerr already faced this problem in \cite{Kerr02}. The main problem can be summarized in the fact that the quotient of a Leibniz seminorm is not always a Leibniz seminorm, as was already known in \cite{blackadar91}. This pathology in the behavior of Leibniz seminorms makes it extremely difficult to find out a proper metric between compact C*-metric spaces. Some past suggestions \cite{li03} could prove useful, but working with them has proven quite difficult as well (see \cite[Section 2]{Rieffel09}). The problem, it would seem, is that there is a gap between the sort of Leibniz Lip-norms one constructs in practice, and any available metric on quantum compact metric spaces.

We offer a solution to this second problem in the present paper. We propose a new distance, the quantum Gro\-mov-Haus\-dorff propinquity, which we intend to play the role for which Rieffel's proximity was introduced, with three very remarkable benefits. First, our quantum propinquity satisfies the triangle inequality by construction. Second, it is an actual distance on compact C*-metric spaces: distance zero implies *-isomorphism and, of course, isometry of the quantum metric space structures. Third, our quantum propinquity is constructed to take full advantage of the methods developed to work with strong Leibniz Lip-norms. To be a little more specific, given two compact quantum metric spaces $(\A,\Lip_\A)$ and $(\B,\Lip_\B)$, the construction of an admissible Lip-norm $\Lip$ for $(\Lip_\A,\Lip_\B)$, since the beginning of the research in this field \cite{Rieffel00}, typically involves finding a seminorm $\mathsf{N}$ on $\A\oplus\B$, called a bridge \cite[Section 5]{Rieffel00}, with sufficient properties so that:
\begin{equation*}
(a,b) \in \A\oplus\B \longmapsto \max\{ \Lip_\A(a),\Lip_\B(b), \mathsf{N}(a,b) \}
\end{equation*}
is an admissible Lip-norm. When $(\A,\Lip_\A)$ and $(\B,\Lip_\B)$ are compact C*-metric spaces, it is natural to make the same sort of construction, looking for bridges which satisfy the proper version of the strong Leibniz property. Rieffel's original proposition \cite{Rieffel10c} is to choose an $\A$-$\B$-bimodule $\Omega$, with the actions of $\A$ and $\B$ non-degenerate, an element $\omega \in \Omega$, and a bimodule norm $\|\cdot\|_\Omega$ on $\Omega$  (i.e. for all $a\in\A$,$b\in\B$ and $\eta\in\Omega$ we have $\|a\eta b\|_\Omega\leq \|a\|_\A\|\eta\|_\Omega\|b\|_\B$). Then one could try to build a bridge of the form:
\begin{equation*}
\mathsf{N}_{\Omega,\omega} : (a,b) \in \A\oplus\B \longmapsto\gamma \|a\omega - \omega b\|_\Omega
\end{equation*}
for some strictly positive scalar $\gamma$.

While very natural, this idea requires, in order to work, a non-trivial choice of bimodule, and to relate what being small in the sense of the seminorm $\mathsf{N}_{\Omega,\omega}$ implies on the distance between state spaces. Both of these issues can present major challenges.

One may prefer to replace general bimodules in the above construction with C*-algebras, both for technical practicality and for the elegance of working only within the category of C*-algebras. This is particularly sensible if one would wish to preserve the multiplicative structure of compact C*-metric spaces through convergence, in some sense. It is possible to do so in some known examples \cite{Rieffel09}, and again, this presents interesting problems.

In our present paper, we use the idea of the construction of bridges based on bimodules which are themselves C*-algebras as our starting point. Thus, our quantum propinquity is very much motivated, not only by the theoretical need to provide a consistent framework for metric noncommutative geometry of compact C*-metric spaces, but also by the practical matter of being applicable to known and future examples. 

The quantum propinquity shares with Li's nuclear distance \cite{li03} the idea to use the Haus\-dorff distance between images for *-mono\-morphisms of the unit balls of the Lip-norms of compact quantum metric spaces, which is akin to using bridges built from bimodules which are C*-algebras. However, it is unclear how our metric and Li's nuclear distance compare. A natural modification of Li's nuclear distance, the unital nuclear distance introduced in \cite{Kerr09}, does dominate the quantum propinquity (and Li's original nuclear distance). Compared with the unital nuclear distance, we introduce a pivot element and use it to define an inner-derivation which we include in the computation of the Haus\-dorff distance between Lip-balls. This pivot element allows us to replace approximations by order-unit spaces used for the quantum Gro\-mov-Haus\-dorff distance, by approximations involving products with the pivot element. Kerr and Li's unital nuclear distance would consist in always taking our pivot element to be the identity. 

It is less clear how to compare our quantum propinquity to Li's C*-algebraic distance. Our quantum propinquity is designed to work well for Leibniz Lip-norms, which do not play any special role for the C*-algebraic distance, and our distance works entirely within the category of C*-algebras. As our goal is to study how various structures over C*-algebras behave under metric convergence, it seems important for our goal to avoid involving more general Banach spaces in our theory, as was done for instance with the C*-algebraic distance. It is an interesting question to determine what relations, if any, our two metrics have.

In a different direction, Kerr's matricial distance \cite{Kerr02} involves complete positive, matrix-valued maps instead of states. Our construction in this paper could be generalized very easily to construct a matricial propinquity, by replacing the state space by completely positive matrix valued maps as well. The resulting matricial propinquity would dominate Kerr's distance for the same reason as our current metric dominates the quantum Gro\-mov-Haus\-dorff distance. Thus, while we do not require complete positivity in our work, and Kerr's original motivation seems to have been to address the distance zero problem with the quantum Gro\-mov-Haus\-dorff distance, which our quantum propinquity solves in its own way, it may well be that a matricial version of the quantum propinquity could find applications in later development of the field. We also point out that our approach for the distance zero property of the quantum propinquity is more in the spirit of Kerr's method than Rieffel's original work on the quantum Gro\-mov-Haus\-dorff distance.

Another approach to a noncommutative generalization of Lipschitz algebras is found in the work of Wei Wu \cite{Wu05,Wu06a,Wu06b}, where a quantized metric space is a matricial order unit space $\mathcal{V}$ endowed with a family $(\Lip_n)_{n\in\N}$ of Lip-norms, each defined on $\mathcal{V}\otimes M_n$ where $M_n$ is the algebra of $n\times n$ matrices, and satisfying natural compatibility conditions with the underlying operator space structure. Wu showed in \cite{Wu06a} how to construct a quantized Gromov-Hausdorff distance on the class of quantized metric spaces, and how, for instance, convergence of matrix algebras to the sphere \cite{Rieffel01} can be recast in this framework. Two quantized metric spaces are at distance zero from each other for the quantized Gromov-Hausdorff distance if and only if there exists a complete isometry between them. Once again, this suggests an interesting generalization of the quantum propinquity to a matricial, or quantized version, where one works with a sequence of Lip-norms on matricial algebras associated with {\Lqcms s}. The quantized Gromov-Hausdorff distance dominates the quantum Gromov-Hausdorff distance \cite[Proposition 4.9]{Wu06a}, but a direct comparison between our metric and the quantized Gromov-Hausdorff metric may prove challenging --- though a quantized version of our quantum propinquity would be expected to dominate Wu's distance for a similar reason as the quantum propinquity dominates Rieffel's distance.

Thus, in summary, this paper proposes a new distance, called the quantum Gro\-mov-Haus\-dorff propinquity, which is designed as a natural distance to work with compact C*-metric spaces, and dominates Rieffel's quantum Gro\-mov-\-Haus\-dorff distance. We hope that our work presents a solution to the long standing problem of working with Leibniz seminorms in the context of noncommutative metric geometry, while remaining practical. Our paper is organized in five sections in addition to this introduction. In a first section, we introduce the category of {\Lqcms s}, in which we shall work, and which is a subcategory of Rieffel's compact quantum metric spaces where all the underlying order-unit spaces are in fact the self-adjoint part of unital C*-algebras (rather than some subspace of it) and all the Lip-norms satisfy a Leibniz inequality and are lower semi-cont\-inuous. We then introduce the notions of bridges and treks upon which our distance is constructed. We define afterward our new distance. The proof of the coincidence property is the central piece of the next section. We conclude by showing that the quantum propinquity dominates the quantum Gro\-mov-Haus\-dorff distance, thus providing a natural interpretation to our distance, and by showing that the Gro\-mov-Haus\-dorff distance dominates our quantum propinquity restricted to classical compact metric spaces. The example of convergences of fuzzy and quantum tori is presented in this last section.


We wish to thank M. Rieffel for his comments on an earlier version of this paper.


\section{\Lqcms s}

At the root of our work is a pair consisting of a C*-algebra and a seminorm enjoying various properties. We start with the most minimal assumptions on such pairs and introduce the context of our paper.

\begin{notation}
We denote the self-adjoint part of any C*-algebra $\A$ by $\sa{\A}$, and the unit of a unital C*-algebra $\A$ by $\unit_\A$. The norm of any space $V$ is denoted by $\|\cdot\|_V$ unless otherwise specified. 
\end{notation}

\begin{definition}\label{Lipschitz-pair-def}
A \emph{unital Lipschitz pair} $(\A,\Lip)$ is a unital C*-algebra $\A$ and a seminorm $\Lip$ defined on a dense subspace $\dom{\Lip}$ of $\sa{\A}$ such that:
\begin{equation*}
\left\{ a \in \dom{\Lip} : \Lip(a) = 0 \right\} = \R \unit_{\A}\text{.}
\end{equation*}
\end{definition}

It will often be convenient to adopt the following convention regarding Lipschitz pairs:

\begin{convention}\label{dom-convention}
We denote the domain of a densely defined seminorm $\Lip$ on $\sa{\A}$ by $\dom{\Lip}$. We will usually regard $\Lip$ as a \emph{generalized} seminorm on the whole of $\sa{\A}$ by setting $\Lip(a) = \infty$ for $a\not\in\dom{\Lip}$. We shall adopt this convention whenever convenient without further mention. Moreover, we adopt the measure-theoretic convention regarding computations with $\infty$, namely: $r\infty = \infty r = \infty$ for all $r\in(0,\infty]$ while $0\infty = \infty 0 = 0$, and $r + \infty = \infty + r = \infty$ for all $r \in \R$. With this convention, we have:
\begin{equation*}
\dom{\Lip} = \left\{ a \in \sa{\A} : \Lip(a) < \infty \right\}\text{.}
\end{equation*}
\end{convention}

At the core of noncommutative metric geometry are the quantum metric spaces, defined in \cite{Rieffel98a,Rieffel99} by Rieffel in the compact case, and generalized by the present author to the locally compact setup in \cite{Latremoliere12b}. In either situation, the metric data encoded by a Lipschitz pair is reflected by the following distance on the state space of the underlying C*-algebra:

\begin{notation}
The \emph{state space} of a C*-algebra $\A$ is denoted by $\StateSpace(\A)$.
\end{notation}

\begin{definition}\label{mongekant-def}
The \emph{\mongekant} associated with a unital Lipschitz pair $(\A,\Lip)$ is the distance $\Kantorovich{\Lip}$ defined on the state space $\StateSpace(\A)$ of $\A$ as:
\begin{equation*}
\Kantorovich{\Lip} : \varphi,\psi \in \StateSpace(\A) \longmapsto \sup \{ |\varphi(a) - \psi(a) | : a\in\sa{\A}\text{ and }\Lip(a) \leq 1 \} \text{.}
\end{equation*} 
\end{definition}

Definition (\ref{Lipschitz-pair-def}) of a unital Lipschitz pair is designed precisely so that the associated {\mongekant} is indeed an extended metric, as one can easily check: it satisfies all axioms of a metric, except that it may take the value $\infty$. This extended metric has a long history and many names, and could fairly be called the Monge - Kantorovich - Rubinstein - Wasserstein - Dobrushin - Connes- Rieffel metric. It finds its roots in the transportation problem introduced in 1781 by Monge, which was later studied by Kantorovich, who introduced this metric \cite{Kantorovich40} in 1940. The form we use was first given for classical metric spaces in 1958 by Kantorovich and Rubinstein \cite{Kantorovich58}. In the context of probability theory, Wasserstein later re-introduced this same metric in 1969 \cite{Wasserstein69}. Dobrushin \cite{Dobrushin70} named this metric after Wasserstein, and also studied the properties of this metric in the noncompact setting, which was fundamental to our work in \cite{Latremoliere12b} on quantum locally compact metric spaces. Connes suggested in \cite{Connes89,Connes} to use Definition (\ref{mongekant-def}) to study metric properties in noncommutative geometry, when $\Lip$ is the seminorm associated with a spectral triple. In \cite{Rieffel98a}, Rieffel expanded on Connes' idea by introducing the more general framework within which this paper is written. Our choice of terminology is somewhat arbitrary, based loosely on the order of appearance of this metric, and in line with Rieffel's terminology.

The fundamental idea of Rieffel \cite{Rieffel98a,Rieffel99} is to match the topology induced by the {\mongekant} to the natural topology on the state space:

\begin{definition}
A \emph{\qcms $(\A,\Lip)$} is a unital Lipschitz pair such that the associated {\mongekant} $\Kantorovich{\Lip}$ induces the weak* topology on the state space $\StateSpace(\A)$ of $\A$. If $(\A,\Lip)$ is a quantum compact metric space, then $\Lip$ is called a \emph{Lip-norm} on $\sa{\A}$.
\end{definition}

\begin{remark}\label{a-remark}
Rieffel's \emph{compact quantum metric spaces} are pairs $(\A,\Lip)$ where $\A$ is an order unit space, and $\Lip$ is a densely defined seminorm on $\A$ which vanishes exactly on the scalar multiple of the order unit of $\A$, such that the {\mongekant} defined by Definition (\ref{mongekant-def}) metrizes the weak* topology of $\StateSpace(\A)$ \cite{Rieffel98a,Rieffel99}, where $\StateSpace(\A)$ is meant for the state space of $\A$ as an order unit space \cite{Alfsen01}. We shall only work in this paper in the context where the order-unit spaces are, in fact, the set of all self-adjoints elements of a unital C*-algebra --- in fact we are particularly interested in the multiplicative structure and its relation to noncommutative metric geometry. As our definition is a special case of Rieffel's, we employ a slightly different terminology, where we inverse the order of the words \emph{quantum} and \emph{compact}, in line with our own terminology in \cite{Latremoliere12b}. \textbf{Thus we emphasize that in this paper, {\qcms s} are always a pair of a $C^\ast$-algebra and a Lip-norm on it}, not more general C*-normed algebras or order unit spaces.
\end{remark}

\begin{remark}
If $(\A,\Lip)$ is a {\qcms}, then the weak* topology of $\StateSpace(\A)$ is metrizable, and since $\StateSpace(\A)$ is weak* compact, $\StateSpace(\A)$ is separable in the weak* topology, which in turn implies that the topological dual of $\A$ is separable by the Hahn-Jordan decomposition. Thus, one easily checks that $\A$ is separable in norm.
\end{remark}

The inspiration for Rieffel's definition is found within the classical picture given by classical compact metric spaces:

\begin{example}\label{classical-ex}
Let $(X,\mathsf{d})$ be a compact metric space. For any function $f : X\rightarrow \R$, we define the \emph{Lipschitz constant of $f$} as:
\begin{equation*}
\Lip(f) = \sup\left\{ \frac{|f(x)-f(y)|}{\mathsf{d}(x,y)} : x,y \in X, x\not= y \right\}\text{,}
\end{equation*}
which in general may be infinite. Of course, if $f:X\rightarrow \R$ is a Lipschitz function, i.e. if $f$ has finite Lipschitz constant, then $f \in \sa{C(X)}$ where $C(X)$ is the C*-algebra of continuous functions from $X$ to $\C$. It is immediate that a function has zero Lipschitz constant if and only if it is constant on $X$. Moreover, the set of Lipschitz functions is norm-dense in $\sa{C(X)}$ by the Stone-{Weierstra\ss} Theorem, and thus $(C(X),\Lip)$ is a unital Lipschitz pair.

As a prelude of the rest of this section, we also note that Arz{\'e}la-Ascoli's Theorem shows that the set:
\begin{equation*}
\{ f \in \sa{C(X)} : f(x_0)=0 \text{ and }\Lip(f) \leq 1 \}
\end{equation*}
is norm compact, for any $x_0\in X$. In turn this property implies that $\Kantorovich{\Lip}$ metrizes the weak* topology on the space $\StateSpace(C(X))$ of Radon probability measures on $X$ (see for instance \cite{Dudley}). Thus $(C(X),\Lip)$ is a quantum compact metric space.
\end{example}

A form of Arz{\'e}la-Ascoli theorem serves as the main characterization of {\qcms s}. It first appeared in \cite{Rieffel98a}, at the very root of this theory of quantum metric spaces. An equivalent formulation of this result appeared in \cite{Ozawa05}, which we then generalized to the locally compact setting in \cite[Theorem 3.10]{Latremoliere12b}. We summarize these characterizations of {\qcms s} as follows:

\begin{theorem}\cite{Rieffel98a, Ozawa05}\label{az-thm}
Let $(\A,\Lip)$ be a unital Lipschitz pair. The following are equivalent:
\begin{enumerate}
\item $(\A,\Lip)$ is a {\qcms}.
\item There exists $\varphi \in\StateSpace(\A)$ such that the set:
\begin{equation*}
\mulip(\A,\Lip,\varphi) = \{ a \in \sa{\A} : \varphi(a) = 0 \text{ and } \Lip(a)\leq 1 \}
\end{equation*}
is totally bounded for the norm $\|\cdot\|_\A$ of $\A$.
\item For all $\varphi \in\StateSpace(\A)$, the set:
\begin{equation*}
\mulip(\A,\Lip,\varphi) = \{ a \in \sa{\A} : \varphi(a) = 0 \text{ and } \Lip(a)\leq 1 \}
\end{equation*}
is totally bounded for the norm $\|\cdot\|_\A$ of $\A$.
\item There exists $r \geq 0$ such that the set:
\begin{equation*}
\left\{ a \in \sa{\A} : \Lip(a)\leq 1 \text{ and }\|a\|_\A \leq r \right\}
\end{equation*}
is totally bounded for the norm $\|\cdot\|_\A$ of $\A$, and the diameter of $\left(\StateSpace(\A),\Kantorovich{\Lip}\right)$ is less or equal to $r$.
\end{enumerate}
\end{theorem}

A class of examples which will prove important to us in this paper is given by quantum homogeneous spaces, as was established by Rieffel at the very start of the study of compact quantum metric spaces:

\begin{example}\label{q-tori-ex}\cite{Rieffel98a}
Let $G$ be a compact group and $l$ be a continuous length function on $G$. Let $e$ be the identity element of $G$. Let $\A$ be a unital C*-algebra whose $\ast$-automorphism group is denoted by $\operatorname{Aut}(\A)$, and let $\alpha: g~\in~G~\mapsto~\alpha^g \in \operatorname{Aut}(\A)$ be a strongly continuous action of $G$ on $\A$, which is \emph{ergodic} in the sense that:
\begin{equation*}
\bigcap_{g\in G}\{ a \in \A : \alpha^g(a) = a\} = \C \unit_\A\text{.}\end{equation*}
For all $a\in\A$, we define:
\begin{equation*}
\Lip(a) = \sup \left\{ \frac{\|\alpha^g(a) - a\|_\A} {l(g)} : g \in G\setminus\{e\} \right\} \text{,}
\end{equation*}
where $\Lip(a)$ may be infinite. Then \cite[Theorem 2.3]{Rieffel98a} proves that $(\A,\Lip)$ is a indeed a {\qcms}. This result employs the deep fact that ergodic actions of compact groups on C*-algebras have finite dimensional spectral subspaces \cite{Hoegh-Krohn81}, and is therefore very nontrivial.
\end{example}

As seen with Example (\ref{q-tori-ex}), proving that a given unital Lipschitz pair is a quantum compact metric space is in general difficult. Another source of potential examples of {\qcms s} comes from the original suggestion of Connes to derive metric information from spectral triples \cite{Connes80} using Definition (\ref{mongekant-def}). For our purpose, a \emph{spectral triple} $(\A,\Hilbert_\pi,D)$ is given by \cite{Connes}:
\begin{enumerate}
\item a unital C*-algebra $\A$,
\item a non-degenerate faithful *-representation $\pi$ of $\A$ on a Hilbert space $\Hilbert$,
\item an operator $D$ on $\Hilbert$, not necessarily bounded, such that:
\begin{equation*}
\A^1 = \{ a \in \sa{\A} : [D,\pi(a)] \text{ is a closeable operator with bounded closure} \}
\end{equation*}
is dense in $\sa{\A}$.
\end{enumerate}
Given a spectral triple $(\A,\Hilbert,D)$ over a unital C*-algebra $\A$, we may define:
\begin{equation}\label{Dirac-Lip-eq}
\forall a \in \A^1\quad \Lip_D(a)=\|[D,\pi(a)]\|_{\B(\Hilbert)}\text{,}
\end{equation}
where $\|\cdot\|_{\B(\Hilbert)}$ is the operator norm for bounded operators on $\Hilbert$. Note that we identify $[D,\pi(a)]$ with its closure when $a\in\A^1$. By construction, $(\A,\Lip_D)$ is a unital Lipschitz pair.

It is not yet known when $(\A,\Lip_D)$ is a quantum compact metric space in general, even with the additional assumptions that $D$ is self-adjoint and $D^2$ has a compact resolvant whose sequence of singular values is $p$-summable for some $p > 0$ (see \cite{Pavlovic98} for some work on this fascinating question, though his conclusions are really a particular case of the main result of \cite{Rieffel98a}). It is however known \cite[Proposition 1]{Connes89} that the {\mongekant} associated with $\Lip_D$ recovers the Riemannian metric on a spin manifold when $D$ is chosen to be the actual Dirac operator associated to the given Riemannian metric, acting on the square-integrable sections of a spin bundle over the manifold. In \cite[Theorem 4.2]{Rieffel98a}, Rieffel proved that the standard (iso-spectral) spectral triple on quantum tori also provides a quantum compact metric space structure by comparing the Lip-norms for these spectral triples with the Lip-norms from Example (\ref{q-tori-ex}). We include two more examples of spectral triples for which the associated Lipschitz pair is indeed a quantum compact metric space.

\begin{example}\cite{Ozawa05}\label{dirac1-ex}
Let $G$ be a finitely generated discrete group and $l$ a length function on $G$ such that $(G,l)$ is a Gro\-mov hyperbolic group. Let $\pi$ be the left regular representation of $C_{\mathrm{red}}^\ast(G)$ on $\ell^2(G)$ and let $D$ be the (unbounded) operator of multiplication by $l$ on $\ell^2(G)$. With the notation of Equation (\ref{Dirac-Lip-eq}), Ozawa and Rieffel proved in \cite{Ozawa05} that the unital Lipschitz pair $(C_{\mathrm{red}}^\ast(G),\Lip_D)$ is a quantum compact metric space.
\end{example}

\begin{example}\cite{Rieffel02}\label{dirac2-ex}
Let $d\in\N\setminus\{0,1\}$ and $l$ be the word length for some set of generators of $\Z^d$. Let $\sigma$ be an arbitrary $2$-cocycle of $\Z^d$. Again, let $\pi$ be the left regular $\sigma$-representation of $\Z^d$ on $\ell^2(\Z^d)$, and $D$ be the (unbounded) multiplication operator by $l$ on $\ell^2(\Z^d)$. Then Rieffel proved in \cite{Rieffel02} that $(C^\ast(\Z^d,\sigma),\Lip_D)$ is a quantum compact metric space, with $\Lip_D$ defined by Equation (\ref{Dirac-Lip-eq}). This proof involves Connes' cosphere algebras.
\end{example}

The situation for more general groups, let alone more general spectral triples, is open, with only a few other known specific examples such as noncommutative spheres \cite{Li05}. In \cite{Latremoliere12b}, we extend the notion of quantum metric space to the locally compact setting, and we then show how the Moyal plane can be seen as a quantum locally compact metric space by using the spectral triple introduced in \cite{Varilly04}.

\bigskip

All the examples we have provided share a property which was not a part of Rieffel's original definition of a compact quantum metric space: all the Lip-norms satisfy the Leibniz inequality (see Equation (\ref{intro-leibniz-eq})). In fact, most examples are given by a form of spectral triple. For instance, as seen in \cite{Rieffel99}, let $(X,\mathsf{d})$ be a compact metric space and let $\Delta=\{(x,x):x\in X\}$. Let $\mu$ be a strictly positive Radon probability measure on the locally compact Haus\-dorff space $Z=X\times X \setminus \Delta$. For any $f \in C(X)$, we define, as is standard, the operators $L_f$ and $R_f$ on $L^2(Z,\mu)$ by $R_f(g)(x,y)=f(x)g(x,y)$ and $L_f(g)(x,y)=g(x,y)f(y)$ for all $g\in L^2(Z,\mu)$ and $(x,y)\in Z$. For all Lipschitz functions $g\in C(X)$ and $(x,y)\in Z$, we set:
\begin{equation*}
\delta(g) (x,y) = \frac{g(x,y)}{\mathsf{d}(x,y)}
\end{equation*}
and note that $\delta$ defines an unbounded operator on $L^2(Z,\mu)$.  Let:
\begin{equation*}
\Hilbert = L^2(Z,\mu)\oplus L^2(Z,\mu)\text{, }\quad D = \begin{pmatrix}
0 & \delta\\
\delta & 0
\end{pmatrix}
\text{ and }\quad\pi(f) = \begin{pmatrix}
L_f & 0 \\
0 & R_f 
\end{pmatrix}
\end{equation*}
for all $f\in C(X)$. Then $(C(X),\Hilbert,D)$ is a spectral triple in our weak sense, and $\Lip_D$ is simply the Lipschitz seminorm associated with $\mathsf{d}$.

Thus, it would be natural to ask that our notion of convergence for quantum compact metric spaces be well-behaved, in some sense, with respect to the Leibniz property, and maybe even with the notion of bimodule-valued derivations. Our quantum propinquity may be a step in this direction, as its construction relies on inner-derivations of special bimodules between {\qcms s}, as we shall see.

\bigskip

The Leibniz property, while unnecessary in the construction of the quantum Gro\-mov-Haus\-dorff distance \cite{Rieffel00}, is desirable for the study of the continuity of various additional properties of quantum compact metric spaces with respect to metric convergence. For instance, Rieffel \cite{Rieffel10} uses a strong version of the Leibniz property to study the convergence of vector bundles associated with a Gro\-mov-Haus\-dorff convergent sequence of compact metric spaces. The purpose of developing the notion of quantum metric spaces is to adapt and extend results in metric geometry to noncommutative C*-algebras with the hope that many interesting geometric properties will be well-behaved with respect to various hypertopologies on quantum metric spaces. In the spirit of \cite{Rieffel10}, it would seem natural to work within the class of quantum metric spaces with Leibniz Lip-norms to carry information about the multiplicative structure of C*-algebras. It may even prove useful to restrict attention to Lip-norms coming from a form of differential calculus. 

To this end, Rieffel introduced the notion of a \emph{compact C*-metric space}, which imposes a strong form of the Leibniz identity to the Lip-norms involved, and attempted to recast the convergence of matrix algebras to spheres entirely within the framework of compact C*-metric spaces. This requires to avoid any recourse to order-unit spaces with no multiplicative structure, and forces one to work only with C*-algebras, which is a first complication. Moreover, one must work only with Leibniz Lip-norms, which again add some high degree of technical difficulties. Last, and maybe most challenging, one would desire to have a  form of metric convergence which strengthens the quantum Gro\-mov-Haus\-dorff convergence by working within the class of compact C*-metric spaces. Motivated by this observation, Rieffel introduced the \emph{quantum proximity}, a semi-pseudo-metric on compact C*-metric spaces. This object suffers from two fundamental problems. First, it does not satisfy the triangle inequality. This issue may be addressed by a standard technique, as we shall see in this paper. A more serious issue is that the quantum proximity may be null between two non-isomorphic C*-algebras.

The purpose of this paper is to introduce a notion of distance on a class of quantum compact metric spaces which provides a solution to the problems raised by Rieffel's quantum proximity. We propose that our new distance could play the role intended by Rieffel for the quantum proximity. The class of {\Lqcms s} which will be the space for our metric includes, but expand slightly on, the class of compact C*-metric spaces.  Our metric satisfies that, among other properties, two Leibniz quantum compact metric spaces are at distance zero only if their underlying C*-algebras are *-isomorphic. To this end, we impose a natural conditions on Lip-norms to control their behavior with respect to multiplication. To keep our framework general, and because of inherent difficulties with Lipschitz seminorms of complex functions (see, in particular, \cite[Example 5.4]{Rieffel10c}), we actually work within the Jordan-Lie-algebra of self-adjoint elements:

\begin{notation}
Let $\A$ be a C*-algebra and $a,b \in \A$. The \emph{Jordan product} $\frac{1}{2}(ab+ba)$ is denoted by $\Jordan{a}{b}$. We denote the \emph{Lie product} $\frac{1}{2i}(ab-ba) = \frac{1}{2i}[a,b]$ of $a$ and $b$ by $\Lie{a}{b}$. Thus $(\sa{\A},\Jordan{\cdot}{\cdot},\Lie{\cdot}{\cdot})$ is a Jordan-Lie algebra \cite[Definition 1.1.2]{landsman98}; in particular it is closed under the Jordan product and the Lie product. (Our Lie product differs by a factor of $i$ from the Lie product from \cite{landsman98}.)
\end{notation}

We shall henceforth employ Convention (\ref{dom-convention}) regarding Lipschitz pairs.

\begin{definition}\label{JLL-pair-def}
A unital Lipschitz pair $(\A,\Lip)$ has the \emph{{\JLL} property}, or equivalently is called a \emph{unital {\JLL} pair}, when for $a,b \in \sa{\A}$, we have:
\begin{equation*}
\Lip\left(\Jordan{a}{b}\right) = \Lip\left(\frac{ab+ba}{2}\right) \leq \|a\|_\A\Lip(b) + \|b\|_\A\Lip(a)
\end{equation*}
and
\begin{equation*}
\Lip\left(\Lie{a}{b} \right) = \Lip\left(\frac{ab-ba}{2i}\right) \leq \|a\|_\A\Lip(b) + \|b\|_\A\Lip(a)\text{.}
\end{equation*}
\end{definition}

\begin{remark}
The domain of a Leibniz seminorm for a unital Lipschitz pair is always a Jordan-Lie algebra.
\end{remark}

Thus, the seminorm of a unital {\JLL} pair satisfies a form of the Leibniz inequality, albeit restricted to the Jordan product and the commutator. A natural and common source of examples of unital {\JLL} pair is:

\begin{proposition}\label{Leibniz-implies-JLL-prop}
Let $\A$ be a unital C*-algebra and $\Lip$ be a seminorm defined on a dense $\C$-subspace $\operatorname{dom}_\C(\Lip)$ of $\A$, such that $\operatorname{dom}_\C(\Lip)$ is closed under the adjoint operation, such that $\{a \in \operatorname{dom}_\C(\Lip) : \Lip(a) = 0 \} = \C \unit_\A$ and, for all $a,b \in \operatorname{dom}_\C(\Lip)$, we have:
\begin{equation*}
\Lip(ab) \leq \|a\|_\A\Lip(b) + \|b\|_\A\Lip(a)\text{.}
\end{equation*}
Then $(\A,\Lip)$ is a unital {\JLL} pair, identifying $\Lip$ with its restriction to $\sa{\A}\cap\operatorname{dom}_\C(\Lip)$.
\end{proposition}

\begin{proof}
Following Convention (\ref{dom-convention}), we extend $\Lip$ to $\A$ by setting $\Lip(a) = \infty$ for all $a\not\in\operatorname{dom}_\C(\Lip)$.

First, we note that if $a\in\sa{\A}$, then by assumption, $a$ is the limit of a sequence $(a_n)_{n\in\N}$ in the domain of $\Lip$. Yet, since $\Lip$ is a seminorm and the domain of $\Lip$ is assumed self-adjoint, we conclude that $\left(\frac{a_n+a_n^\ast}{2}\right)_{n\in\N}$ lies within the domain of $\Lip$, while still being a sequence in $\sa{\A}$ which converges to $a$. Thus, the restriction of $\Lip$ to $\sa{\A}$ has dense domain.

Now, for all $a,b \in \sa{\A}$ we have:
\begin{equation*}
\begin{split}
\Lip(\Jordan{a}{b})&=\Lip\left(\frac{1}{2}(ab+ba)\right)
\\ &\leq \frac{1}{2}\left(2\|a\|_\A\Lip(b) + 2\|b\|_\A\Lip(a)\right) = \|a\|_\A\Lip(b) + \|b\|_\A\Lip(a) \text{,}
\end{split}
\end{equation*}
and similarly for the Lie product.
\end{proof}

\begin{example}
The unital Lipschitz pairs defined in Examples (\ref{classical-ex}) and (\ref{q-tori-ex}), or constructed from spectral triples or derivations \cite{Connes} are unital {\JLL} pairs.
\end{example}

We now can present the objects of interest for this paper:

\begin{definition}\label{lqcms-def}
A \emph{\Lqcms} $(\A,\Lip)$ is a quantum compact metric space $(\A,\Lip)$ such that $(\A,\Lip)$ is a unital {\JLL} pair and $\Lip$ is lower semi-cont\-inuous with respect to the norm $\|\cdot\|_\A$ of $\A$. If $(\A,\Lip)$ is a {\Lqcms}, then we call $\Lip$ a {\JLL} Lip-norm.
\end{definition}

\begin{notation}\label{lcqmss-not}
The class of all {\Lqcms s} is denoted by the symbol $\LCQMS$.
\end{notation}

\begin{example}
Examples (\ref{classical-ex}), (\ref{q-tori-ex}), as well as the Lip-norms from the spectral triples in Examples (\ref{dirac1-ex}) and (\ref{dirac2-ex}) are all {\Lqcms}.
\end{example}

\begin{remark}\label{closed-rmk}
Let $(\A,\Lip)$ be a {\Lqcms}. As $\Lip$ is lower semi-cont\-inuous, for any $r > 0$, the set:
\begin{equation*}
\Lipball{r}{\A,\Lip} = \{a\in\sa{\A} : \Lip(a)\leq 1\text{ and }\|a\|_\A \leq r \}
\end{equation*}
is closed, and as $\Lip$ is a Lip-norm, by Theorem (\ref{az-thm}), $\Lipball{r}{\A,\Lip}$ is totally bounded, and thus compact. Thus $\Lip$ is a \emph{closed} Lip-norm in the sense of \cite[Definition 4.5]{Rieffel99}. Since $\A$ is a C*-algebra, and thus $\sa{\A}$ is always complete, Lip-norms are lower semi-cont\-inuous if and only if they are closed. The main advantage of closed Lip-norms is that they are uniquely determined by their associated {\mongekant} \cite[Theorem 4.1]{Rieffel99} and this advantage makes the notion of quantum isometry unambiguous. In our context, starting from a {\JLL} Lip-norm which is not lower semi-cont\-inuous, it is unclear whether the lower-semi-cont\-inuous Lip-norm one recovers by duality from the {\mongekant} would be {\JLL}. Thus our assumption is natural.
\end{remark}

\begin{remark}
  Rieffel introduces the class of \emph{compact C*-metric spaces} in \cite{Rieffel10c}, which, using our terminology, are {\Lqcms s} with strong-Leibniz Lip-norms (see Inequalities (\ref{intro-leibniz-eq}) and (\ref{intro-strong-leibniz-eq})), rather than only Leibniz, and with additional requirements on the domains of the Lip-norms. Our work applies to this sub-class of {\Lqcms s}, and in fact we shall develop our theory so that its users may adapt our constructions to various sub-classes of {\Lqcms s} (and even of {\qcms s} with well-behaved Lip-norms). It is a very interesting question to determine which such sub-classes are closed for the quantum propinquity. On the other hand, the quantum propinquity is well-behaved within the context of {\Lqcms s}, and thus we choose this natural class for our exposition.
\end{remark}

We now return to the description of our category of {\Lqcms s}. We define our notions of morphisms for the objects we have introduced in Definition (\ref{lqcms-def}). In fact, the category consists of the natural restriction of the category of quantum compact metric spaces to the class of {\Lqcms s}. Throughout this paper, we shall use the following notation:

\begin{notation}
If $\A$, $\B$ are two unital C*-algebras, and $\pi : \A\rightarrow \B$ is a unital *-homomorphism of C*-algebra, we denote by $\StateSpace(\pi)$ the continuous affine map:
\begin{equation*}
\StateSpace(\pi) : \varphi \in \StateSpace(\B) \longmapsto \varphi \circ \pi \text{,}
\end{equation*}
and we refer to $\StateSpace(\pi)$ as the dual map of $\pi$, as a slight abuse of language.
\end{notation}

In particular, if $\pi : \A\hookrightarrow\B$ is a unital *-monom\-orphism, then $\StateSpace(\pi): \StateSpace(\B)\twoheadrightarrow\StateSpace(\A)$ is a surjective weak* continuous affine map.

\begin{definition}
Let $(\A,\Lip_\A)$ and $(\B,\Lip_\B)$ be two {\qcms s}. A \emph{quantum Lipschitz homomorphism} $\Phi : \A\rightarrow\B$ is a unital *-homomorphism from $\A$ to $\B$ such that its dual map $\StateSpace(\Phi) : \StateSpace(\B)\rightarrow\StateSpace(\A)$ is a Lipschitz map from $(\StateSpace(\B),\Kantorovich{\Lip_\B})$ to $(\StateSpace(\A),\Kantorovich{\Lip_\A})$, where $\Kantorovich{\Lip_\A}$ (respectively $\Kantorovich{\Lip_\B}$) is the {\mongekant} associated with $(\A,\Lip_\A)$ (respectively $(\B,\Lip_\B)$).
\end{definition}

It is an easy exercise to check that the composition of two quantum Lipschitz homomorphisms is a quantum Lipschitz homomorphism, where composition is the usual composition of *-homomorphisms. Thus, a quantum Lipschitz isomorphism is a *-isomorphism whose dual map is a bi-Lipschitz function between the state spaces metrized by their respective {\mongekant s}. A special case of such quantum Lipschitz isomorphisms is given by:

\begin{definition}\label{quantum-isometry-def}
Let $(\A,\Lip_\A)$ and $(\B,\Lip_\B)$ be two {\qcms s}. A \emph{isometric isomorphism} $\Phi : \A\rightarrow\B$ is a *-isomorphism from $\A$ onto $\B$ such that the dual map $\StateSpace(\Phi) : \StateSpace(\B)\rightarrow\StateSpace(\A)$ is an isometry from $(\StateSpace(\B),\Kantorovich{\Lip_\B})$ onto $(\StateSpace(\A),\Kantorovich{\Lip_\A})$, where $\Kantorovich{\Lip_\A}$ (respectively $\Kantorovich{\Lip_\B}$) is the {\mongekant} associated with $(\A,\Lip_\A)$ (respectively $(\B,\Lip_\B)$).
\end{definition}

It is useful to recall that, by \cite[Theorem 6.2]{Rieffel00}, we have the following characterization of isometric isomorphisms:
\begin{theorem}\label{quantum-isometry-prop}
Let $(\A,\Lip_\A)$ and $(\B,\Lip_\B)$ be two {\Lqcms s}. A *-isomorphism $\Phi : \A\rightarrow\B$ is a isometric isomorphism if and only if $\Lip_\B\circ \Phi = \Lip_\A$.
\end{theorem}

\begin{proof}
Apply \cite[Theorem 6.2]{Rieffel00}, since $\Lip_\A$ and $\Lip_\B$ are closed by assumption (see Remark (\ref{closed-rmk})).
\end{proof}

The definition of the quantum Gro\-mov-Haus\-dorff distance relies on the notion of a quantum metric subspace, which is a quotient object in the category of order-unit spaces, endowed with the quotient seminorm from a Lip-norm --- which is itself a Lip-norm \cite{Rieffel00}. Unfortunately, the quotient of a Leibniz seminorm may not be Leibniz. This difficulty is discussed in \cite{Rieffel10c}, and was encountered before in noncommutative geometry (e.g. \cite{Kerr02},\cite{blackadar91}). Rieffel proposed in \cite{Rieffel10c} a notion of compatibility between a seminorm and its quotient which would allow the Leibniz property to be inherited by the quotient seminorm. However, even this additional condition does not appear natural in the category we have just defined. In general, sub-objects in noncommutative geometry are a challenge to define, and it would appear that Leibniz seminorms, which are tied to the C*-algebraic structures (unlike general Lip-norms) exhibit the same reluctance to localize to some ``subspace'' in general. Thus, we take a completely different approach, which simply does not need the notion of sub-object at all, and thus fits well within our noncommutative framework.

\section{Bridges and Treks}

The approach we propose to define our new metric is based on our notion of \emph{bridge}, and associated numerical quantities. A bridge between two {\Lqcms s} $(\A,\Lip_\A)$ and $(\B,\Lip_\B)$ allows us to define Leibniz Lip-norms on $\A\oplus\B$, in the manner which then allows, by the work of Rieffel in \cite{Rieffel00}, to compute upper bounds for the quantum Gro\-mov-Haus\-dorff distance between $(\A,\Lip_\A)$ and $(\B,\Lip_\B)$. However, our bridges will allow to define Lip-norms on $\A\oplus\B$ with the Leibniz property of Definition (\ref{JLL-pair-def}), or even the strong Leibniz property (Inequalities (\ref{intro-leibniz-eq},\ref{intro-strong-leibniz-eq})), as long as $\Lip_\A$ and $\Lip_\B$ have these same properties. 

Moreover, our distance is computed from the bridges themselves. Our idea thus allows us to avoid using sub-objects in the dual category of {\Lqcms s}, as those are not in general well-behaved. Instead, we employ paths consisting of bridges, which we will refer to as \emph{treks}, to construct our distance. Using treks essentially amount to enforcing the triangle inequality, yet it places all the weight of our efforts on proving that our distance is zero exactly when two {\Lqcms s} are isometrically isomorphic. The definition of the quantum Gro\-mov-Haus\-dorff propinquity and its main properties will be the matter of the next two sections, once we define the core notions of bridges, treks and their lengths in this section.

\bigskip

We first introduce the following notion, whose role will be central to our work.

\begin{definition}\label{one-level-def}
The \emph{$1$-level} of an element $\omega \in \D$ in a unital C*-algebra $\D$ is the set:
\begin{equation*}
\mathscr{S}_1(\omega) = \left\{ \varphi \in \StateSpace(\D) \middle| 
\begin{aligned}
\varphi\left( (\unit_\D-\omega)^\ast(\unit_\D-\omega)\right) = 0\text{,}\\
\varphi\left( (\unit_\D-\omega)(\unit_\D-\omega)^\ast\right) = 0
\end{aligned}
\right\} \text{.}
\end{equation*}
\end{definition}

\begin{remark}
A state $\varphi$ of a unital C*-algebra $\A$ which is a member of the $1$-level set of a \emph{self-adjoint} element $\omega \in \sa{\A}$ is \emph{definite} on $\omega$, in the sense, for instance, of \cite[Exercise 4.6.16]{Kadison91,Kadison97}. Our notion of $1$-level thus extends the notion of a definite state for a non-self-adjoint element.
\end{remark}

\begin{remark}\label{selfadjointpivot-rmk}
The theory developed in this section does not require either assumptions of self-adjointness (or even normality) or any bound on the norm. We believe that it is likely that future applications of the quantum propinquity will benefit from either or both assumptions, though at the moment we find no reason to restrict the level of generality of our construction. Thus, we put the minimal assumption on our pivot elements, and we simply remark that the proofs of this section would carry identically if we added the restriction of self-adjointness of the pivot, or imposed some bound on the norm of the pivot.
\end{remark}

A priori, the $1$-level of an element may be empty, but we shall only be interested in elements with non-empty $1$-level in our work. The key purpose of this set for us is the following lemma:

\begin{lemma}
Let $\D$ be a unital C*-algebra. A state $\varphi\in\StateSpace(\D)$ belongs to the $1$-level of a given $\omega \in \D$, if and only if for all $d\in \D$, we have $\varphi(d) = \varphi(\omega d) = \varphi(d \omega)$.
\end{lemma}

\begin{proof}
Assume that $\varphi\in\mathscr{S}_1(\omega)$. Let $d\in\D$. Then, by Cauchy-Schwarz:
\begin{equation*}
\begin{split}
|\varphi(d)-\varphi(d\omega)| &= |\varphi(d(\unit_\D-\omega))|\\
&\leq \sqrt{\varphi(d d^\ast)\varphi((\unit_\D-\omega)^\ast(\unit_\D-\omega))} = 0\text{.}
\end{split}
\end{equation*}
The proof is identical to show that $\varphi(\omega d) = \varphi(d)$.

Conversely, assume given $\varphi\in\StateSpace(\D)$ such that $\varphi(d)=\varphi(d\omega)=\varphi(\omega d)$ for all $d\in\D$. Then in particular, $1=\varphi(\unit_\D) = \varphi(\omega\unit_\D) = \varphi(\omega)$. Moreover, $\varphi(\omega^\ast\omega) = \varphi(\omega^\ast) = 1$. Thus:
\begin{equation*}
\varphi((\unit_\D-\omega)^\ast(\unit_\D-\omega)) = 1 - \varphi(\omega) - \varphi(\omega^\ast) + \varphi(\omega^\ast\omega) = 0\text{.}
\end{equation*}
A similar computation shows $\varphi((\unit_\D-\omega)(\unit_\D-\omega)^\ast) = 0$, so $\varphi \in \mathscr{S}_1(\omega)$.
\end{proof}

\begin{remark}
Since, for any element $\omega$ of a unital C*-algebra $\D$, and for any state $\varphi$ in its $1$-level, we have $\varphi(\omega)=1$, any element with non-empty $1$-level has norm at least one.
\end{remark}

The fundamental notion on which our construction relies is:

\begin{definition}\label{bridge-def}
A \emph{bridge} $(\D,\omega,\pi_\A,\pi_\B)$ from a unital C*-algebra $\A$ to another unital C*-algebra $\B$ is a unital C*-algebra $\D$, an element $\omega \in \D$ with non-empty $1$-level, and two unital *-monomorphisms $\pi_\A : \A \hookrightarrow \D$ and $\pi_\B:\B \hookrightarrow \D$. The element $\omega$ is called the \emph{pivot} element of the bridge.
\end{definition}

\begin{notation}
For any two unital C*-algebras $\A$ and $\B$, the set of all bridges from $\A$ to $\B$ is denoted by $\bridgeset{\A}{\B}$.
\end{notation}

\begin{remark}
The theory developed in this section would carry over if we changed the notion of bridge slightly to $4$-tuples $(\B(\Hilbert),\omega,\pi_\A,\pi_\B)$, where $\Hilbert$ is a Hilbert space, $\pi_\A$ and $\pi_\B$ are faithful non-degenerate *-representations of $\A$ and $\B$, respectively, and $\omega$ is a possibly unbounded operator on $\Hilbert$ with non-empty $1$-level. On the other hand, we shall try, in our work, to minimize numerical quantities associated with bridges which, it would seem, would only get larger by allowing unbounded operators.
\end{remark}

\begin{remark}
We pause for a note on our terminology. A bridge is defined \emph{from} a unital C*-algebra \emph{to} another, which may appear needlessly asymmetric. However, as we noted in Remark (\ref{selfadjointpivot-rmk}), we find no justification to impose our pivot elements to be self-adjoint; thus our bridges are indeed asymmetric, and we prefer to emphasize this matter with our choice of vocabulary. In Proposition (\ref{triangle-prop}), we shall see that to any bridge from a space to another, we can associate an inverse bridge going in the opposite direction.
\end{remark}

To any bridge between unital two C*-algebras, we can associate a seminorm:
\begin{definition}\label{bridgenorm-def}
The \emph{seminorm of a bridge} $\gamma=(\D,\omega,\pi_\A,\pi_\B)$ from a unital C*-algebra $\A$ to a unital C*-algebra $\B$ is the seminorm $\bridgenorm{\gamma}{\cdot}$ on $\A\oplus\B$ defined for all $(a,b) \in\sa{\A\oplus\B}$ by:
\begin{equation*}
\bridgenorm{\gamma}{a,b} = \left\|\pi_\A(a)\omega - \omega\pi_\B(b)\right\|_\D \text{.}
\end{equation*}
\end{definition}

We easily check that, given two unital C*-algebras $\A$ and $\B$, for all $(a,b)$,$(c,d)$ in $\A\oplus\B$ and any bridge $\gamma = (\D,\omega,\pi_\A,\pi_\B)$ from $\A$ to $\B$, we have:
\begin{equation}\label{bridgenorm-eq}
\begin{split}
\bridgenorm{\gamma}{\left(a,b\right)\left(c,d\right)} &= \|\pi_\A(ac)\omega - \omega \pi_\B(bd)\|_\D\\
&= \|\pi_\A(a)\pi_\A(c) \omega - \pi_\A(a)\omega \pi_\B(d)\\
&\quad +  \pi_\A(a)\omega \pi_\B(d) - \omega \pi_\B(b)\pi_\B(d)\|_\D\\
&\leq \|a\|_\A \|\pi_\A(c)\omega - \omega \pi_\B(d)\|_\D+\|\pi_\A(a)\omega - \omega \pi_\B (b)\|_\D \|d\|_\B\\
&= \|a\|_\A \bridgenorm{\gamma}{c,d} + \bridgenorm{\gamma}{a,b}\|d\|_\B \text{.}
\end{split}
\end{equation}
Thus, $\bridgenorm{\gamma}{\cdot}$ satisfies a form of the Leibniz identity. Now, if $(\A,\Lip_\A)$ and $(\B,\Lip_\B)$ are two {\Lqcms s} and $r > 0$, and if we set:
\begin{equation*}
\Lip_r : (a,b) \in \sa{\A\oplus\B} \longmapsto \max \left\{ \Lip_\A(a),\Lip_\B(b), r \bridgenorm{\gamma}{a,b} \right\}
\end{equation*}
then $(\A\oplus\B,\Lip_r)$ is a unital Leibniz pair (which follows from $\bridgenorm{\gamma}{\unit_\A,0}=\|\omega\|_\D \geq 1$ by assumption). As we shall see later in Theorem (\ref{Leibniz-lip-thm}), with the right choice for $r$ (given by the notion of a length of a bridge, to be defined shortly), the seminorm $\Lip_r$ is in fact an admissible Lip-norm for $(\Lip_\A,\Lip_\B)$ in the sense of \cite{Rieffel00}, and thus $(\A\oplus\B,\Lip_r)$ becomes a {\Lqcms} which can be used to construct estimates for the quantum Gro\-mov-Haus\-dorff distance. This is the starting point for our idea of a new metric, and we will take advantage of this observation in Theorem (\ref{Leibniz-lip-thm}) below. Furthermore, if $\Lip_\A$ and $\Lip_\B$ are strong Leibniz, then $\Lip_r$ can easily be checked to be strong Leibniz, by \cite[Theorem 6.2]{Rieffel10c}. This kind of strong Leibniz Lip-norms appears explicitly in \cite{Rieffel09}. Now, in the last paragraph of \cite[Section 2]{Rieffel09}, the challenge of using such Lip-norms to derive interesting estimates for Rieffel's proximity or for Rieffel's quantum Gro\-mov-Haus\-dorff distance is raised. We propose a solution to this challenge in this paper.

\bigskip
The main question is now: given two {\Lqcms s} $(\A,\Lip_\A)$ and $(\B,\Lip_\B)$, how do we associate numerical values to a bridge as defined in Definition (\ref{bridge-def}), such that these values provide estimates for the quantum Gro\-mov-Haus\-dorff distance?

In order to define our distance, we start with the following sets.

\begin{notation}\label{mulip-notation}
For any Lipschitz pair $(\A,\Lip)$, we define:
\begin{equation*}
\Lipball{1}{\A,\Lip} = \{ a \in \sa{\A} : \Lip(a)\leq 1\}
\end{equation*}
and, as in Theorem (\ref{az-thm}), we also define:
\begin{equation*}
\mulip(\A,\Lip,\varphi) = \{ a \in \sa{\A} : \Lip(a)\leq 1\text{ and }\varphi(a) = 0\}
\end{equation*}
for any given $\varphi \in \StateSpace(\A)$.
\end{notation}

Using Theorem (\ref{az-thm}), we have the following key observation:

\begin{lemma}\label{az-lemma}
Let $(\A,\Lip)$ be a {\qcms} such that $\Lip$ is lower semi-continuous for the norm of $\A$. Then the set $\mulip(\A,\Lip,\varphi)$ is norm-compact in $\A$ for any $\varphi\in\StateSpace(\A)$.
\end{lemma}

\begin{proof}
Fix $\varphi \in \StateSpace(\A)$. By lower semi-continuity of $\Lip$, the set $\Lipball{1}{\A,\Lip}$ is a norm-closed set in $\A$. Therefore, $\mulip(\A,\Lip,\varphi)$ is closed in $\A$ as well by continuity of $\varphi$. Hence, since $\A$ is complete, the space $\mulip(\A,\Lip,\varphi)$ is complete in norm. Now, $\mulip(\A,\Lip,\varphi)$ is totally bounded in $\A$ by Theorem (\ref{az-thm}) since $(\A,\Lip)$ is a {\qcms}. Thus, as a complete totally bounded space in the norm topology, $\mulip(\A,\Lip,\varphi)$ is norm-compact in $\A$.
\end{proof}

\bigskip

At the core of our construction is the notion of the \emph{length} of a bridge between {\Lqcms s}, which is build out of two numerical values, which we call the \emph{reach} and the \emph{height} of a bridge. We shall use:

\begin{notation}
For any metric space $(E,\mathsf{d})$, the Haus\-dorff distance induced by $\mathsf{d}$ on the set of all compact subsets of $E$ is denoted by $\Haus{\mathsf{d}}$. If $(E,\|\cdot\|_E)$ is a vector space and $\mathsf{d}$ is the distance induced by the norm on $E$, we denote $\Haus{\mathsf{d}}$ simply by $\Haus{E}$.
\end{notation}

As usual, if $E\subseteq \A$ for some algebra $\A$, and $a \in \A$, then $Ea = \{ ea : e\in E\}$ and $aE = \{ ae : e\in E \}$. 

Let $(\A,\Lip_\A)$ and $(\B,\Lip_\B)$ be two {\Lqcms s} and let $\gamma = (\D,\omega,\pi_\A,\pi_\B)$ be a bridge from $\A$ to $\B$. Fix $\varphi_\A \in \StateSpace(\A)$ and $\varphi_\B \in \StateSpace(\B)$. Since both $(\A,\Lip_\A)$ and $(\B,\Lip_\B)$ are {\Lqcms s}, by Lemma (\ref{az-lemma}), the sets:
\begin{equation*}
\pi_\A\left(\mulip\left(\A,\Lip_\A,\varphi_\A\right) \right)\omega\,\text{ and }\,
\omega \pi_\B\left(\mulip\left(\B,\Lip_\B,\varphi_\B\right) \right)
\end{equation*}
are norm-compact in $\D$ since $\pi_\A$ and $\pi_\B$ are continuous. Thus, their Haus\-dorff distance $\delta$ is finite. Let us denote it by $\delta$. Now, let $b \in \sa{\B}$ with $\Lip_\B(b)\leq 1$. Then there exists $a' \in \sa{\A}$ with $\Lip_\A(a')\leq 1$ and $\varphi_\A(a')=0$ such that:
\begin{multline*}
\bridgenorm{\gamma}{a',b-\varphi_\B(b)\unit_\A} =\\
\|\pi_\A(a')\omega - \omega\pi_\B(b-\varphi_\B(b)\unit_\B) \|_\D = \\ \min \{ \|\pi_\A(a)\omega - \omega\pi_\B(b-\varphi_\B(b)\unit_\B)\|_\D : a\in\mulip(\A,\Lip_\A,\varphi_\A) \} \leq \delta
\end{multline*}
by compactness. Thus:
\begin{equation*}
\bridgenorm{\gamma}{a'+\varphi_\B(b)\unit_\A,b}=\|\pi_\A(a'+\varphi_\B(b)\unit_\A)\omega - \omega\pi_\B(b)\|_\D\leq \delta\text{.}
\end{equation*}

As the situation is symmetric in $(\A,\Lip_\A)$ and $(\B,\Lip_\B)$, we conclude that the Haus\-dorff distance between the sets $\{\pi_\A(a)\omega:a\in\Lipball{1}{\A,\Lip_\A}\}$ and $\{\omega\pi_\B(b) : b\in \Lipball{1}{\B,\Lip_\B}\}$ is finite. With this in mind, we define:

\begin{definition}\label{bridgereach-def}
Let $(\A,\Lip_\A)$ and $(\B,\Lip_\B)$ be two {\qcms s}. The \emph{reach} of a bridge $\gamma=(\D,\omega,\pi_\A,\pi_\B)$ from $\A$ to $\B$ for $(\Lip_\A,\Lip_\B)$ is the non-negative real number $\bridgereach{\gamma}{\Lip_\A,\Lip_\B}$ defined by:
\begin{equation*}
\Haus{\D}\left(\pi_\A\left(\Lipball{1}{\A,\Lip_\A}\right)\omega,\omega \pi_\B\left(\Lipball{1}{\B,\Lip_\B}\right)\right)\text{.}
\end{equation*}
\end{definition}

Due to its central importance, we shall pause to provide a slightly different expression for the reach of a bridge. Let $(\A_1,\Lip_1)$ and $(\A_2,\Lip_2)$ be two {\qcms s} and let $\gamma = (\D,\omega,\pi_1,\pi_2)$ be a bridge between $\A_1$ and $\A_2$. The bridge seminorm was defined in Definition (\ref{bridgenorm-def}) by:
\begin{equation*}
(a,b) \in \sa{\A_1\oplus\A_2} \longmapsto \bridgenorm{\gamma}{a,b} = \| \pi_1(a)\omega-\omega\pi_2(b) \|_\D \text{.}
\end{equation*}
The reach of $\gamma$ can be expressed using the bridge seminorm as follows. First, we define for all $a\in\Lipball{1}{\A_1,\Lip_1}$:
\begin{equation*}
\bridgenorm{\gamma}{a,\Lipball{1}{\A_2,\Lip_2}} = \inf \{ \bridgenorm{\gamma}{a,b} : b \in \Lipball{1}{\A_2,\Lip_2} \}\text{,}
\end{equation*}
and similarly, we define $\bridgenorm{\gamma}{\Lipball{1}{\A_1,\Lip_1},b}$ for all $b\in\Lipball{1}{\A_2,\Lip_2}$.

Then we have by definition:
\begin{multline*}
\bridgereach{\gamma}{\Lip_1,\Lip_2} = \sup\{ \bridgenorm{\gamma}{a,\Lipball{1}{\A_2,\Lip_2}}, \bridgenorm{\gamma}{\Lipball{1}{\A_1,\Lip_1},b}:\\a\in\Lipball{1}{\A_1,\Lip_1},b\in\Lipball{1}{\A_2,\Lip_2} \}\text{.}
\end{multline*}

Thus, the reach of $\gamma$ measures how far apart our two {\qcms s} are \emph{in terms of the bridge seminorm}. The use of the bridge seminorm, rather than the norm in $\D$, allows us to ``cut-off'' elements in the Lip-balls of $\A$ and $\B$ by the pivot element $\omega$, which is our replacement for the sort of truncation arguments used in \cite{Rieffel00,Latremoliere05,Rieffel01} where strict order-unit subspaces of $\sa{\A}$ and $\sa{\B}$ would be involved.

\bigskip

Lip-norms of {\Lqcms s} are lower semi-cont\-inuous by assumption. Under this assumption, the following lemma shows that the minimization problem for the bridge seminorm over Lip-balls admits a solution, which will prove useful for the proof of  our main result, Theorem (\ref{main}).

\begin{lemma}\label{bridgereach-reached-lem}
Let $(\A,\Lip_\A)$ and $(\B,\Lip_\B)$ be two {\qcms s} such that $\Lip_\B$ is lower-semi-continuous. Let $\gamma\in\bridgeset{\A}{\B}$. Then, for all $a\in\dom{\Lip_\A}$ there exists $b \in \dom{\Lip_\B}$ such that:
\begin{equation}\label{bridgereach-reached-eq}
\bridgenorm{\gamma}{a,b} = \min \left\{ \bridgenorm{\gamma}{a,c} : c \in \dom{\Lip_\B} \text{ and }\Lip_\B(c)\leq \Lip_\A(a)  \right\}\text{.}
\end{equation}
\end{lemma}

\begin{proof}
We shall write $\gamma=(\D,\omega,\pi_\A,\pi_\B)$. Let $a\in\dom{\Lip_\A}$. If $a=t\unit_\A$ for some $t\in\R$, then $b = t\unit_\B$ satisfies Equality (\ref{bridgereach-reached-eq}). Assume $\Lip_\A(a)>0$ and let $a' = \Lip_\A(a)^{-1}a$. Fix $\varphi \in \StateSpace(\B)$. By Theorem (\ref{az-thm}), the set:
\begin{equation*}
\mulip(\B,\Lip_\B,\varphi) = \left\{ c \in \sa{\B} : \Lip_\B(c) \leq 1\text{ and } \varphi(c) =0 \right\}
\end{equation*}
introduced in Notation (\ref{mulip-notation}) is norm-compact by Lemma (\ref{az-lemma}).

Let:
\begin{equation*}
\alpha =  \inf \left\{ \bridgenorm{\gamma}{a',c} : c \in \dom{\Lip_\B} \text{ and }\Lip_\B(c)\leq 1  \right\}\text{,}
\end{equation*}
and note that:
\begin{multline*}
\inf \left\{ \bridgenorm{\gamma}{a,c} : c \in \dom{\Lip_\B} \text{ and }\Lip_\B(c)\leq \Lip_\A(a)  \right\} \\ = \Lip_\A(a)\inf \left\{ \bridgenorm{\gamma}{a',c} : c \in \dom{\Lip_\B} \text{ and }\Lip_\B(c)\leq 1  \right\}=\Lip_\A(a)\alpha\text{.}
\end{multline*}

Let $n\in \N, n>0$. By definition of $\alpha$, there exists $b_n \in \mulip(\B,\Lip_\B,\varphi)$ and $t_n \in \R$ such that:
\begin{equation*}
\alpha\leq \bridgenorm{\gamma}{a',b_n+t_n\unit_\B} \leq \alpha  + \frac{1}{n} \text{.}
\end{equation*}
Since $(b_n)_{n\in\N}$ is a sequence in the norm compact set $\mulip(\B,\Lip_\B,\varphi)$, it admits a norm convergent subsequent, which we denote by $(b_{\mu(n)})_{n\in\N}$, with limit some $b \in \mulip(\B,\Lip_\B,\varphi)$. Now, for any $n\in\N$, and since $\|\omega\|_\D\geq 1$ by Definitions (\ref{bridge-def}) and (\ref{one-level-def}), we have:
\begin{equation}\label{bridgereach-reached-eq1}
\begin{split}
|t_n| &\leq \|\omega\|_\D|t_n| \\
&\leq \bridgenorm{\gamma}{a',b_n+t_n\unit_\B} + \bridgenorm{\gamma}{a',b_n}\\
&\leq \alpha + 1 + \bridgenorm{\gamma}{a',b_n} \text{.}
\end{split}
\end{equation}
Now, 
$\left(\bridgenorm{\gamma}{a',b_{\mu(n)}}\right)_{n\in\N}$ converges in norm and is thus bounded. Consequently, Inequality (\ref{bridgereach-reached-eq1}) shows that $(t_{\mu(n)})_{n\in\N}$ is a bounded sequence of real numbers, which must thus admits a convergent sub-sequence, denoted by $(t_{\nu(n)})_{n\in\N}$, whose limit is denoted by $t$. Thus, we have:
\begin{equation*}
\begin{split}
\alpha&\leq 
\bridgenorm{\gamma}{a',b+t \unit_\B}\\
&= \lim_{n\rightarrow\infty} 
\bridgenorm{\gamma}{a',b_{\nu(n)}+t_{\nu(n)}}\\
&\leq \lim_{n\rightarrow\infty}\left(\alpha + \frac{1}{\nu(n)}\right) = \alpha \text{.}
\end{split}
\end{equation*}
Thus $\Lip_\A(a)(b+t \unit_\B)$ satisfies Equality (\ref{bridgereach-reached-eq}) as desired.
\end{proof}

The reach of the bridge is the numerical value we shall use to measure how far apart two {\Lqcms s} are from each other for a given bridge. Let us pause to give a metaphor for the ideas behind the notion of reach and the upcoming notion of a height of a bridge. The reach, using a physical bridge as a metaphor, measures the size of the ``span of the bridge''. However, in order to cross one of our bridge, it is necessary to be able to get to its span: a bridge is only as useful, after all, as you can actually get to it. If we wish to bring the state spaces of our two {\Lqcms s} close together using a bridge, then, in addition to the crossing span of the bridge, we also must know how far any given state is to the bridge. This other quantity is measured by our second core numerical value associated to a bridge, and is an essential component in estimating how useful a given bridge is to compute estimates on distances between states.  We call this second quantity the \emph{height} of the bridge. The metaphor of a bridge can be seen now as follows: a bridge is a span which goes from the $1$-level set of the pivot element in one {\Lqcms} to the $1$-level set of the pivot element in the other {\Lqcms} (this could be our horizontal component of the bridge); to ``travel'' from one state space to another, one then must climb an ``access ramp'', to get to the $1$-level set of the pivot, before crossing the bridge, and this climb size is measured by the height of the bridge.

Somewhat more formally, given a bridge $\gamma = (\D,\omega,\pi,\rho)$ from some {\Lqcms} to another, we wish to know how far $\omega$ is from the identity of $\D$, in some sense. If $\omega = \unit_\D$, then the bridge norm for $\gamma$ is given by the norm in $\D$, and it becomes easy to compute estimates of the distance between states based on the bridge seminorm. In general, we would want $\omega$ to be close enough to $\unit_\D$ to still give relevant estimates. An appropriate choice is to use the quantum metric structures to introduce a weak form a measurement of the difference between $\omega$ and $\unit_\D$ which fits our purpose. This second quantity associated to bridges is based on the $1$-level introduced in Definition (\ref{one-level-def}), which is never empty for a bridge.

\begin{definition}\label{bridgeheight-def}
Let $(\A,\Lip_\A)$ and $(\B,\Lip_\B)$ be two {\qcms s} and $\gamma=(\D,\omega,\pi_\A,\pi_\B)$ be a bridge from $\A$ to $\B$. Let:
\begin{equation*}
\StateSpace(\A|\gamma) = \left\{ \varphi \circ \pi_\A : \varphi \in \mathscr{S}_1(\omega) \right\} 
\end{equation*}
and similarly:
\begin{equation*}
\StateSpace(\B|\gamma) = \left\{ \varphi \circ \pi_\B : \varphi \in \mathscr{S}_1(\omega) \right\}\text{.} 
\end{equation*}

The \emph{height} $\bridgeheight{\gamma}{\Lip_\A,\Lip_\B}$ of the bridge $\gamma$ for $(\Lip_\A,\Lip_\B)$ is the real number defined by:
\begin{equation*}
\max \left\{ \Haus{\Kantorovich{\Lip_\A}}(\StateSpace(\A),\StateSpace(\A|\gamma)),  \Haus{\Kantorovich{\Lip_\B}}(\StateSpace(\B),\StateSpace(\B|\gamma)) \right\}\text{.}
\end{equation*}
\end{definition}


\bigskip

For our purpose, the proper notion of the length of a bridge is given by:

\begin{definition}\label{bridgelength-def}
Let $(\A,\Lip_\A)$ and $(\B,\Lip_\B)$ be two {\Lqcms s}. The \emph{length} of a bridge $\gamma$ for $(\Lip_\A,\Lip_\B)$ is the non negative real number $\bridgelength{\gamma}{\Lip_\A,\Lip_\B}$ defined by:
\begin{equation*}
\max \{ \bridgereach{\gamma}{\Lip_\A,\Lip_\B}, \bridgeheight{\gamma}{\Lip_\A,\Lip_\B} \}\text{.}
\end{equation*}
\end{definition}

We wish to use our notions of bridge reach and height to build a metric on {\Lqcms s}. The natural idea is to consider, given two {\Lqcms s} $(\A,\Lip_\A)$ and $(\B,\Lip_\B)$, the infimum of the reach, the height or the length over all possible bridges from $\A$ to $\B$. All three ideas, however, do not quite give a satisfactory answer. The infimum of the reach would be zero in too many cases, as checked on very simple examples. The idea of the height of a bridge $(\D,\omega,\pi_\A,\pi_\B)$ is precisely to ensure that, in some sense, $\omega$ is close to the unit of $\D$ for the purpose of constructing a distance rather than a pseudo-distance. Without this, the reach is essentially meaningless in general. However, the height does not satisfy any triangle inequality. Indeed, if one is given three {\Lqcms s} $(\A_1,\Lip_1)$, $(\A_2,\Lip_2)$ and $(\A_3,\Lip_3)$ and two bridges $(\D_1,\omega_1,\pi_1,\rho_2)$ and $(\D_2,\omega_2,\pi_2,\rho_2)$, respectively from $\A_1$ to $\A_2$ and from $\A_2$ to $\A_3$, then a natural way to construct a bridge from $\A_1$ to $\A_3$ is to consider the amalgamated free product $\mathfrak{E} = \D_1\star_{\A_2}\D_2$ \cite{Blackadar78}. If $\iota_j: \D_j \rightarrow \mathfrak{E}$ are the canonical *-monomorphisms constructed in \cite{Blackadar78}, then the obvious candidate for an ``amalgamated bridge'' is $(\mathfrak{E},\omega_1\omega_2,\iota_1\circ \pi_1,\iota_2\circ\rho_2)$. Such a construction, for instance, was used by Li in \cite{li03} for a similar purpose in dealing with the nuclear distance. However, two problems arise. First, we have no reason to expect $\omega_1\omega_2 \not = 0$, let alone $\|\omega_1\omega_2\|_{\mathfrak{E}} \geq 1$. Second of all, even under the best circumstances, the height of this amalgamated bridge would most likely grow by an unpredictable amount, as it is very unclear how to relate the $1$-level set of $\omega_1\omega_2$ with the $1$-level sets of $\omega_1$ and $\omega_2$. This construction, though the natural path to attempt a proof of the triangle inequality for the natural ideas of metrics introduced at the start of this paragraph, would thus fail to provide the desired estimates. However, we shall still work with the same general idea, yet enforcing the triangle inequality in our very definition.

\bigskip

To address this problem, we introduce the notion of a (finite) path between two {\qcms s}, built with bridges. The bridges in such paths involve intermediate {\qcms s}, and we must choose which class of {\qcms s} we wish to allow in this construction. This choice actually is a nice feature of our construction, as it allows us to adapt the idea of the quantum propinquity to various settings, as we shall see later in this section.

\begin{notation}
The set $\N\setminus\{0\}$ of nonzero natural numbers is denoted by $\N_+$.
\end{notation}

\begin{notation}
Whenever convenient, $n$-tuples $(a_1,\ldots,a_n)$ will be denoted by:
\begin{equation*}
(a_j : j=1,\ldots, n)\text{.}
\end{equation*}
Moreover, if $a_1,\ldots,a_n$ are themselves tuples, we will drop the parentheses around each $a_j$ ($j=1,\ldots,n)$ to simplify our notation further.
\end{notation}

\begin{definition}\label{trek-def}
Let $\mathcal{C}$ be a nonempty class of {\qcms s} and let $(\A,\Lip_\A)$ and $(\B,\Lip_\B)$ be two {\qcms s} in $\mathcal{C}$. A \emph{$\mathcal{C}$-trek} from $(\A,\Lip_\A)$ to $(\B,\Lip_\B)$ is a $n$-tuple:
\begin{equation*}
( \A_j,\Lip_j,\gamma_j,\A_{j+1},\Lip_{j+1} : j=1,\ldots,n )
\end{equation*}
for some $n\in \N_+$, where $(\A_1,\Lip_1)=(\A,\Lip_\A)$, $(\A_{n+1},\Lip_{n+1}) = (\B,\Lip_\B)$, while for all $j = \{ 1,\ldots,n\}$, the pair $(\A_j,\Lip_j)$ is a {\qcms} in $\mathcal{C}$ and $\gamma_j$ is a bridge from $\A_j$ to $\A_{j+1}$.
\end{definition}

\begin{notation}
Let $\mathcal{C}$ be a nonempty class of {\qcms s}. For any two {\qcms s} $(\A,\Lip_\A)$ and $(\B,\Lip_\B)$ in $\mathcal{C}$, the set of all $\mathcal{C}$-treks from $(\A,\Lip_\A)$ to $(\B,\Lip_\B)$ is denoted by:
\begin{equation*}
\startrekset{\mathcal{C}}{\A,\Lip_\A}{\B,\Lip_\B}\text{.}
\end{equation*}
\end{notation}

Our main interest will lie with the class $\LCQMS$ of {\Lqcms s} introduced in Notation (\ref{lcqmss-not}). We have two reasons to introduce the more general notion of a trek associated to some class of {\qcms s}. First, this will allow us to make clear, in our exposition, when the Leibniz property plays an important role, by stating when a given result is proven using treks along general classes of {\qcms s} or when it uses sub-classes of {\Lqcms s}. Second, we may use treks along sub-classes of {\Lqcms s} to define various mild strengthening of our quantum propinquity, which may better fit some specific situations. Thus, our terminology will make it easy to incorporate all these variants into one theory.

We can thus define the length of a trek:

\begin{definition}\label{treklength-def}
The length $\treklength{\Gamma}$ of a trek $\Gamma = ( \A_j,\Lip_j,\gamma_j,\A_{j+1},\Lip_{j+1} : j=1,\ldots,n )$ is the non-negative real number:
\begin{equation*}
\treklength{\Gamma} = \sum_{j = 1}^n \bridgelength{\gamma_j}{\Lip_j,\Lip_{j+1}} \text{.}
\end{equation*}
\end{definition}

\section{The Quantum Propinquity}

We now define the core concept of this paper:


\begin{definition}\label{propinquity-def}
Let $\mathcal{C}$ be a nonempty class of a {\qcms s}. The \emph{quantum Gro\-mov-Haus\-dorff $\mathcal{C}$-propinquity} between two {\qcms s} $(\A,\Lip_\A)\in \mathcal{C}$ and $(\B,\Lip_\B)\in \mathcal{C}$  is the non-negative real number:
\begin{equation*}
\propinquity_{\mathcal{C}}((\A,\Lip_\A),(\B,\Lip_\B)) = \inf \left\{ \treklength{\Gamma} : \Gamma \in \startrekset{\mathcal{C}}{\A,\Lip_\A}{\B,\Lip_\B} \right\} \text{.}
\end{equation*}
\end{definition}

We shall often refer to the quantum Gro\-mov-Haus\-dorff $\mathcal{C}$-propinquity simply as the quantum $\mathcal{C}$-propinquity. Our paper is mostly devoted to the quantum $\LCQMS$-propinquity $\propinquity_{\LCQMS}$, and thus we introduce a specific notation and terminology for the associated quantum propinquity:
\begin{definition}
The \emph{quantum Gro\-mov-Haus\-dorff propinquity} between two {\Lqcms s} $(\A,\Lip_\A)$ and $(\B,\Lip_\B)$  is the non-negative real number  $\propinquity((\A,\Lip_\A),(\B,\Lip_\B))$ defined as:
\begin{equation*}
\inf \left\{ \treklength{\Gamma} : \Gamma \in \startrekset{\LCQMS}{\A,\Lip_\A}{\B,\Lip_\B} \right\} \text{.}
\end{equation*}
\end{definition}

The quantum Gro\-mov-Haus\-dorff propinquity $\propinquity$ is the main subject of this paper. However, our definition allows for some refinements of the following type:
\begin{example}
A natural choice of a subclass $\mathcal{C}$ of {\Lqcms s} is the class $\mathcal{C}^\ast$ of compact C*-metric spaces introduced by Rieffel in \cite{Rieffel10c}, which are {\Lqcms s} whose Lip-norms are the restrictions of densely defined, lower-semi continuous strong Leibniz seminorms. We can thus work with $\propinquity_{\mathcal{C}^\ast}$ instead of $\propinquity$ if one wishes to ensure all {\Lqcms s} involved in treks are, in fact, compact C*-metric spaces.
\end{example}

More generally, our theory is particularly adapted to any nonempty subclass of {\Lqcms s} which may be relevant to a particular application. In general, the motivation to choose a subclass of {\Lqcms s} would be to study the behavior for the quantum propinquity of some structure which is only well-defined over some {\Lqcms s}. Nonetheless, it should also be noted that many results of this section do not require the Leibniz property, and in fact our main Theorem (\ref{main}) could be adapted to other classes of {\qcms s} for which the Lip-norms are well-behaved with respect to the multiplication. Our framework would allow for such extensions to be quite routine, if they ever prove desirable.

Another direction in which one could specialize our construction of the quantum propinquity would be to make additional requirements on the pivot elements of the bridges involved in admissible treks. For instance, one could impose a uniform bound on the norm of these pivot elements, or that these pivot elements should be self-adjoint. This shows that our construction allows for quite some flexibility and adaptability to future needs in noncommutative metric geometry.

\begin{remark}
We note that for any nonempty class $\mathcal{C}$ of {\Lqcms s}, the quantum $\mathcal{C}$-propinquity $\propinquity_{\mathcal{C}}$ dominates the quantum propinquity, and that the quantum propinquity $\propinquity$ is also, with our notation, $\propinquity_{\LCQMS}$. 
\end{remark}

\bigskip

Our purpose is to prove that the quantum propinquity is a distance on classes of isometrically isomorphic {\Lqcms s}. We first check that our quantum propinquity is always real-valued.

\begin{notation}
We denote the diameter of a metric space $(E,\mathsf{d})$ by $\diam{E}{\mathsf{d}}$.
\end{notation}

\begin{proposition}\label{first-bridge-prop}
Let $\mathcal{C}$ be a nonempty class of {\qcms s}. Let $(\A,\Lip_\A)$ and $(\B,\Lip_\B)$ be two {\qcms s} in $\mathcal{C}$. Then:
\begin{equation*}
\propinquity_{\mathcal{C}}((\A,\Lip_\A),(\B,\Lip_\B)) \leq \max\{ \diam{\StateSpace(\A)}{\Kantorovich{\Lip_\A}}, \diam{\StateSpace(\B)}{\Kantorovich{\Lip_\B}} \}\text{.}
\end{equation*}
\end{proposition}

\begin{proof}
Let:
\begin{equation*}
D = \max\{ \diam{\StateSpace(\A)}{\Kantorovich{\Lip_\A}}, \diam{\StateSpace(\B)}{\Kantorovich{\Lip_\B}} \}\text{.}
\end{equation*}
Let $\Hilbert$ be a separable, infinite dimensional Hilbert space. Since $\A$ and $\B$ are separable, there exist faithful non-degenerate *-representations $\pi_\A$ and $\pi_\B$ of, respectively, $\A$ and $\B$ on $\Hilbert$. Let $\Omega$ be the C*-algebra of all bounded operators on $\Hilbert$ and let $\omega$ be the identity operator on $\Hilbert$. Let $\gamma = (\Omega,\omega,\pi_\A,\pi_\B)$. First, note that the $1$-level set $\mathscr{S}_1(\omega)$ is $\StateSpace(\Omega)$, and thus is nonempty. By Definition (\ref{bridge-def}), $\gamma$ is a bridge from $\A$ to $\B$. Furthermore, since $\StateSpace(\pi_\A)$ and $\StateSpace(\pi_\B)$ are both surjective, by Definition (\ref{bridgeheight-def}), we have $\bridgeheight{\gamma}{\Lip_\A,\Lip_\B} = 0$.

Let $\varphi_\A \in \StateSpace(\A)$. For all $\psi \in \StateSpace(\A)$ and for all $a\in\dom{\Lip_\A}$, we have:
\begin{multline*}
|\psi(a-\varphi_\A(a)\unit_\A)| = |\psi(a) - \varphi_\A(a)| \leq \Lip_\A(a)\Kantorovich{\Lip_\A}(\psi,\varphi_\A) \\ \leq \Lip_\A(a) \diam{\StateSpace(\A)}{\Kantorovich{\Lip}}\leq \Lip_\A(a)D \text{,}
\end{multline*}
hence $\|a-\varphi_\A(a)\unit_\A\|_\A \leq \Lip_\A(a)D$.

Let $a\in\dom{\Lip_\A}$ with $\Lip_\A(a)\leq 1$. Let $b = \varphi_\A(a)\unit_\B$ and note that $\Lip_\B(b) = 0\leq 1$. Then:
\begin{equation*}
\|\pi_\A(a)\omega - \omega\pi_\B(b)\|_\Omega = \|a-\varphi(a)\unit_\A\|_\A \leq D\text{.}
\end{equation*}
Similarly, given any $\varphi_\B \in \StateSpace(\B)$, we have $\|\pi_\A(\varphi_\B(b)\unit_\A)\omega-\omega\pi_\B(b)\|\leq D$ for all $b\in\dom{\Lip_\B}$ with $\Lip_\B(b)\leq 1$. Hence $\bridgereach{\gamma}{\Lip_\A,\Lip_\B} \leq D$ by Definition (\ref{bridgereach-def}).

Therefore, by Definition (\ref{bridgelength-def}), we have:
\begin{equation*}
\bridgelength{\gamma}{\Lip_\A,\Lip_\B} = \max\{\bridgereach{\gamma}{\Lip_\A,\Lip_\B}, \bridgeheight{\gamma}{\Lip_\A,\Lip_\B}\} \leq D\text{.}
\end{equation*}

Let $\Gamma = ((\A,\Lip_\A,\gamma,\B,\Lip_\B))$, which is a $\mathcal{C}$-trek from $(\A,\Lip_\A)$ to $(\B,\Lip_\B)$ with the same length as the bridge $\gamma$ for $(\Lip_\A,\Lip_\B)$. Thus by Definitions (\ref{treklength-def}) and (\ref{propinquity-def}), we have:
\begin{equation*}
\propinquity_{\mathcal{C}}((\A,\Lip_\A),(\B,\Lip_\B)) \leq \treklength{\Gamma} \leq D \text{,}
\end{equation*}
as desired.
\end{proof}

Now, we check that by construction, symmetry and the triangle inequality are satisfied:

\begin{proposition}\label{triangle-prop}
Let $\mathcal{C}$ be a nonempty class of {\qcms s}. Let $(\A,\Lip_\A)$, $(\B,\Lip_\B)$ and $(\D,\Lip_\D)$ be three elements of $\mathcal{C}$. We have:
\begin{enumerate}
\item $\propinquity_{\mathcal{C}}((\A,\Lip_\A),(\B,\Lip_\B)) = \propinquity_{\mathcal{C}}((\B,\Lip_\B),(\A,\Lip_\A))$,
\item $\propinquity_{\mathcal{C}}((\A,\Lip_\A),(\D,\Lip_\D)) \leq \propinquity_{\mathcal{C}}((\A,\Lip_\A),(\B,\Lip_\B)) + \propinquity_{\mathcal{C}}((\B,\Lip_\B),(\D,\Lip_\D))$.
\end{enumerate}
\end{proposition}

\begin{proof}
As a matter of notations, if $(\mathfrak{E},\omega,\pi,\rho)$ is a bridge from $(\A,\Lip_\A)$ to $(\B,\Lip_\B)$, then $\gamma^{-1} = (\mathfrak{E},\omega^\ast,\rho,\pi)$ is a bridge from $(\B,\Lip_\B)$ to $(\A,\Lip_\A)$. Of course, by Definition (\ref{bridgelength-def}):
\begin{equation*}
\bridgelength{\gamma^{-1}}{\Lip_\B,\Lip_\A} = \bridgelength{\gamma}{\Lip_\A,\Lip_\B}\text{.}
\end{equation*}

If $\Gamma = ( \A_j,\Lip_j,\gamma_j,\A_{j+1},\Lip_{j+1} : j=1,\ldots,n )$ is a $\mathcal{C}$-trek from $(\A,\Lip_\A)$ to $(\B,\Lip_\B)$, then $\Gamma^{-1} = ( \A_{n+2-j},\Lip_{n+2-j},\gamma^{-1}_{n+1-j},\A_{n+1-j},\Lip_{n+1-j} : j=1,\ldots,n )$ is a $\mathcal{C}$-trek from $(\B,\Lip_\B)$ to $(\A,\Lip_\A)$ by definition, and of course $\treklength{\Gamma} = \treklength{\Gamma^{-1}}$. This proves that our quantum propinquity is symmetric.

Let $\Gamma_1$ and $\Gamma_2$ be $\mathcal{C}$-treks, respectively, from $(\A,\Lip_\A)$ to $(\B,\Lip_\B)$ and from $(\B,\Lip_\B)$ to $(\D,\Lip_\D)$. We write:
\begin{align}
\Gamma_1 &= ( \A_j,\Lip_j, \gamma_j,\A_{j+1},\Lip_{j+1} : j = 1,\ldots,n_1 )\\
\intertext{and}
\Gamma_2 &= ( \B_j,\Lip'_j,\delta_j,\B_{j+1},\Lip'_{j+1} : j = 1,\ldots,n_2 )\text{.}
\end{align}
We have $(\A_{n_1+1},\Lip_{n_1+1}) = (\B,\Lip_\B) = (\B_1,\Lip'_1)$ by definition of treks. Thus, if we define for all $j \in \{1,\ldots,n_1+n_2\}$:
\begin{equation*}
(\mathfrak{E}_j,\Lip''_j) = \begin{cases}
(\A_j,\Lip_j)\text{ if $j = 1,\ldots,n_1$,}\\
(\B_{j-n_1},\Lip'_{j-n_1}) \text{ if $j=n_1+1,\ldots,n_2+n_1$,}
\end{cases}
\end{equation*}
and
\begin{equation*}
\kappa_j = \begin{cases}
\gamma_j \text{ if $j = 1,\ldots,n_1$,}\\
\delta_{j-n_1} \text{ if $j=n_1+1,\ldots,n_2+n_1$,}
\end{cases}
\end{equation*}
and if we write:
\begin{equation*}
\Gamma_1\star\Gamma_2 = ( \mathfrak{E}_j,\Lip''_j,\kappa_j, \mathfrak{E}_{j+1},\Lip_{j+1}'' : j = 1,\ldots,n_1+n_2 )
\end{equation*}
then $\Gamma_1\star\Gamma_2$ is a $\mathcal{C}$-trek from $(\A,\Lip_\A)$ to $(\D,\Lip_\B)$. Moreover, by construction, we have:
\begin{equation}\label{star-trek-product-eq}
\begin{split}
\treklength{\Gamma_1\star\Gamma_2} &= \sum_{j=1}^{n_1+n_2} \bridgelength{\kappa_j}{\Lip''_j,\Lip''_{j+1}}\\
&= \sum_{j=1}^{n_1}\bridgelength{\gamma_j}{\Lip_j,\Lip_{j+1}} + \sum_{j=1}^{n_2}\bridgelength{\delta_j}{\Lip'_j,\Lip'_{j+1}} \\
&= \treklength{\Gamma_1} + \treklength{\Gamma_2}\text{.}
\end{split}
\end{equation}

Now, let $\varepsilon > 0$. By definition of the quantum propinquity, there exists a $\mathcal{C}$-trek $\Gamma_1$ from $(\A,\Lip_\A)$ to $(\B,\Lip_\B)$ and a $\mathcal{C}$-trek from $\Gamma_2$ from $(\B,\Lip_\B)$ to $(\D,\Lip_\D)$ such that:
\begin{align*}
\treklength{\Gamma_1} &\leq \propinquity_{\mathcal{C}}((\A,\Lip_\A),(\B,\Lip_\B))+\frac{1}{2}\varepsilon\text{,}\\
\intertext{and}
\treklength{\Gamma_2} &\leq \propinquity_{\mathcal{C}}((\B,\Lip_\B),(\D,\Lip_\D)) + \frac{1}{2}\varepsilon\text{.}
\end{align*}

By Equation (\ref{star-trek-product-eq}), we have:
\begin{equation*}
\begin{split}
\propinquity_{\mathcal{C}}((\A,\Lip_\A),(\D,\Lip_\D)) &\leq \treklength{\Gamma_1\star\Gamma_2}\\
&= \treklength{\Gamma_1} + \treklength{\Gamma_2}\\
&\leq \propinquity_{\mathcal{C}}((\A,\Lip_\A),(\B,\Lip_\B)) + \propinquity_{\mathcal{C}}((\B,\Lip_\B),(\D,\Lip_\D)) + \varepsilon \text{.}
\end{split}
\end{equation*}
As $\varepsilon > 0$ is arbitrary, we conclude that the quantum propinquity $\propinquity$ satisfies the triangle inequality.
\end{proof}

\section{Distance Zero}
The idea to enforce the triangle inequality by taking our metric to be an infimum over all possible paths is a standard technique, but it has a price: we must put all our efforts in the proof that distance zero is equivalent to isometric isomorphism. To this end, we introduce some natural notions.


Let $(\A,\Lip_\A)$ and $(\B,\Lip_\B)$ be two {\Lqcms s} such that $\propinquity((\A,\Lip_\A),(\B,\Lip_\B)) = 0$. Our goal is to construct a quantum isometric isomorphism from $\A$ onto $\B$; in particular, we must associate to every $a\in\dom{\Lip_\A}$ an element $b\in\dom{\Lip_\B}$, in a manner consistent with the underlying algebraic structures. The following definition introduces, for any $a\in\dom{\Lip_\A}$, a set of natural candidates for the target of $a$ among all elements in $\dom{\Lip_\B}$, using first only bridges:

\begin{definition}\label{target-set-def}
Let $(\A,\Lip_\A)$ and $(\B,\Lip_\B)$ be two {\qcms s} and $\gamma \in \bridgeset{\A}{\B}$ be a bridge from $\A$ to $\B$. For any $r \geq \Lip_\A(a)$, we define the \emph{$r$-target set of $a$ for $\gamma$} by:
\begin{equation*}
\targetsetbridge{\gamma}{a}{r} = \left\{ b \in \sa{\B} : \Lip_\B(b)\leq r \text{ and }\bridgenorm{\gamma}{a,b} \leq r \bridgereach{\gamma}{\Lip_\A,\Lip_\B} \right\}\text{.}
\end{equation*}
\end{definition}

The first observation is that the target sets are never empty.

\begin{lemma}\label{nonempty-target-lem}
Let $(\A,\Lip_\A)$ and $(\B,\Lip_\B)$ be two {\qcms s} such that $\Lip_\B$ is lower semi-cont\-inuous, and let $\gamma \in \bridgeset{\A}{\B}$ be a bridge from $\A$ to $\B$. For any $r \geq \Lip_\A(a)$, the set $\targetsetbridge{\gamma}{a}{r}$ is not empty.
\end{lemma}

\begin{proof}
Let $a\in \dom{\Lip_\A}$. If $a = t\unit_\A$ for some $t\in\R$, then $t\unit_\B \in \targetsetbridge{\gamma}{a}{r}$. Assume $\Lip_\A(a)>0$. By Lemma (\ref{bridgereach-reached-lem}), there exists $b \in \sa{\B}$ such that $\Lip_\B(b) \leq \Lip_\A(a)$ and:
\begin{equation*}
\begin{split}
\bridgenorm{\gamma}{a,b} &= \min \{ \bridgenorm{\gamma}{a,c} : c\in\sa{\B}\text{ and }\Lip_\B(c) \leq \Lip_\A(a) \}\\
&=\Lip_\A(a) \min \{ \bridgenorm{\gamma}{\Lip_\A(a)^{-1}a,c} : c \in \sa{\B}\text{ and }\Lip_\B(c) \leq 1 \}\\
&\leq \Lip_\A(a) \bridgereach{\gamma}{\Lip_\A,\Lip_\B}\leq r\bridgereach{\gamma}{\Lip_\A,\Lip_\B} \text{.}
\end{split}
\end{equation*}
Thus $b \in \targetsetbridge{\gamma}{a}{r}$ as desired.
\end{proof}

We pause for a remark regarding the assumption of lower-semi-continuity in Lemma (\ref{nonempty-target-lem}). In general, let $(\A,\Lip_\A)$ and $(\B,\Lip_\B)$ be two {\qcms s} and define, for any $a\in\dom{\Lip_\A}$, $r \geq \Lip_\A(a)$ and $\varrho \geq \bridgereach{\gamma}{\Lip_\A,\Lip_\B}$, the generalized target set:
\begin{equation}\label{generalized-target-def}
\targetsetbridge{\gamma}{a}{r,\varrho} = \left\{ b \in \sa{\B} : \Lip_\B(b) \leq r\text{ and }\bridgenorm{\gamma}{a,b} \leq r\varrho \right\}\text{.}
\end{equation}

Then in general, $\targetsetbridge{\gamma}{a}{r,\varrho}$ is not empty if $\varrho > \bridgereach{\gamma}{\Lip_\A,\Lip_\B}$. Indeed, if $a = t\unit_\A$ for some $t\in \R$ then $t\unit_\B \in \targetsetbridge{\gamma}{a}{r,\varrho}$ for any $r \geq 0$. We now assume $\Lip_\A(a)>0$ and let $r \geq \Lip_\A(a)$. Let $a' = \Lip_\A(a)^{-1} a$. By Definition (\ref{bridgereach-def}) of the reach of a bridge, since $\Lip_\A(a')=1$ \emph{and $\varrho > \bridgereach{\gamma}{\Lip_\A,\Lip_\B}$}, there exists $b' \in \sa{\B}$ with $\Lip_\B(b') \leq 1$ such that:
\begin{equation*}
\|\pi_\A(a')\omega-\omega\pi_\B(b')\|_\D\leq \varrho \text{.}
\end{equation*}
Let $b = \Lip_\A(a)b'$. Then $\Lip_\B(b) \leq \Lip_\A(a)\Lip_\B(b') \leq \Lip_\A(a) \leq r$. Moreover, by construction:
\begin{equation*}
\bridgenorm{\gamma}{a,b} \leq \Lip_\A(a)\bridgenorm{\gamma}{a',b'}\leq\Lip_\A(a)\varrho \leq r\varrho \text{.}
\end{equation*}
Hence, $b \in \targetsetbridge{\gamma}{a}{r,\varrho}$, as desired. 

However, the set $\targetsetbridge{\gamma}{a}{r} = \targetsetbridge{\gamma}{a}{r,\bridgereach{\gamma}{\Lip_\A,\Lip_\B}}$ may be empty in general unless, as stated in Lemma (\ref{nonempty-target-lem}), we assume $\Lip_\B$ to be lower semi-cont\-inuous, so that we may apply Lemma (\ref{bridgereach-reached-lem}). Now, Theorem (\ref{main}) below will be established for the class of {\Lqcms s} and its sub-classes, and thus, to avoid making our exposition needlessly confusing by introducing two parameters in the definition of our target sets, when only one is needed for our main purpose, we shall work with lower semi-cont\-inuous seminorms henceforth. The curious reader may check that subsequent Propositions (\ref{fundamental-prop}),(\ref{diameter-prop}) remain true after dropping the requirement that $\Lip_\B$ be lower semi-cont\-inuous, and replacing the target sets $\targetsetbridge{\gamma}{a}{r}$ by the generalized target sets $\targetsetbridge{\gamma}{a}{r,\varrho}$ with $\varrho > \bridgelength{\gamma}{\Lip_\A,\Lip_\B}$, while substituting $\max\{\varrho,\bridgelength{\gamma}{\Lip_\A,\Lip_\B}\}$ for all occurrence of $\bridgelength{\gamma}{\Lip_\A,\Lip_\B}$. Simple similar adjustments could be made to Propositions (\ref{fundamental-trek-proposition}) and (\ref{product-prop}) to follow. 

\bigskip

The fundamental property of target sets, which is the subject of our next two results, is that their diameter admits an upper bound controlled by the bridge length from which they are built. The reason for the importance of this fact is that target sets, as defined, contain many possible candidates for the image of a single element; however if the quantum propinquity between two {\Lqcms s} $(\A,\Lip_\A)$ and $(\B,\Lip_\B)$ is small, then we can find relatively small target sets for each element $a\in\dom{\Lip_\A}$ thanks to the following fundamental lemma. This will help us focusing on a target singleton set for each element in $\dom{\Lip_\A}$ (and, by symmetry, for each element in $\dom{\Lip_\B}$). For this process to succeed, we however will need to use the completeness of the underlying C*-algebras of {\Lqcms s}.

The main observation of this section is thus given by the following Proposition, where the condition on the $1$-level set of pivot elements in bridges is used for the first time.

\begin{proposition}\label{fundamental-prop}
Let $(\A,\Lip_\A)$ and $(\B,\Lip_\B)$ be two {\qcms s} with $\Lip_\B$ lower semi-cont\-inuous, and $\gamma \in\bridgeset{\A}{\B}$ be a bridge from $\A$ to $\B$. For all $a \in \dom{\Lip_\A}$ and for all $r \geq \Lip_\A(a)$, we have:
\begin{equation*}
\sup \{ \|b\|_\B : b \in \targetsetbridge{\gamma}{a}{r} \} \leq 2r \bridgelength{\gamma}{\Lip_\A,\Lip_\B} + \|a\|_\A \text{.}
\end{equation*}
\end{proposition}

\begin{proof}
Let $a\in\dom{\Lip_\A}$ and $r\geq \Lip_\A(a)$. By Lemma (\ref{nonempty-target-lem}), we know that $\targetsetbridge{\gamma}{a}{r}$ is not empty. Let $b\in\targetsetbridge{\gamma}{a}{r,\lambda}$. By Definition (\ref{target-set-def}), we have:
\begin{equation*}
\left\{\begin{array}{l}
\Lip_\B(b)\leq r\\
\bridgenorm{\gamma}{a,b}\leq r\bridgereach{\gamma}{\Lip_\A,\Lip_\B}\leq r\bridgelength{\gamma}{\Lip_\A,\Lip_\B}\text{.}
\end{array}\right.
\end{equation*}
Let $\varphi \in \StateSpace(\B)$. By Definition (\ref{bridgeheight-def}) of the bridge height $\bridgeheight{\gamma}{\Lip_\A,\Lip_\B}$, there exists $\psi \in \StateSpace(\D)$ such that $\psi(d)=\psi(\omega d)=\psi(d\omega)$ for all $d\in \D$, and:
\begin{equation*}
\Kantorovich{\Lip_\B}(\varphi,\psi\circ\pi_\B) \leq \bridgeheight{\gamma}{\Lip_\A,\Lip_\B}\text{.}
\end{equation*}
Thus:
\begin{equation*}
\begin{split}
|\varphi(b)| &\leq |\varphi(b) - \psi(\pi_\B(b))| + |\psi(\omega\pi_\B(b))|\\
&\leq \Kantorovich{\Lip_\B}(\varphi,\psi\circ\pi_\B)\Lip_\B(b) + |\psi(\omega\pi_\B(b)) - \psi(\pi_\A(a)\omega)| + |\psi(\pi_\A(a)\omega)|\\
&\leq r\bridgeheight{\gamma}{\Lip_\A,\Lip_\B} + \|\pi_\A(a)\omega-\omega\pi_\B(b)\|_\D + |\psi(\pi_\A(a))|\\
&= r\bridgeheight{\gamma}{\Lip_\A,\Lip_\B} + \bridgenorm{\gamma}{a,b} + |\psi(\pi_\A(a))|\\
&\leq r\bridgeheight{\gamma}{\Lip_\A,\Lip_\B} + r\bridgereach{\gamma}{\Lip_\A,\Lip_\B} + \|a\|_\A\\
&\leq 2r\bridgelength{\gamma}{\Lip_\A,\Lip_\B} + \|a\|_\A\text{.}
\end{split}
\end{equation*}
Thus, since $b\in\sa{\B}$, 
 we have:
\begin{equation*}
\|b\|_\B = \sup \{ |\varphi(b)| : \varphi\in\StateSpace(\B) \} \leq 2r\bridgelength{\gamma}{\Lip_\A,\Lip_\B} + \|a\|_\A\text{,}
\end{equation*}
as desired. 
\end{proof}

We may now use our fundamental Proposition (\ref{fundamental-prop}) to prove the key property of target sets:

\begin{proposition}\label{diameter-prop}
Let $(\A,\Lip_\A)$ and $(\B,\Lip_\B)$ be two {\qcms s} with $\Lip_\B$ lower semi-cont\-inuous, and $\gamma \in\bridgeset{\A}{\B}$ be a bridge from $\A$ to $\B$. 
\begin{enumerate}
\item For all $a,a' \in \dom{\Lip_\A}$ and for $r \geq \Lip_\A(a)$ and $r' \geq \Lip_\A(a')$, if $b\in\targetsetbridge{\gamma}{a}{r}$ and $b'\in\targetsetbridge{\gamma}{a'}{r'}$ and $t\in \R$, then:
\begin{equation*}
b+tb' \in \targetsetbridge{\gamma}{a+ta'}{r+|t|r'}\text{.}
\end{equation*}
\item For all $a,a' \in \dom{\Lip_\A}$ and for $r \geq \max\{\Lip_\A(a),\Lip_\A(a')\}$, we have:
\begin{equation}\label{base-ineq}
\sup \{ \|b-b'\|_\B : b\in \targetsetbridge{\gamma}{a}{r}, b'\in\targetsetbridge{\gamma}{a}{r} \} \leq 4r \bridgelength{\gamma}{\Lip_\A,\Lip_\B} + \|a-a'\|_\A\text{,}
\end{equation}
and thus in particular, for all $r \geq \Lip_\A(a)$:
\begin{equation}\label{diameter-ineq}
\diam{\targetsetbridge{\gamma}{a}{r}}{\|\cdot\|_\B} \leq 4r \bridgelength{\gamma}{\Lip_\A,\Lip_\B} \text{.}
\end{equation}
\end{enumerate}
\end{proposition}

\begin{proof}
Let $a,a'\in\dom{\Lip_\A}$ and $t\in\R$. Then:
\begin{equation*}
\Lip_\A(a+ta') \leq \Lip_\A(a)+|t|\Lip_\A(a')< \infty\text{.}
\end{equation*}

Let $r \geq \Lip_\A(a)$ and $r' \geq \Lip_\A(a')$. 

Let $b\in\targetsetbridge{\gamma}{a}{r}$ and $b'\in\targetsetbridge{\gamma}{a'}{r'}$. We thus have:
\begin{equation*}
\Lip_\B(b+t b')\leq \Lip_\B(b)+|t|\Lip_\B(b') \leq r+|t|r'
\end{equation*}
and
\begin{multline}
\|\pi_\A(a+t a')\omega-\omega\pi_\B(b + tb')\|_\D \leq \bridgenorm{\gamma}{a,b} + |t|\bridgenorm{\gamma}{a',b'} \\ \leq \left(r+|t|r'\right) \bridgereach{\gamma}{\Lip_\A,\Lip_\B}\text{.}
\end{multline}
Hence by Definition (\ref{target-set-def}), we have $b+tb' \in \targetsetbridge{\gamma}{a+ta'}{r+|t|r}$. This proves (1).

In particular, let $r \geq \max\{\Lip_\A(a),\Lip_\A(a')\}$. Then $b-b' \in \targetsetbridge{\gamma}{a-a'}{2r}$, and thus by Proposition (\ref{fundamental-prop}), we have:
\begin{equation*}
\|b-b'\|_\B \leq 4r\bridgelength{\gamma}{\Lip_\A,\Lip_\B} + \|a-a'\|_\A\text{,}
\end{equation*}
as desired.

Inequality (\ref{diameter-ineq}) is proven by taking $a=a'$ in Inequality (\ref{base-ineq}).
\end{proof}


The quantum propinquity is constructed using treks rather than bridges, so we will extend Propositions (\ref{fundamental-prop}) and (\ref{diameter-prop}) to treks, using the following natural notion.

\begin{definition}\label{itinerary-def}
An \emph{itinerary} for a trek $\Gamma = (\A_j,\Lip_j,\gamma_j,\A_{j+1},\Lip_{j+1} : j=1,\ldots,n)$ from a {\qcms} $(\A,\Lip_\A)$ to a {\qcms} $(\B,\Lip_\B)$ is a $(n+1)$-tuple $\eta = (\eta_1,\ldots,\eta_{n+1})$ with $\eta_j \in \sa{\A_j}$ for all $j\in\{1,\ldots,n+1\}$. The itinerary $\eta$ \emph{starts} at $\eta_1$ and \emph{ends} at $\eta_{n+1}$, or equivalent is an itinerary \emph{from $\eta_1$ to $\eta_{n+1}$}.
\end{definition}

\begin{notation}\label{itineraries-not}
Let $\Gamma$ be a trek from a {\qcms} $(\A,\Lip_\A)$ to a {\qcms} $(\B,\Lip_\B)$ and let $a\in\sa{\A}$ and $b\in\sa{\B}$. The set of all itineraries from $a$ to $b$ is denoted by $\itineraries{\Gamma}{a}{b}$. The set of all itineraries starting at $a$ is denoted by $\itineraries{\Gamma}{a}{\cdot}$ while the set of all itineraries ending at $b$ is denoted by $\itineraries{\Gamma}{\cdot}{b}$.
\end{notation}

\begin{definition}\label{lambda-itinerary-def}
Let $(\A,\Lip_\A)$ and $(\B,\Lip_\B)$ be two {\qcms s} in some nonempty class $\mathcal{C}$ of {\qcms s}. Let:
\begin{equation*}
\Gamma \in\startrekset{\mathcal{C}}{\A,\Lip_\A}{\B,\Lip_\B}
\end{equation*}
be a trek from $(\A,\Lip_\A)$ to $(\B,\Lip_\B)$. For any $a\in \dom{\Lip_\A}$, $b\in\sa{\B}$ and any $r \geq \Lip_\A(a)$, the \emph{set $\lambdaitineraries{\Gamma}{a}{b}{r}$ of $r$-itineraries for $\Gamma$ starting at $a$ and ending at $b$} is defined as:
\begin{equation*}
 \left\{ \eta\in\itineraries{\Gamma}{a}{b} : \forall j\in\{1,\ldots,n\} \quad \eta_{j+1} \in \targetsetbridge{\gamma_j}{\eta_j}{r} \right\}\text{.}
\end{equation*}
\end{definition}

\begin{remark}
In general, the set of $r$-itineraries for some trek, from a point to another point, may be empty. For instance, in a {\Lqcms} whose underlying space is at least of dimension 2, there is no $0$-itinerary from $0$ to an element $a$ which is not a scalar multiple of the identity; more generally, a necessary condition for a set of $r$-itineraries from some $a$ to some $b$ to be nonempty is that $\Lip_\B(b)\leq r$.
\end{remark}

\begin{remark}
By Definition (\ref{target-set-def}), the set of $r$-itineraries for $r > \Lip_\A(a)$ is given by:
\begin{equation*}
\left\{ \eta \in \itineraries{\Gamma}{a}{b} \middle| \forall j \in \{1,\ldots,n\}\quad
\begin{aligned}
&\,\bridgenorm{\gamma_j}{\eta_j,\eta_{j+1}} \leq r\bridgereach{\gamma_j}{\Lip_j,\Lip_{j+1}}\text{,}\\
&\,\Lip_{j+1}(\eta_{j+1})\leq r \text{.}
\end{aligned}
 \right\} 
\end{equation*}
\end{remark}

We can now generalize the notion of a target set from bridges to treks, and deduce their fundamental properties by extending by induction the properties of target sets for bridges.

\begin{definition}\label{target-trek-def}
Let $(\A,\Lip_\A)$ and $(\B,\Lip_\B)$ be two {\qcms s} in some nonempty class $\mathcal{C}$ of {\qcms s}. Let:
\begin{equation*}
\Gamma \in\startrekset{\mathcal{C}}{\A,\Lip_\A}{\B,\Lip_\B}
\end{equation*}
be a $\mathcal{C}$-trek from $(\A,\Lip_\A)$ to $(\B,\Lip_\B)$. For any $a\in\dom{\Lip_\A}$ and $r \geq \Lip_\A(a)$, the \emph{target set for $a$ of spread $r$} is given as:
\begin{equation*}
\targetsettrek{\Gamma}{a}{r} = \left\{ b \in \B : \lambdaitineraries{\Gamma}{a}{b}{r} \not=\emptyset \right\}\text{.}
\end{equation*}
\end{definition}

\begin{proposition}\label{fundamental-trek-proposition}
Let $(\A,\Lip_\A)$ and $(\B,\Lip_\B)$ be two {\qcms s} in a nonempty class $\mathcal{C}$ of {\qcms s} whose Lip-norms are all lower semi-cont\-inuous. Let:
\begin{equation*}
\Gamma \in\startrekset{\mathcal{C}}{\A,\Lip_\A}{\B,\Lip_\B}
\end{equation*}
be a $\mathcal{C}$-trek from $(\A,\Lip_\A)$ to $(\B,\Lip_\B)$. Then, the following statements hold true.
\begin{enumerate}
\item For all $a\in \dom{\Lip_\A}$ and $r \geq \Lip_\A(a)$, we have $\targetsettrek{\Gamma}{a}{r} \not=\emptyset$.
\item For all $a\in \dom{\Lip_\A}$ and $r \geq \Lip_\A(a)$, and for all $b \in \targetsettrek{\Gamma}{a}{r}$, we have:
\begin{equation}\label{fundamental-trek-prop-basic-eq}
\|b\|_\B \leq 2r \treklength{\Gamma} + \|a\|_\A \text{.}
\end{equation}
\item For all $a,a' \in\dom{\Lip_\A}$ and for all $r \geq \Lip_\A(a)$ and $r' \geq \Lip_\A(a')$, if $b\in\targetsettrek{\Gamma}{a}{r}$ and $b'\in\targetsettrek{\Gamma}{a'}{r'}$, then for all $t\in\R$ we have:
\begin{equation}\label{fundamental-trek-prop-linear-eq}
b+tb' \in \targetsettrek{\Gamma}{a+ta'}{r+|t|r'} \text{.}
\end{equation}
\item For all $a,a'\in \dom{\Lip_\A}$ and $r \geq \max\{\Lip_\A(a),\Lip_\A(a')\}$ we have:
\begin{equation}\label{fundamental-trek-prop-distance-eq}
\sup \left\{ \|b-b'\|_\B : b\in\targetsettrek{\Gamma}{a}{r}, b'\in\targetsettrek{\Gamma}{a'}{r} \right\} \leq 4r\treklength{\Gamma} + \|a-a'\|_\A\text{,}
\end{equation}
and thus in particular:
\begin{equation}\label{fundamental-trek-prop-diameter-eq}
\diam{\targetsettrek{\Gamma}{a}{r}}{\|\cdot\|_\B} \leq 4r\treklength{\Gamma} \text{.}
\end{equation}
\end{enumerate}
\end{proposition}

\begin{proof}
We write $\Gamma = \left(\A_j,\Lip_j,\gamma_j,\A_{j+1},\Lip_{j+1} : j=1,\ldots,n\right)$, where, by Definition (\ref{trek-def}), we have:
\begin{itemize}
\item $(\A,\Lip_\A) = (\A_1,\Lip_1)$ and $(\B,\Lip_\B) = (\A_{n+1},\Lip_{n+1}) $,
\item for all $j\in\{1,\ldots,n+1\}$, the pair $(\A_j,\Lip_j)$ is a {\qcms} in $\mathcal{C}$, and in particular $\Lip_j$ is lower semi-con\-tinuous,
\item for all $j\in\{1,\ldots,n+1\}$, we have $\gamma_j\in\bridgeset{\A_j}{\A_{j+1}}$.
\end{itemize}

Moreover, for all $j \in \{1,\ldots,n\}$ we denote $\bridgelength{\gamma_j}{\Lip_j,\Lip_{j+1}}$ by $\lambda_j$. 

\bigskip

By Lemma (\ref{nonempty-target-lem}), the set $\targetsetbridge{\gamma_1}{a}{r}$ is not empty, and if $a_{j+1} \in \targetsetbridge{\gamma_j}{a_j}{r}$ for some $j\in \{1,\ldots,n-1\}$, then, since $\Lip_{j+1}(a_{j+1})\leq r$, we conclude again by Lemma (\ref{nonempty-target-lem}) that $\targetsetbridge{\gamma_{j+1}}{a_{j+1}}{r}$ is not empty. 
Therefore, by induction and by Definition (\ref{lambda-itinerary-def}), the set $\lambdaitineraries{\Gamma}{a}{\cdot}{r}$ is not empty, and thus by Definition (\ref{target-trek-def}), the set $\targetsettrek{\Gamma}{a}{r}$ is not empty. This proves (1).

Fix now $a \in \dom{\Lip_\A}$ and let $r \geq \Lip_\A(a)$. Let $b \in \targetsettrek{\Gamma}{a}{r}$. By Definition (\ref{target-trek-def}), there exists $\eta\in\lambdaitineraries{\Gamma}{a}{b}{r}$. By Definition (\ref{lambda-itinerary-def}), we have $\eta_{j+1}\in\targetsetbridge{\gamma_j}{\eta_j}{r}$ for all $j\in\{1,\ldots,n\}$. By Proposition (\ref{fundamental-prop}), we thus have:
\begin{equation*}
\forall j\in\{1,\ldots,n\}\quad \|\eta_{j+1}\|_{\A_{j+1}} \leq \|\eta_j\|_{\A_j} + 2r\lambda_j \text{.}
\end{equation*}
Hence, by induction, we conclude that for all $j\in\{2,\ldots,n+1\}$:
\begin{equation*}
\|\eta_j\|_{\A_{j}} \leq \|a\|_\A + 2r\sum_{k=1}^{j-1}\lambda_k \text{,} 
\end{equation*}
and thus:
\begin{equation*}
\|b\|_\B = \|\eta_{n+1}\|_{\A_{n+1}} \leq \|a\|_\A + 2r \treklength{\Gamma}\text{.}
\end{equation*}
This proves (2).

Let now $a,a' \in \dom{\Lip_\A}$ and choose $r \geq\Lip_\A(a)$, $r'\geq \Lip_\A(a')$ and $t\in\R$. Let $b \in \targetsettrek{\Gamma}{a}{r}$ and $b' \in \targetsettrek{\Gamma}{a'}{r'}$.

Let $(\eta_j)_{j=1}^{n+1}\in\lambdaitineraries{\Gamma}{a}{b}{r}$ and $(\zeta_j)_{j=1}^{n+1}\in\lambdaitineraries{\Gamma}{a'}{b'}{r'}$. For each $j \in \{1,\ldots,n\}$, we have $\eta_{j+1} \in \targetsetbridge{\gamma_j}{\eta_j}{r}$ and $\zeta_{j+1}\in\targetsetbridge{\gamma_j}{\zeta_j}{r'}$, and thus by Proposition (\ref{diameter-prop}), we have:
\begin{equation*}
\forall j \in \{1,\ldots,n\} \quad \eta_{j+1} + t\zeta_{j+1} \in \targetsetbridge{\gamma_j}{\eta_j+t\zeta_j}{r+|t|r'}
\end{equation*}
which implies, in turn, that $(\eta_j+t \zeta_j)_{j=1}^{n+1} \in \lambdaitineraries{\Gamma}{a+ta'}{b+tb'}{r+|t|r'}$. Thus by Definition (\ref{target-trek-def}), we have $b + tb' \in \targetsettrek{\Gamma}{a+ta'}{r+|t|r'}$. This proves (3).

Similarly, let $a,a' \in \dom{\Lip_\A}$ and $r \geq \max\{\Lip_\A(a),\Lip_\A(a')\}$. Let $b\in \targetsettrek{\Gamma}{a}{r}$ and $b'\in\targetsettrek{\Gamma}{a'}{r}$. Then by Definition (\ref{target-trek-def}), there exist two $r$-itineraries $\eta\in\lambdaitineraries{\Gamma}{a}{b}{r}$ and $\eta'\in \lambdaitineraries{\Gamma}{a'}{b'}{r}$. By Proposition (\ref{diameter-prop}), for all $j\in \{1,\ldots,n\}$:
\begin{equation*}
\|\eta_{j+1}-\eta'_{j+1}\|_{\A_{j+1}}\leq 4r\lambda_j + \|\eta_j-\eta'_j\|_{\A_j}\text{.}
\end{equation*}
By induction, we thus establish:
\begin{equation*}
\|b-b'\|_\B \leq 4r\sum_{j=1}^n \lambda_j + \|a-a'\|_\A = 4r\treklength{\Gamma} + \|a-a'\|_\A\text{.}
\end{equation*}
which proves (4), as we note that the diameter inequality follows by taking $a=a'$.
\end{proof}

Thus far, we have not used the Leibniz property in our work on bridges and treks. Our main reason to work with the class $\LCQMS$ of {\Lqcms s} is indeed the next Proposition, where we see how the Leibniz property allows us to relate the multiplication and our target sets. This relation will be the tool we use to prove that the maps we shall construct in Theorem (\ref{main}) between {\Lqcms s} at quantum propinquity zero is indeed a multiplicative map.

\begin{proposition}\label{product-prop}
Let $(\A,\Lip_\A)$ and $(\B,\Lip_\B)$ be two {\Lqcms s}. Let $\Gamma \in\startrekset{\LCQMS}{\A,\Lip_\A}{\B,\Lip_\B}$ be a $\LCQMS$-trek from $(\A,\Lip_\A)$ to $(\B,\Lip_\B)$. Then, for all $a,a'\in\dom{\Lip_\A}$ and for all $r \geq \max\{\Lip_\A(a),\Lip_\A(a')\}$, if $b\in\targetsettrek{\Gamma}{a}{r}$ and $b'\in\targetsettrek{\Gamma}{a'}{r}$ then:
\begin{equation}
\Jordan{b}{b'} \in \targetsettrek{\Gamma}{\Jordan{a}{a'}}{r\left(4r\treklength{\Gamma} + \|a\|_\A + \|a'\|_\A\right)}\text{,}
\end{equation}
and
\begin{equation}
\Lie{b}{b'} \in \targetsettrek{\Gamma}{\Lie{a}{a'}}{r\left(4r\treklength{\Gamma} + \|a\|_\A + \|a'\|_\A\right)}\text{.}
\end{equation}
\end{proposition}

\begin{proof}
Fix $a,a'\in\dom{\Lip_\A}$, and $r  \geq \max\{\Lip_\A(a),\Lip_\A(a')\}$. By Proposition (\ref{fundamental-trek-proposition}), the sets $\targetsettrek{\Gamma}{a}{r}$ and $\targetsettrek{\Gamma}{a'}{r}$ are not empty since $\Lip_\B$ is lower semi-cont\-inuous as $(\B,\Lip_\B)$ is a {\Lqcms}. Let $b \in \targetsettrek{\Gamma}{a}{r}$, so that by Definition (\ref{target-trek-def}), there exists $\eta\in\lambdaitineraries{\Gamma}{a}{b}{r}$. Similarly, let $b' \in \targetsettrek{\Gamma}{a'}{r}$ and $\eta'\in\lambdaitineraries{\Gamma}{a'}{b'}{r}$. 

We write $\Gamma = \left(\A_j,\Lip_j,\gamma_j,\A_{j+1},\Lip_{j+1} : j=1,\ldots,n\right)$ with the same notations as in the proof of Proposition (\ref{fundamental-trek-proposition}), where $\mathcal{C}$ is understood to be the class {$\LCQMS$} of all {\Lqcms s}. Moreover, for each $j\in\{1,\ldots,n\}$, we write $\gamma_j = (\D_j,\omega_j,\pi_j,\rho_j)$ with $\D_j$ a unital C*-algebra, $\omega_j \in \D_j$ an element with nonempty 1-level, and $\pi_j$ and $\rho_j$ unital *-monomorphism of, respectively, $\A_j$ and $\A_{j+1}$ in $\D_j$. We denote by $\lambda_j$  the bridge length $\bridgelength{\gamma_j}{\Lip_j,\Lip_{j+1}}$ and $\varrho_j$ the reach $\bridgereach{\gamma_j}{\Lip_j,\Lip_{j+1}}$ of the bridge $\gamma_j$ for all $j\in\{1,\ldots,n\}$.

With these notations, we have the following computation (see Equation (\ref{bridgenorm-eq})):
\begin{multline}\label{product-prop-eq1}
\|\pi_{j}(\eta_j\eta'_j)\omega_{j} - \omega_{j}\rho_{j}(\eta_{j+1}\eta'_{j+1})\|_{\D_{j}}\\
\leq \|\eta_j\|_{\A_{j}}\|\pi_{j}(\eta_j')\omega_{j}-\omega_{j}\rho_{j}(\eta'_{j+1})\|_{\D_{j}} \\ + \|\pi_{j}(\eta_j)\omega_{j} - \omega_{j}\rho_{j}(\eta_{j+1})\|_{\D_{j}}\|\eta'_{j+1}\|_{\A_{j+1}}\text{.}
\end{multline}

Now, by Proposition (\ref{fundamental-prop}), we have for all $j\in\{1,\ldots,n+1\}$:
\begin{equation}\label{product-prop-eq2}
\|\eta_j\|_{\A_j} \leq \|a\|_\A + 2r\treklength{\Gamma} \text{ and }\|\eta'_j\|_{\A_j}\leq \|a'\|_{\A} + 2r\treklength{\Gamma} \text{.}
\end{equation}
Hence, bringing Inequalities (\ref{product-prop-eq1}) and (\ref{product-prop-eq2}) together, we obtain:
\begin{equation}\label{product-prop-eq2b}
\begin{split}
\|\pi_{j}(\eta_j\eta'_j)\omega_{j} - \omega_{j}\rho_{j}(\eta_{j+1}\eta'_{j+1})\|_{\D_{j}} &\leq \left(\|a\|_\A+ 2r \treklength{\Gamma}\right)r\lambda_{j} + \left(\|a'\|_\A+ 2r \treklength{\Gamma} \right)r\lambda_{j}\\
&= r\lambda_{j}\left(4r\treklength{\Gamma} + \|a\|_\A + \|a'\|_\A \right)\text{.}
\end{split}
\end{equation}

Hence, noting that we can switch the role of $\eta$ and $\eta'$ in Equation (\ref{product-prop-eq2b}, we have:
\begin{equation}\label{product-prop-eq3}
\begin{split}
\bridgenorm{\gamma_{j}}{\Jordan{\eta_j}{\eta'_j}, \Jordan{\eta_{j+1}}{\eta'_{j+1}}} &= \left\|\pi_{j}\left(\Jordan{\eta_j}{\eta_j'}\right)\omega_{j} - \omega_{j}\rho_{j}\left(\Jordan{\eta_{j+1}}{\eta_{j+1}'}\right)\right\|_{\D_{j}}\\
&\leq\varrho_{j}r\left(4r\treklength{\Gamma} + \|a\|_\A + \|a'\|_\A\right)\text{.}
\end{split}
\end{equation}

On the other hand, since for all $j\in\{1,\ldots,n\}$ the Lip-norm $\Lip_{j+1}$ is a {\JLL} Lip-norm , we have, using Proposition (\ref{fundamental-prop}) again:
\begin{equation}\label{product-prop-eq4}
\begin{split}
\Lip_{j+1}\left(\Jordan{\eta_{j+1}}{\eta'_{j+1}}\right) &\leq \|\eta_{j+1}\|_{\A_{j+1}}\Lip_{j+1}(\eta'_{j+1}) + \|\eta'_{j+1}\|_{\A_{j+1}}\Lip_{j+1}(\eta_{j+1})\\
& \leq \left(\|a\|_\A + 2r\treklength{\Gamma}\right)r + (\|a'\|_\A + 2r\treklength{\Gamma})r\\
&= r \left( 4r\treklength{\Gamma} + \|a\|_\A +\|a'\|_\A \right)\text{.}
\end{split}
\end{equation}

We thus have proven, by Inequalities (\ref{product-prop-eq3}) and (\ref{product-prop-eq4}), that:
\begin{equation*}
(\Jordan{\eta_j}{\eta'_j}: j =1,\ldots,n+1)\in\lambdaitineraries{\Gamma}{\Jordan{a}{a'}}{\Jordan{b}{b'}}{r\left(4r\treklength{\Gamma} + \|a\|_\A + \|a'\|_\A\right)}
\end{equation*}
noting that $\Lip_\A(\Jordan{a}{a'}) \leq (\|a\|_\A + \|a'\|_\A)r$. Thus by Definition (\ref{target-trek-def}):
\begin{equation}\label{j-product-eq}
\Jordan{b}{b'} \in \targetsettrek{\Gamma}{\Jordan{a}{a'}}{r\left(4r\treklength{\Gamma} + \|a\|_\A + \|a'\|_\A\right)}\text{.}
\end{equation}

A similar computation would prove that:
\begin{equation}\label{l-product-eq}
\Lie{b}{b'} \in \targetsettrek{\Gamma}{\Lie{a}{a'}}{r\left(4r\treklength{\Gamma} + \|a\|_\A + \|a'\|_\A\right)}\text{.}
\end{equation}
This completes the proof of our proposition.
\end{proof}

The fundamental property of our quantum propinquity is that it satisfies the axiom of coincidence on the quantum isometry classes of {\Lqcms s}. This constitutes the main theorem of this section, which we now prove.


\begin{theorem}\label{main}
Let $(\A,\Lip_\A)$ and $(\B,\Lip_\B)$ be two {\Lqcms s}. If
\begin{equation*}
\propinquity((\A,\Lip_\A),(\B,\Lip_\B)) = 0\text{,}
\end{equation*}
then there exists a *-isomorphism $h : \A \rightarrow \B$ such that $\Lip_\B\circ h = \Lip_\A$.
\end{theorem}

\begin{proof}
For every $\varepsilon > 0$, by Definition (\ref{propinquity-def}), there exists an $\LCQMS$-trek $\Gamma_\varepsilon$ from $(\A,\Lip_\A)$ to $(\B,\Lip_\B)$ such that:
\begin{equation*}
\treklength{\Gamma_\varepsilon} \leq \varepsilon \text{.}
\end{equation*}

For each $\varepsilon > 0$, we fix, once and for all, such a $\LCQMS$-trek $\Gamma_\varepsilon$. We present our proof as a succession of claims, followed by their own proof, to expose the main structure of our argument.

\emph{As our first step, given $a\in\dom{\Lip_\A}$, we show how to extract a potential image for $a$ in $\B$ from our target sets $\targetsettrek{\cdot}{a}{\cdot}$, using a compactness argument.}

\begin{claim}\label{main-claim-singleton}
For any $a\in\dom{\Lip_\A}$, any $r\geq \Lip_\A(a)$ and for any strictly increasing function $f: \N \rightarrow \N_+$, there exists a strictly increasing function $g : \N \rightarrow \N$ such that the sequence
\begin{equation*}
\left(\targetsettrek{\Gamma_{(f\circ g(n))^{-1}}}{a}{r}\right)_{n\in\N}
\end{equation*}
converges to a singleton in the Haus\-dorff distance $\Haus{\B}$ induced by $\|\cdot\|_\B$ on the bounded subsets of $\B$. Moreover, if we denote this singleton by $\{b\}$ then, for any sequence $(b_m)_{m\in\N}$ in $\B$, we have the following property:
\begin{equation*}\label{unique-limit-eq0}
\left(\forall m\in\N\quad b_m \in \targetsettrek{\Gamma_{(f\circ g(m))^{-1}}}{a}{r} \right)\implies \lim_{m\rightarrow\infty} b_m = b\text{.}
\end{equation*}
\end{claim}

Fix $a \in \dom{\Lip_\A}$, and let $r \geq \Lip_\A(a)$. Let $f : \N\rightarrow\N_+$ be any strictly increasing function. Let $\varepsilon \in (0,1)$. By Inequality (\ref{fundamental-trek-prop-basic-eq}) of Proposition (\ref{fundamental-trek-proposition}), the set $\targetsettrek{\Gamma_\varepsilon}{a}{r}$ is a subset of the set:
\begin{equation*}
\mathfrak{c}(a,r) = \{ b \in \sa{\B} : \Lip_\B(b) \leq r \text{ and }\|b\|_\B \leq 2r + \|a\|_\A\}\text{.}
\end{equation*}
By \cite{Rieffel99} and Lemma (\ref{az-lemma}), as $\Lip_\B$ is a closed Lip-norm, the set $\mathfrak{c}(r,a)$ is compact for the norm topology of $\sa{\B}$. Moreover, the set $\targetsettrek{\Gamma_\varepsilon}{a}{r}$ is itself a norm compact subset of $\mathfrak{c}(a,r)$. The metric space consisting of the compact subsets of $\mathfrak{c}(a,r)$ endowed with the Haus\-dorff distance $\Haus{\B}$ induced by the norm $\|\cdot\|_\B$ of $\B$ is itself compact since $\mathfrak{c}(a,r)$ is compact. Thus, there exists a strictly increasing function $g :\N \rightarrow \N$ such that $\left(\targetsettrek{\Gamma_{(f\circ g)(m)^{-1}}}{a}{r}\right)_{m\in\N}$ converges for $\Haus{\B}$ to some compact subset of $\mathfrak{c}(r,a)$. We denote the (closed) limit of $\left(\targetsettrek{\Gamma_{(f\circ g(n))^{-1}}}{a}{r}\right)_{m\in\N}$ by $\mathfrak{P}_{f\circ g}(a,r)$.

We now prove that $\mathfrak{P}_{f\circ g}(a,r)$ is a singleton. By Inequality (\ref{fundamental-trek-prop-diameter-eq}) of Proposition (\ref{fundamental-trek-proposition}), we conclude that:
\begin{equation}\label{zero-diameter-eq}
\diam{\mathfrak{P}_{f\circ g}(a,r)}{\|\cdot\|_\B} = 0\text{.}
\end{equation}
For each $m\in\N$, let $b_m \in \targetsettrek{\Gamma_{(f\circ g(m))^{-1}}}{a}{r}$ be chosen arbitrarily (which is possible since $\targetsettrek{\Gamma_{(f\circ g(m))^{-1}}}{a}{r}$ is not empty by Proposition (\ref{fundamental-trek-proposition})). Then Proposition (\ref{fundamental-trek-proposition}) proves that $(b_m)_{m\in\N}$ is a Cauchy sequence in $\sa{\B}$ for the norm of $\B$. Indeed, let $\varepsilon > 0$. Then by Inequality (\ref{fundamental-trek-prop-diameter-eq}) of Proposition (\ref{fundamental-trek-proposition}), there exists $M \in \N$ such that for all $m\geq M$, we have:
\begin{equation}\label{main-cauchy-eq0}
\diam{\targetsettrek{\Gamma_{(f\circ g(m))^{-1}}}{a}{r}}{\|\cdot\|_\B}<\frac{1}{2}\varepsilon\text{.}
\end{equation}
There also exists $M' \in\N$ such that, for all $p,q \geq M'$, we have:
\begin{equation}\label{main-cauchy-eq1}
\Haus{\B}\left(\targetsettrek{\Gamma_{(f\circ g(p))^{-1}}}{a}{r},\targetsettrek{\Gamma_{(f\circ g(q))^{-1}}}{a}{r}\right) < \frac{1}{2}\varepsilon\text{,}
\end{equation}
as $\left(\targetsettrek{\Gamma_{(f\circ g(m))^{-1}}}{a}{r}\right)_{m\in\N}$ converges, and hence is Cauchy, for $\Haus{\B}$.

Let $p,q \geq \max\{M,M'\}$. By Inequality (\ref{main-cauchy-eq1}), there exists:
\begin{equation*}
c_q\in\targetsettrek{\Gamma_{(f\circ g(q))^{-1}}}{a}{r}
\end{equation*}
such that $\|c_q-b_p\|_\B~\leq~\frac{1}{2}\varepsilon$. By Inequality (\ref{main-cauchy-eq0}), we also have $\|b_q-c_q\|_\B~\leq~\frac{1}{2}\varepsilon$. Hence $\|b_p-b_q\|_\B~\leq~\varepsilon$. Thus $(b_m)_{m\in\N}$ is indeed Cauchy in $\sa{\B}$. Since $\sa{\B}$ is complete, the sequence $(b_m)_{m\in\N}$ converges. Let us temporarily denote its limit by $b$. 

It is easy to check that $b\in\mathfrak{P}_{f\circ g}(a,r)$: for any $\varepsilon > 0$, there exists $M \in \N$ such that for all $m\geq M$, the diameter of $\targetsettrek{\Gamma_{(f\circ g(m))^{-1}}}{a}{r}$ is less than $\frac{1}{2}\varepsilon$, and there exists $M'\in\N$ such that for all $m\geq M'$, we have $\|b_m - b\|_\B<\frac{1}{2}\varepsilon$, so for all $m\geq \max\{M,M'\}$, the set $\targetsettrek{\Gamma_{(f\circ g(m))^{-1}}}{a}{r}$ lies within the open $\varepsilon$-neighborhood of $b$ for $\|\cdot\|_\B$. So $\{b\} \in \mathfrak{P}_{f\circ g}(a,r)$ by definition. Hence $\mathfrak{P}_{f\circ g}(a,r) = \{ b \}$ as desired.

\bigskip
\emph{Our next step is to choose images for all elements in $\dom{\Lip_\A}$ in a coherent fashion, based upon Claim (\ref{main-claim-singleton}) and a diagonal argument:}

\begin{claim}\label{main-claim-map}
There exists an increasing function $f : \N \rightarrow \N_+$ and a function $h : \dom{\Lip_\A} \rightarrow \dom{\Lip_\B}$ such that, for any $a\in \dom{\Lip_\A}$ and for any $r\geq \Lip_\A(a)$, we have:
\begin{equation*}
\lim_{n\rightarrow\infty} \Haus{\B}\left(\targetsettrek{\Gamma_{f(n)^{-1}}}{a}{r}, \{ h(a) \}\right) = 0 \text{.}
\end{equation*}
Moreover, for any $a\in\dom{\Lip_\A}$, any sequence $(b_m)_{m\in\N}$ in $\B$ and any $r\geq \Lip_\A(a)$, we have:
\begin{equation*}
\left(\forall m\in\N\quad b_m \in \targetsettrek{\Gamma_{f(m)^{-1}}}{a}{r} \right)\implies \lim_{m\rightarrow\infty} b_m = h(a)\text{.}
\end{equation*}
\end{claim}

Let $\mathfrak{a}=\{a_k : k\in\N\}$ be a countable, dense subset of $\dom{\Lip_\A}$. To ease notations, let $l_k = \Lip_\A(a_k)$ for all $k\in\N$. By Claim (\ref{main-claim-singleton}), for each $k\in\N$, there exists $g_k~:~\N~\rightarrow~\N_+$ strictly increasing, such that $\left(\targetsettrek{\Gamma_{g_k(m)^{-1}}}{a_k}{l_k}\right)_{n\in\N}$ converges to some singleton $\{ h(a_k) \}$ in $\B$ for $\Haus{\B}$. One then easily checks that, setting:
\begin{equation*}
f: n \in \N \longmapsto g_0 \circ g_1 \circ\ldots g_n(n)\text{,}
\end{equation*}
the sequence $\left(\targetsettrek{\Gamma_{f(m)^{-1}}}{a_k}{l_k}\right)_{n\in\N}$ converges in $\Haus{\B}$ to $\{h(a_k)\}$ for all $k\in\N$. Moreover, by Claim (\ref{main-claim-singleton}), for all sequences $(b_m)_{m\in\N}\in\sa{\B}$ and $k\in\N$, we have:
\begin{equation}\label{unique-limit-eq0}
\left(\forall m\in\N\quad b_m \in \targetsettrek{\Gamma_{f(m)^{-1}}}{a_k}{l_k} \right)\implies \lim_{m\rightarrow\infty} b_m = h(a)\text{.}
\end{equation}

\bigskip
\emph{We now show that we can relax the choice of the constant $r$ and obtain the same limit.}

Let $a\in\mathfrak{a}$ and let $r \geq \Lip_\A(a)$. Let $q:\N\rightarrow\N$ be a strictly increasing function. By Claim (\ref{main-claim-singleton}), there exists a strictly increasing function $g : \N\rightarrow\N$ such that $\left(\targetsettrek{\Gamma_{(f\circ q\circ g(m))^{-1}}}{a}{r}\right)_{m\in\N}$ converges to some singleton $\{b\}$ for $\Haus{\B}$. On the other hand, by construction, for all $m\in\N$:
\begin{equation}\label{main-claim-map-eq1}
\emptyset\not= \targetsettrek{\Gamma_{(f\circ q\circ g(m))^{-1}}}{a}{\Lip_\A(a)} \subseteq \targetsettrek{\Gamma_{(f\circ q\circ g(m))^{-1}}}{a}{r}\text{.}
\end{equation}
Let $(b_m)_{m\in\N}$ be a sequence in $\B$ such that for all $m\in\N$, we have
\begin{equation*}
b_m \in \targetsettrek{\Gamma_{(f\circ q \circ g(m))^{-1}}}{a}{\Lip_\A(a)}
\end{equation*}
(such sequence exists since by Proposition (\ref{fundamental-trek-proposition}), these target sets are not empty). By Claim (\ref{main-claim-singleton}), the sequence $(b_m)_{m\in\N}$ converges to $h(a)$. By Inclusion (\ref{main-claim-map-eq1}), $(b_m)_{m\in\N}$ satisfies the property that $b_m\in\targetsettrek{\Gamma_{(f\circ q\circ g(m))^{-1}}}{a}{r}$ for all $m\in\N$. Hence, again by Claim (\ref{main-claim-singleton}), we conclude that $b = \lim_{m\rightarrow\infty}b_m = h(a)$.

Thus, as $q$ was an arbitrary strictly increasing function, we conclude that every sub-sequence of $\left(\targetsettrek{\Gamma_{f(m)^{-1}}}{a}{r}\right)_{m\in\N}$ admits a sub-sequence converging to the set $\{h(a)\}$ in $\Haus{\B}$. Thus:
\begin{equation}\label{unique-limit-eq1}
\forall a\in\mathfrak{a}\;\forall r\geq\Lip_\A(a)\quad \lim_{m\rightarrow\infty} \Haus{\B}\left(\targetsettrek{\Gamma_{f(m)^{-1}}}{a}{r},\{h(a)\}\right) = 0\text{.}
\end{equation}
where $\Haus{\B}$ is the Haus\-dorff distance induced by $\|\cdot\|_\B$ on nonempty bounded subsets of $\B$.

\bigskip
\emph{We now extend $h$ thus defined to all of $\dom{\Lip_\A}$.}

Let $a\in\dom{\Lip_\A}$ be chosen. Let $\varepsilon > 0$. There exists $a' \in \mathfrak{a}$ with $\|a-a'\|_\A<\frac{\varepsilon}{3}$. Let $ r \geq \max\{\Lip_\A(a),\Lip_\A(a')\}$. Using Equation (\ref{unique-limit-eq1}), let $M\in\N$ such that, for all $m\geq M$, we have:
\begin{equation*}
\Haus{\B}\left(\targetsettrek{\Gamma_{f(m)^{-1}}}{a'}{r},\{h(a')\}\right) < \frac{\varepsilon}{3}\text{.}
\end{equation*}
Let $M'\in\N$ such that for all $m\geq M'$, we have:
\begin{equation*}
f(m)^{-1} < \frac{\varepsilon}{12r}\text{.}
\end{equation*}
Let $m\geq\max\{M,M'\}$. For all $b \in \targetsettrek{\Gamma_{f(m)^{-1}}}{a}{r}$ and $b' \in \targetsettrek{\Gamma_{f(m)^{-1}}}{a'}{r}$, by Inequality (\ref{fundamental-trek-prop-distance-eq}) of Proposition (\ref{fundamental-trek-proposition}), we have
\begin{equation*}
\|b-b'\|_\B \leq \|a-a'\|_\A + 4rf(m)^{-1} \leq \frac{2\varepsilon}{3}\text{.}
\end{equation*}
Thus, for $m~\geq~\max\{M,M'\}$, the target set $\targetsettrek{\Gamma_{f(m)^{-1}}}{a}{r}$ is a subset of the $\frac{2\varepsilon}{3}$-neighborhood of $\targetsettrek{\Gamma_{f(m)^{-1}}}{a'}{r}$. Hence, $\targetsettrek{\Gamma_{f(m)^{-1}}}{a}{r}$ lies within the closed ball of center $h(a')$ and radius $\varepsilon$. By the triangle inequality, we deduce that the sequence $(\targetsettrek{\Gamma_{f(m)^{-1}}}{a}{r})_{m\in\N}$ is Cauchy for the Haus\-dorff distance $\Haus{\B}$. Now, since the Haus\-dorff distance associated to a complete distance is itself complete, $(\targetsettrek{\Gamma_{f(m)^{-1}}}{a}{r})_{m\in\N}$ converges as well for $\Haus{\B}$. Let $\mathfrak{p}(a,r)$ be its limit. By Claim (\ref{main-claim-singleton}), $\mathfrak{p}(a,r)$ is a singleton consisting of the limit of any sequence $(b_m)_{m\in\N}$ chosen so that $b_m\in\mathfrak{P}(a,f^{-1}(m),r)$ for all $m\in\N$ and for all $r \geq \Lip_\A(a)$. We denote this singleton by $\{h(a)\}$.

\bigskip
We thus have defined a map $h : \dom{\Lip_\A} \rightarrow \sa{\B}$. \emph{We now enter the third  phase of our construction of $h$, when we use our understanding of the target sets to study $h$}.
\begin{claim}\label{main-claim-linear}
The map $h : \dom{\Lip_\A} \rightarrow \dom{\Lip_\B}$ is $\R$-linear and such that $\Lip_\B\circ h \leq \Lip_\A$ on $\dom{\Lip_\A}$. Moreover, $h$ has norm at most one and thus can be extended to an $\R$-linear map, still denoted by $h$, from $\sa{\A}$ to $\sa{\B}$, of norm at most one.
\end{claim}
Let $a,a' \in \dom{\Lip_\A}$ and $t \in \R$. Let $r \geq \max\{\Lip_\A(a),\Lip_\A(a')\}$. For all $m\in\N$, we let $b_m \in \targetsettrek{\Gamma_{f(m)^{-1}}}{a}{r}$ and $b_m'\in\targetsettrek{\Gamma_{f(m)^{-1}}}{a'}{r}$. By Claim (\ref{main-claim-map}), we have $\lim_{m\rightarrow\infty}b_m = h(a)$ and $\lim_{m\rightarrow\infty}b'_m = h(a')$.

Now, by Identity (\ref{fundamental-trek-prop-linear-eq}) of Proposition (\ref{fundamental-trek-proposition}):
\begin{equation*}
b_m+tb'_m \in \targetsettrek{\Gamma_{f(m)^{-1}}}{a+ta'}{(1+|t|)r}
\end{equation*}
for all $m\in\N$. Since:
\begin{equation*}
\Lip_\A(a+ta')\leq \Lip_\A(a) + |t|\Lip_\A(a') \leq (1+|t|)r
\end{equation*} 
by construction, we conclude from Claim (\ref{main-claim-map}) that:
\begin{equation*}
h(a+ta') = \lim_{m\rightarrow\infty} (b_m+tb'_m) = \lim_{m\rightarrow\infty}b_m + t\lim_{m\rightarrow \infty} b'_m = h(a) + th(a')\text{,}
\end{equation*}
as desired. \emph{Hence, $h$ is linear on $\dom{\Lip_\A}$.}

\bigskip

Now, let $a\in\dom{\Lip_\A}$ and let $(b_m)_{m\in\N}$ be a sequence in $\sa{\B}$ such that for all $m\in\N$, we have $b_m\in\targetsettrek{\Gamma_{f(m)^{-1}}}{a}{\Lip_\A(a)}$. Since $\Lip_\B$ is lower-semi-cont\-inuous by assumption, and by Definition (\ref{trek-def}), we have:
\begin{equation*}\label{h-lip-eq}
\Lip_\B(h(a)) \leq \liminf_{m\rightarrow\infty} \Lip_\B(b_m) \leq \Lip_\A(a) \text{.}
\end{equation*}
Moreover, by Inequality (\ref{fundamental-trek-prop-basic-eq}) of Proposition (\ref{fundamental-trek-proposition}), we can prove:
\begin{equation}\label{h-norm-eq}
\|h(a)\|_\B = \lim_{m\rightarrow\infty} \|b_m\|_\B \leq \lim_{m\rightarrow\infty} \left(\|a\|_\A + 2\Lip_\A(a)f(m)^{-1}\right) = \|a\|_\A\text{.}
\end{equation}

Hence, $h$ is a uniformly continuous linear map from $\dom{\Lip_\A}$, and thus it extends uniquely to a continuous linear map from $\sa{\A}$ to $\sa{\B}$, which we still denote by $h$. We note that the norm of $h$ is, at most, one.

\bigskip
\emph{We now turn to the multiplicative properties of $h$.}
\begin{claim}\label{main-claim-Jordan}
The map $h : \sa{\A}\rightarrow\sa{\B}$ is a unital Jordan-Lie algebra homomorphism of norm $1$.
\end{claim}

Let $a,a' \in \dom{\Lip_\A}$ and $r \geq \max\{\Lip_\A(a),\Lip_\A(a')\}$. Let $m\in\N$ and choose $b_m\in \targetsettrek{\Gamma_{f(m)^{-1}}}{a}{r}$ and $b_m'\in\targetsettrek{\Gamma_{f(m)^{-1}}}{a'}{r}$. By Proposition (\ref{product-prop}), we have
\begin{equation*}
\Jordan{b_m}{b_m'} \in \targetsettrek{\Gamma_{f(m)^{-1}}}{\Jordan{a}{a'}}{r\left(\|a\|_\A + \|a'\|_\A + 4rf(m)^{-1}\right)}\text{.}
\end{equation*} 
Since $\Lip_\A$ is Leibniz:
\begin{equation*}
\Lip_\A\left(\Jordan{a}{a'}\right) \leq \|a\|_\A\Lip_\A(a') + \|a'\|_\A\Lip_\A(a) \leq r(\|a\|_\A + \|a'\|_\A) \leq r'\text{.}
\end{equation*}
Thus, we conclude by Claim (\ref{main-claim-map}) that:
\begin{equation*}
h(\Jordan{a}{a'}) = \lim_{m\rightarrow\infty} \Jordan{b_m}{b_m'} = \Jordan{\lim_{m\rightarrow\infty}b_m}{\lim_{m\rightarrow\infty}b_m'} = \Jordan{h(a)}{h(a')} \text{.}
\end{equation*}

Similarly, we would prove that $\Lie{h(a)}{h(a')} = h(\Lie{a}{a'})$. 

\bigskip
From the construction of $h$, since $\unit_\B \in \targetsettrek{\Gamma_{\varepsilon}}{\unit_\A}{r}$ for any $\varepsilon>0$ and $r\geq 0$, we conclude easily that $h(\unit_\A) = \unit_\B$. Since $h$ is a linear map of norm at most $1$ by Claim (\ref{main-claim-linear}), we conclude that $h$ has norm one.

\emph{Thus $h : \sa{\A} \rightarrow\sa{\B}$ is a homomorphism for $\Jordan{\cdot}{\cdot}$ and $\Lie{\cdot}{\cdot}$ as well as a linear map of norm $1$, such that $\Lip_\B\circ h \leq \Lip_\A$ on $\dom{\Lip_\A}$.}
\bigskip
\begin{claim}\label{main-claim-extension}
The continuous unital Jordan-Lie homomorphism $h : \sa{\A}\rightarrow\sa{\B}$ extends uniquely to a unital *-homomorphism $\A\rightarrow\B$, still denoted by $h$.
\end{claim}

Let $\Re : a\in\A\mapsto\frac{a+a^\ast}{2}$ and $\Im:a\in\A \mapsto\frac{a-a^\ast}{2i}$ be, respectively, the real and imaginary part on $\A$, which are both valued in $\sa{\A}$. We set, for all $a\in\A$:
\begin{equation*}
h(a) = h(\Re(a)) + ih(\Im(a))\text{,}
\end{equation*}
and we trivially check that this definition of $h$ is consistent on $\sa{\A}$. Moreover, it is straightforward that $h$ thus extended is a continuous linear map on $\A$, with values in $\B$. Moreover, by construction, $h(a^\ast) = h(a)^\ast$ for all $a\in\A$. So $h$ is a $\ast$-preserving linear map on $\A$.

Let $a,b \in \A$ be given. Using the fact that $h$, restricted to $\sa{\A}$, is a homomorphism for the Jordan product by Claim (\ref{main-claim-Jordan}), and $h$ is linear on $\A$, we have:
\begin{equation}\label{j-product-eq1}
\begin{split}
h(\Jordan{a}{b}) &= h(\Jordan{\Re(a)}{Re(b)}) + i h(\Jordan{\Re(a)}{\Im(b)}) \\
&\quad + ih(\Jordan{\Im(a)}{\Re(b)}) - h(\Jordan{\Im(a)}{\Im(b)})\\
&= \Jordan{h(\Re(a))}{h(\Re(b))} + i\Jordan{h(\Re(a))}{h(\Im(b))} \\
&\quad + i\Jordan{h(\Im(a))}{h(\Re(b))} - \Jordan{h(\Im(a))}{h(\Im(b))}\\
&= \Jordan{h(a)}{h(b)} \text{.}
\end{split}
\end{equation}
Again, the computation carries similarly to the Lie product.

To conclude, for all $a,b\in\A$, by Equation (\ref{j-product-eq1}) and its equivalent for the Lie product, as well as by linearity of $h$:
\begin{equation*}
h(ab) = h(\Jordan{a}{b}) + ih(\Lie{a}{b}) = \Jordan{h(a)}{h(b)} + i\Lie{h(a)}{h(b)} = h(a)h(b) \text{.}
\end{equation*}

\emph{We have thus proven that $h$ is a unital *-homomorphism from $\A$ to $\B$ with $\Lip_\B\circ h \leq \Lip_\A$ on $\dom{\Lip_\A}$.} This completes our construction of $h$. Now, we conclude our proof with the following last claim.

\bigskip
\begin{claim}
The *-homomorphism $h:\A\rightarrow\B$ is a *-isomorphism onto $\B$, such that for all $a\in \dom{\Lip_\A}$ we have $\Lip_\B(h(a)) = \Lip_\A(a)$.
\end{claim}

This last step of our proof consists in constructing the inverse of $h$ using the same technique as used for the construction of $h$ itself.

We recall the following construction from Proposition (\ref{triangle-prop}). For any bridge $\gamma = (\D,\omega,\pi,\rho)$, we define $\gamma^{-1}=(\D,\omega^\ast,\rho,\pi)$. For any $\LCQMS$-trek
\begin{equation*}
\Gamma = (\A_j,\Lip_j,\gamma_j,\A_{j+1},\Lip_{j+1} : j=1,\ldots,n)\text{,}
\end{equation*}
from $(\A,\Lip_\A)$ to $(\B,\Lip_\B)$, the inverse $\LCQMS$-trek from $(\B,\Lip_\B)$ to $(\A,\Lip_\A)$ is defined as
\begin{equation*}
\Gamma^{-1} = (\A_{n+2-j},\Lip_{n+2-j},\gamma_j^{-1},\A_{n+1-j},\Lip_{n+1-j}:j=1,\ldots,n)\text{.}
\end{equation*}
By Definition (\ref{lambda-itinerary-def}), we then note that, for any $\LCQMS$-trek $\Gamma$ from $(\A,\Lip_\A)$ to $(\B,\Lip_\B)$, and for any $a\in\A$ and $b\in\B$, we have:
\begin{equation*}
(\eta_j)_{j=1}^{n+1} \in \lambdaitineraries{\Gamma}{a}{b}{r} \iff (\eta_{n+2-j})_{j=1}^{n+1} \in \lambdaitineraries{\Gamma^{-1}}{b}{a}{r} \text{.}
\end{equation*}

With these notations, we may apply the construction we carried out in Claim (\ref{main-claim-singleton})--(\ref{main-claim-extension}) to construct a *-homomorphism $s:\B\rightarrow\A$ and a strictly increasing map $g : \N\rightarrow\N$ with the following properties:
\begin{enumerate}
\item For all $b\in\dom{\Lip_\B}$, we have $\Lip_\A(s(b))\leq\Lip_\B(b)$.
\item Let $b\in\dom{\Lip_\B}$ and $r\geq\Lip_\B(b)$. If $(a_m)_{m\in\N}$ is a sequence in $\A$ such that for any $m\in\N$, we have $a_m \in \targetsettrek{\Gamma^{-1}_{(f\circ g(m))^{-1}}}{b}{r}$, then $(a_m)_{m\in\N}$ converges to $s(b)$ in $\A$.
\end{enumerate}

We pause this proof for an important remark. It is essential to note that we applied the construction carried out in Claims (\ref{main-claim-singleton})--(\ref{main-claim-extension}), not to the sequence of treks $\left(\Gamma_{n^{-1}}\right)_{n\in\N,n>0}$, but to its sub-sequence $\left(\Gamma_{f(m)^{-1}}\right)_{n\in\N}$, which was already used to construct $h$. Indeed, if we started from the original sequence of treks, then the construction of $s$ would require us to make a choice again, which would have no reason to match the choice for $h$, and thus $s\circ h$ could only be expected to be an isometric automorphism, rather than the identity. We wrote Claim (\ref{main-claim-singleton}) with an arbitrary sub-sequence of $\left(\Gamma_{n^{-1}}\right)_{n\in\N,n>0}$ precisely for this purpose.

This said, the construction of $s$ using the sub-sequence of treks $\left(\Gamma_{f(n)^{-1}}\right)_{n\in\N}$ also involves the choice of a sub-sequence of $\left(\Gamma_{f\circ g(n)^{-1}}\right)_{n\in\N}$ by compactness, which would seem to violate the uniqueness of the inverse --- as we desire to prove in this claim that $s$, indeed, in the inverse of $h$, and thus the chosen function $g$ should not matter. As the careful reader will however notice, we actually will prove in this claim that, actually, we could chose $g$ to be the identity. Indeed, if $g':\N\rightarrow\N_+$ is \emph{any} another strictly increasing function, then there exists, again by Claims (\ref{main-claim-singleton})--(\ref{main-claim-extension}), a strictly increasing function $g'':\N\rightarrow\N_+$ and a unital *-morphism $s'':\B\rightarrow\A$ satisfying the same conditions (1) and (2) as $g$ and $s$, except for $f$ being replaced by $f\circ g'$. As the proof to follow will show, the resulting *-morphism $s''$ will satisfy the same property as $s$, which will imply $s=s''$. Thus, the choice of a particular $s$ and $g$ are, in fact, illusory: all choices would lead to the same answer. We now return to our main proof.

We claim that $s$ is the inverse of $h$. To do so, let $a\in \dom{\Lip_\A}$ and $r \geq \Lip_\A(a)$. Let $(b_m)_{m\in\N} \in\sa{\B}^\N$ with $b_m\in\targetsettrek{\Gamma_{(f\circ g(m))^{-1}}}{a}{r}$ for all $m\in\N$ and note that $\lim_{m\rightarrow\infty}b_m = h(a)$ by Claim (\ref{main-claim-map}). Similarly, let $(a_m)_{m\in\N}\in\sa{\A}^\N$ such that for all $m\in\N$, we have $a_m \in \targetsettrek{\Gamma^{-1}_{(f\circ g(m))^{-1}}}{h(a)}{r}$ (note that $r \geq \Lip_\A(a)\geq \Lip_\B(h(a))$). Again, we have $\lim_{m\rightarrow\infty} a_m = s(h(a))$. The key observation here is that by definition, since $b_m \in \targetsettrek{\Gamma_{(f\circ g(m))^{-1}}}{a}{r}$, the set $\lambdaitineraries{\Gamma_{(f\circ g(m))^{-1}}}{a}{b_m}{r}$ is not empty, and thus $\lambdaitineraries{\Gamma^{-1}_{(f\circ g(m))^{-1}}}{b_m}{a}{r}$ is not empty, so $a \in \targetsettrek{\Gamma^{-1}_{(f\circ g(m))^{-1}}}{b_m}{r}$.

Let $\varepsilon > 0$. Let $M\in \N$ such that for all $m\geq M$, we have $\|b_m - h(a)\|_\B \leq \frac{1}{3}\varepsilon$. Let $M'\in\N$ be chosen so that for all $m\geq M'$ we have $\|a_m - s(h(a))\|_\A < \frac{1}{3}\varepsilon$. Let $M'' \in \N$ such that $(f\circ g)^{-1}(m)< \frac{\varepsilon}{12r}$ for all $m\geq M''$. Let $m\geq \max\{M,M',M''\}$. By Inequality (\ref{fundamental-trek-prop-distance-eq}) of Proposition (\ref{fundamental-trek-proposition}), since $a\in\targetsettrek{\Gamma^{-1}_{(f\circ g(m))^{-1}}}{b_m}{r}$ and $a_m \in \targetsettrek{\Gamma^{-1}_{(f\circ g(m))^{-1}}}{h(a)}{r}$, we have:
\begin{equation*}
\|a-a_m\|_\A \leq 4r\varepsilon + \|b_m - h(a)\|_\B \leq \frac{2}{3}\varepsilon \text{.}
\end{equation*} 
Hence:
\begin{equation*}
\|a-s(h(a))\|_\A \leq \|a-a_m\|_\A + \|a_m - s(h(a))\|_\A \leq \varepsilon\text{.}
\end{equation*}
As $\varepsilon > 0$ was arbitrary, we conclude that $a = s(h(a))$ for all $a\in \dom{\Lip_\A}$.
Now, since $\dom{\Lip_\A}$ is total in $\A$ and $h$, $s$ are *-homomorphisms, we conclude that:
\begin{equation*}
\forall a\in\A \quad a = s(h(a)) \text{.}
\end{equation*}
Similarly, we would prove that for all $b\in\B$, we have $b = h(s(b))$. Thus \emph{$h$ is a *-isomorphism from $A$ onto $\B$}. In particular, we conclude:
\begin{enumerate}
\item For all $a\in \dom{\Lip_\A}$, we have $\Lip_\A(a) \geq \Lip_\B(h(a)) \geq \Lip_\A(s(h(a))) = \Lip_\A(a)$, so $\Lip_\B\circ h = \Lip_\A$ on $\dom{\Lip_\A}$.
\item Similarly, $\Lip_\A\circ s = \Lip_\B$ and $s : \B\rightarrow\A$ is a *-isomorphism.
\end{enumerate}
This concludes the proof of our main theorem.
\end{proof}

\begin{corollary}
Let $\mathcal{C}$ be a nonempty subclass of $\LCQMS$, i.e. a class of {\Lqcms s}. Then, for any $(\A,\Lip_\A), (\B,\Lip_\B) \in \mathcal{C}$, if we have:
\begin{equation*}
\propinquity_{\mathcal{C}}((\A,\Lip_\A,\B,\Lip_\B)) = 0
\end{equation*}
then there exists a *-isomorphism $h : \A \rightarrow \B$ such that $\Lip_\B\circ h = \Lip_\A$.
\end{corollary}

\begin{proof}
By Definition (\ref{trek-def}), any $\mathcal{C}$-trek is a $\LCQMS$-trek, and thus:
\begin{equation*}
\propinquity((\A,\Lip_\A),(\B,\Lip_\B)) \leq \propinquity_{\mathcal{C}}((\A,\Lip_\A),(\B,\Lip_\B))
\end{equation*}
for any $(\A,\Lip_\A),(\B,\Lip_\B)\in \mathcal{C}$. Our corollary follows from Theorem (\ref{main}).
\end{proof}

\section{Comparison of the quantum propinquity with other metrics}

We have thus established that our quantum propinquity is a metric:

\begin{theorem}
Let $\mathcal{C}$ be a nonempty class of {\Lqcms s}. The quantum Gro\-mov-Haus\-dorff $\mathcal{C}$-Propinquity $\propinquity_{\mathcal{C}}$ is a metric on the isometric isomorphic classes of {\Lqcms s} in $\mathcal{C}$, in the following sense. For any three {\Lqcms s} $(\A_1,\Lip_1)$, $(\A_2,\Lip_2)$ and $(\A_3,\Lip_3)$ in $\mathcal{C}$, we have:
\begin{enumerate}
\item $\propinquity_{\mathcal{C}}((\A_1,\Lip_1),(\A_2,\Lip_2)) \in [0,\infty)$,
\item $\propinquity_{\mathcal{C}}((\A_1,\Lip_1),(\A_2,\Lip_2)) = \propinquity_{\mathcal{C}}((\A_2,\Lip_2),(\A_1,\Lip_1))$,
\item $\propinquity_{\mathcal{C}}((\A_1,\Lip_1),(\A_2,\Lip_2)) \leq \propinquity_{\mathcal{C}}((\A_1,\Lip_1),(\A_3,\Lip_3)) + \propinquity_{\mathcal{C}}((\A_3,\Lip_3),(\A_2,\Lip_2))$,
\item $\propinquity_{\mathcal{C}}((\A_1,\Lip_1),(\A_2,\Lip_2)) = 0$ if and only if there exists a isometric isomorphism $\varphi : (\A_1,\Lip_1) \rightarrow (\A_2,\Lip_2)$.
\end{enumerate}
\end{theorem}

\begin{proof}
Property (1) is a consequence of Proposition (\ref{first-bridge-prop}). Property (2)  and Property (3) are proven in Proposition (\ref{triangle-prop}). Property (4) is necessary, as shown in our main Theorem (\ref{main}) (and invoking Proposition (\ref{quantum-isometry-prop})) . Last, assume that there exists a isometric isomorphism $\varphi : (\A_1,\Lip_1)\rightarrow (\A_2,\Lip_2)$. By Definition (\ref{quantum-isometry-def}), the map $\varphi$ is a unital *-isomorphism. Define the bridge $\gamma = (\A_2,\unit_{\A_2},\varphi,\iota)$ where $\iota$ is the identity of $\A_2$. We check trivially that the length of $\gamma$ is zero, and thus Property (4) is sufficient as well.
\end{proof}

Our quantum propinquity $\propinquity$ is thus a new metric on the class of {\Lqcms s}. We now prove that it dominates the quantum Gro\-mov-Haus\-dorff distance. The key to this fact is the following construction.

\begin{notation}
The quantum Gro\-mov-Haus\-dorff distance introduced in \cite{Rieffel00} will be denoted by $\operatorname{dist}_q$.
\end{notation}

\begin{theorem}\label{Leibniz-lip-thm}
Let $(\A,\Lip_\A)$ and $(\B,\Lip_\B)$ be two {\Lqcms s}. Let $\gamma$ be a bridge from $(\A,\Lip_\A)$ to $(\B,\Lip_\B)$. Let $\varepsilon \geq 0$ be chosen so that:
\begin{equation}\label{leibniz-lip-thm-cond-1}
\bridgelength{\gamma}{\Lip_\A,\Lip_\B} + \varepsilon > 0\text{.}
\end{equation}
Define for all $(a,b) \in \A\oplus\B$:
\begin{equation}
\Lip_\varepsilon(a,b) = \max\left\{ \Lip_\A(a) ,\Lip_\B(b), \frac{1}{\bridgelength{\gamma}{\Lip_\A,\Lip_\B}+\varepsilon}\bridgenorm{\gamma}{a,b} \right\}\text{.}
\end{equation}
Then $\Lip_\varepsilon$ is an admissible Leibniz Lip-norm for $(\Lip_\A,\Lip_\B)$ and:
\begin{equation}
\Haus{\Kantorovich{\Lip_\varepsilon}} (\StateSpace(\A),\StateSpace(\B)) \leq 2\bridgelength{\gamma}{\Lip_\A,\Lip_\B}+\varepsilon \text{.}
\end{equation}
Consequently:
\begin{equation*}
\operatorname{dist}_q((\A,\Lip_\A),(\B,\Lip_\B)) \leq 2\bridgelength{\gamma}{\Lip_\A,\Lip_\B}\text{.}
\end{equation*}
\end{theorem}

\begin{proof}
Let $\gamma = (\D,\omega,\pi_\A,\pi_\B)$, and assume $\varepsilon \geq 0$ is chosen so that Inequality (\ref{leibniz-lip-thm-cond-1}) holds.

We first prove that $(\A\oplus\B,\Lip_\varepsilon)$ is a unital Lipschitz pair. First, note that $\bridgenorm{\gamma}{\cdot,\cdot}$ is continuous on $\A\oplus\B$ by construction. In particular, one concludes that the domain of $\Lip_\varepsilon$ is dense in $\A\oplus\B$.

Second, $\bridgenorm{\gamma}{\unit_\A,0} = \|\omega\|_\D \geq 1$, so for any $(a,b)\in\sa{\A}$, if $\Lip_\varepsilon(a,b) = 0$ then $(a,b) \in \R(\unit_\A,\unit_\B)$. Indeed, if $\Lip_\varepsilon(a,b) = 0$ for some $(a,b)\in\A\oplus\B$, then $\Lip_\A(a) = \Lip_\B(b) = \bridgenorm{\gamma}{a,b} = 0$, so there exists $s,t \in \R$ such that $a = s\unit_\A$ $b = t\unit_\B$ as $(\A,\Lip_\A)$ and $(\B,\Lip_\B)$ are unital Lipschitz pairs. Thus $\bridgenorm{\gamma}{s\unit_\A,t\unit_\B} = 0$. We have:
\begin{equation*}
\begin{split}
|s-t| &= \|\omega\|_\D^{-1}\bridgenorm{\gamma}{(s-t)\unit_\A,0}\\
&= \|\omega\|_\D^{-1} \bridgenorm{\gamma}{s\unit_\A - t\unit_\A, t\unit_\B-t\unit_\B} \\
&\leq \|\omega\|_\D^{-1}\left(\bridgenorm{\gamma}{s\unit_\A,t\unit_\B} + \bridgenorm{\gamma}{t\unit_\A,t\unit_\B}  \right)\\
&= 0\text{,}
\end{split}
\end{equation*}
where we used $\bridgenorm{\gamma}{\unit_\A,\unit_\B} = 0$. Thus $s=t$, i.e. $(a,b) \in \R\unit_{\A\oplus\B}$ as claimed. So $(\A\oplus\B,\Lip_\varepsilon)$ is a unital Lipschitz pair.

We now prove that $\Lip_\varepsilon$ is admissible for $(\Lip_\A,\Lip_\B)$. Let $a\in\sa{\A}$. If $\Lip_\A(a) = \infty$ then $\Lip_\varepsilon(a,b) = \infty$ for all $b\in\sa{\B}$, so $\Lip_\A(a) = \inf\{\Lip_\varepsilon(a,b) : b\in\sa{\B}\}$ immediately. If, instead, $\Lip_\A(a)=0$, then $a = t\unit_\A$ for some $t\in\R$. By construction, $\bridgenorm{\gamma}{t\unit_\A,t\unit_\B} = 0$, and thus $\Lip_\varepsilon(t\unit_\A,t\unit_\B) = 0$, as desired.

Assume now that $0<\Lip_\A(a)<\infty$, and let $a' = \Lip_\A(a)^{-1}a$. By definition of $\bridgelength{\gamma}{\Lip_\A,\Lip_\B}$, there exists $b' \in \sa{\B}$ such that $\bridgenorm{\gamma}{a',b'} \leq \bridgelength{\gamma}{\Lip_\A,\Lip_\B}+\varepsilon$ and $\Lip_\B(b')\leq 1$. Thus, if we let $b  = \Lip_\A(a)b'$ then:
\begin{equation*}
\Lip_\B(b) \leq \Lip_\A(a)\text{ and }\bridgenorm{\gamma}{a,b}\leq \Lip_\A(a)\left(\bridgelength{\gamma}{\Lip_\A,\Lip_\B}+\varepsilon\right)\text{.}
\end{equation*}
Thus, we have shown:
\begin{equation*}
\Lip_\A(a) = \inf \{ \Lip_\varepsilon(a,b) : b \in \sa{\B} \} \text{.}
\end{equation*}
The proof is symmetric in $\A$ and $\B$ and thus, we conclude that $\Lip$ is admissible for $\Lip_\A$ and $\Lip_\B$. In particular, the canonical injection of $\StateSpace(\A)$ into $\StateSpace(\A\oplus\B)$ is an isometry from $\Kantorovich{\Lip_\A}$ to $\Kantorovich{\Lip_\varepsilon}$, and similarly the injection from $(\StateSpace(\B),\Kantorovich{\Lip_\B})$ into $(\StateSpace(\A\oplus\B),\Kantorovich{\Lip_\varepsilon})$ is an isometry. We identify $\StateSpace(\A)$ and $\StateSpace(\B)$ with their images in $\StateSpace(\A\oplus\B)$ for the canonical injections.

By \cite[Theorem 5.2]{Rieffel00}, our seminorm $\Lip_\varepsilon$ thus constructed is a Lip-norm on $\A\oplus\B$, which is admissible for $(\Lip_\A,\Lip_\B)$.

Furthermore, for any $(a,b),(a',b')\in\sa{\A\oplus\B}$, since $(\A,\Lip_\A)$ is a {\Lqcms}:
\begin{equation*}
\begin{split}
\Lip_\A(\Jordan{a}{a'}) &\leq \|a\|_\A\Lip_\A(a')+\Lip_\A(a)\|a'\|_\A\\
&\leq \|(a,b)\|_{\A\oplus\B}\Lip_\varepsilon(a,b) + \Lip_{\varepsilon}(a',b')\|(a',b')\|_{\A\oplus\B}\text{.}
\end{split}
\end{equation*}
We obtain the same upper bound for $\Lip_\B(\Jordan{b}{b'})$ as $(\B,\Lip_\B)$ is a {\Lqcms} and for $\frac{\bridgenorm{\gamma}{\Jordan{a}{a'},\Jordan{b}{b'}}}{\bridgelength{\gamma}{\Lip_\A,\Lip_\B}+\varepsilon}$ by Inequality (\ref{bridgenorm-eq}). Thus: 
\begin{equation*}
\begin{split}
\Lip_\varepsilon(\Jordan{a}{a'},\Jordan{b}{b'}) &= \max\left\{\Lip_\A(\Jordan{a}{a'}),\Lip_\B(\Jordan{b}{b'}),\frac{\bridgenorm{\gamma}{\Jordan{a}{a'},\Jordan{b}{b'}}}{\bridgelength{\gamma}{\Lip_\A,\Lip_\B}+\varepsilon}\right\}\\
&\leq \|(a,b)\|_{\A\oplus\B}\Lip_\varepsilon(a',b') + \Lip_\varepsilon(a,b)\|(a',b')\|_{\A\oplus\B}\text{.}
\end{split}
\end{equation*}
The same computation holds for the Lie product. Thus $(\A\oplus\B,\Lip_\varepsilon)$ is a {\Lqcms}.

\bigskip

We now compute the Haus\-dorff distance between $\StateSpace(\A)$ and $\StateSpace(\B)$ as subsets of $(\StateSpace(\A\oplus\B),\Kantorovich{\Lip_\varepsilon})$, for the {\mongekant} associated with $\Lip_\varepsilon$. Let $\varphi \in\StateSpace(\A)$. By Definition (\ref{bridgelength-def}) and Definition (\ref{bridgeheight-def}), there exists $\psi \in \StateSpace(\D)$ such that $\psi(\omega d) = \psi(d \omega) = \psi(d)$ for all $d\in\D$ and $\Kantorovich{\Lip_\A}(\varphi,\psi\circ\pi_\A) \leq \bridgelength{\gamma}{\Lip_\A,\Lip_\B}$. Since $\Lip_\A$ is the quotient of $\Lip$ on $\sa{\A}$, we conclude that:
\begin{equation}\label{psi-d-eq0}
\Kantorovich{\Lip_\varepsilon}(\varphi,\psi\circ\pi_\A) \leq \bridgelength{\gamma}{\Lip_\A,\Lip_\B}\text{.}
\end{equation}

Let $(a,b) \in \sa{\A\oplus\B}$ such that $\Lip_\varepsilon(a,b)\leq 1$. By definition of $\Lip_\varepsilon$, we have:
\begin{equation*}
\begin{split}
|\psi\circ\pi_\A(a) - \psi\circ\pi_\B(b)| &= |\psi(\pi_\A(a)\omega) - \psi(\omega\pi_\B)|\\
&\leq \|\pi_\A(a)\omega - \omega\pi_\B(b)\|_\D =\bridgenorm{\gamma}{a,b}\\
&\leq \bridgelength{\gamma}{\Lip_\A,\Lip_\B} + \varepsilon\text{.}
\end{split}
\end{equation*}

Thus, by Definition (\ref{mongekant-def}), we have:
\begin{equation}\label{psi-d-eq1}
\Kantorovich{\Lip_\varepsilon}(\psi\circ\pi_\A,\psi\circ\pi_\B) \leq \bridgelength{\gamma}{\Lip_\A,\Lip_\B} + \varepsilon\text{.}
\end{equation}

Thus, using Inequality (\ref{psi-d-eq0}) and (\ref{psi-d-eq1}), we have:
\begin{equation*}
\Kantorovich{\Lip_\varepsilon}(\varphi,\psi\circ\pi_\B) \leq 2\bridgelength{\gamma}{\Lip_\A,\Lip_\B} + \varepsilon\text{.}
\end{equation*}

Since $\psi\circ\pi_\B \in \StateSpace(\B)$ by construction, we have shown that $\varphi$ lies within the $(2\bridgelength{\gamma}{\Lip_\A,\Lip_\B}+\varepsilon)$-neighborhood of $\StateSpace(\B)$ for $\Kantorovich{\Lip_\varepsilon}$. As $\varphi \in \StateSpace(\A)$ was arbitrary, we conclude tat $\StateSpace(\A)$ lies within the $(2\bridgelength{\gamma}{\Lip_\A,\Lip_\B}+\varepsilon)$-neighborhood of $\StateSpace(\B)$ for $\Kantorovich{\Lip_\varepsilon}$. By symmetry in our proof for the roles of $\A$ and $\B$, we conclude that:
\begin{equation*}
\Haus{\Kantorovich{\Lip_\varepsilon}}(\StateSpace(\A),\StateSpace(\B)) \leq 2\bridgelength{\gamma}{\Lip_\A,\Lip_\B} + \varepsilon\text{.}
\end{equation*}
This proves the first statement of our Theorem.

Now:
\begin{equation*}
\operatorname{dist}_q((\A,\Lip_\A),(\B,\Lip_\B)) \leq 2\bridgelength{\gamma}{\Lip_\A,\Lip_\B} + \varepsilon\text{.}
\end{equation*}
As $\varepsilon \geq 0$ is arbitrary, the second statement of our theorem is proven as well. We remark that if $\bridgelength{\gamma}{\Lip_\A,\Lip_\B} > 0$ then we can pick $\varepsilon = 0$, and thus we have an admissible Lip-norm $\Lip_0$ which makes the bound in our theorem sharp.
\end{proof}

We can now prove the following important result:

\begin{corollary}\label{propinquity-dominates-cor}
Let $\mathcal{C}$ be any nonempty class of {\Lqcms s}. For any $(\A,\Lip_\A), (\B,\Lip_\B)\in \mathcal{C}$, we have:
\begin{equation*}
\operatorname{dist}_q((\A,\Lip_\A),(\B,\Lip_\B)) \leq 2 \propinquity_{\mathcal{C}}((\A,\Lip_\A),(\B,\Lip_\B))\text{.}
\end{equation*}
In particular, convergence of a net for the quantum propinquity implies convergence of the same net to the same limit for the quantum Gro\-mov-Haus\-dorff distance.
\end{corollary}

\begin{proof}
Let $\varepsilon > 0$. By Definition (\ref{propinquity-def}), there exists a $\mathcal{C}$-trek
\begin{equation*}
\Gamma = (\A_j,\Lip_j,\gamma_j,\A_{j+1},\Lip_{j+1} : j = 1,\ldots,n)
\end{equation*}
such that:
\begin{equation*}
\treklength{\Gamma} \leq \propinquity_{\mathcal{C}}((\A,\Lip_\A),(\B,\Lip_\B)) + \frac{\varepsilon}{2}\text{.}
\end{equation*}
Now, since $\operatorname{dist}_q$ satisfies the triangle inequality, we have:
\begin{equation*}
\begin{split}
\operatorname{dist}_q((\A,\Lip_\A),(\B,\Lip_\B)) &\leq \sum_{j=1}^n \operatorname{dist}_q((\A_j,\Lip_j),(\A_{j+1},\Lip_{j+1}))\\
&\leq 2 \sum_{j=1}^n \bridgelength{\gamma_j}{\Lip_j,\Lip_{j+1}} \text{ by Theorem (\ref{Leibniz-lip-thm}),}\\
&\leq 2\propinquity_{\mathcal{C}}((\A,\Lip_\A),(\B,\Lip_\B)) + \varepsilon\text{.}
\end{split}
\end{equation*}
As $\varepsilon > 0$, we have proven our corollary.
\end{proof}

As a remark, we can of course construct natural Lip-norms which gives us the result of Corollary (\ref{propinquity-dominates-cor}). Let $(\A,\Lip_\A)$ and $(\B,\Lip_\B)$ be two {\Lqcms s}, and $\Gamma$ be a trek from $(\A,\Lip_\A)$ and $(\B,\Lip_\B)$. Write $\Gamma = (\A_j,\Lip_j,\gamma_j,\A_{j+1},\Lip_{j+1} : j=1,\ldots,n)$. For any $\varepsilon > 0$, define:

\begin{equation*}
\treknorm{\Gamma,\varepsilon}{a,b} = \inf \left\{ \max\left\{ \frac{\bridgenorm{\gamma_j}{\eta_j,\eta_{j+1}}}{\bridgelength{\gamma_j}{\Lip_j,\Lip_{j+1}} + \frac{\varepsilon}{n}}  \right\} : \eta \in \itineraries{\Gamma}{a}{b}  \right\} \text{.}
\end{equation*}

Then define, for all $(a,b)\in\sa{\A\oplus\B}$:
\begin{equation*}
\Lip_\varepsilon(a,b) = \max\{ \Lip_\A(a),\Lip_\B(b), \treknorm{\Gamma,\varepsilon}{a,b} \}\text{.}
\end{equation*}
Then the reader may check that $\Lip_\varepsilon$ is an admissible Lip-norm on $\sa{\A\oplus\B}$ for $(\Lip_\A,\Lip_\B)$. Moreover, if $\Gamma$ is chosen as in the proof of Corollary (\ref{propinquity-dominates-cor}), then:
\begin{equation*}
\Haus{\Kantorovich{\Lip_\varepsilon}}(\StateSpace(\A),\StateSpace(\B)) \leq 2\propinquity((\A,\Lip_\A),(\B,\Lip_\B)) + \varepsilon\text{.}
\end{equation*}
Yet, even though $\Lip_\A$ and $\Lip_\B$ are Leibniz, the Lip-norm $\Lip_\varepsilon$ may not be. However, in this case, we can find finitely many Leibniz Lip-norms and intermediate {\Lqcms s} given by Theorem (\ref{Leibniz-lip-thm}) to connect $(\A,\Lip_\A)$ and $(\B,\Lip_\B)$.

A careful inspection of the proof of Theorem (\ref{main}) reveals that the key to our quantum propinquity's behavior is that it provides not only a control on the Lip-norms, but also a control on the norms of elements. More precisely, we get the following quantum version of the Lipschitz extension property: if
\begin{equation*}
\propinquity((\A,\Lip_\A),(\B,\Lip_B)) = d\text{,}
\end{equation*}
then for all $a\in \sa{\A}$ and $\varepsilon > 0$, we can find $b \in \sa{\B}$ such that:
\begin{equation*}
\Lip_\B(b)\leq \Lip_\A(a)\text{ and }\|b\|_\B\leq \|a\|_\A + 2\Lip(a)(d + \varepsilon)\text{,}
\end{equation*}
using Proposition (\ref{fundamental-prop}). Thus, we have some control on how far from Leibniz the trek norms are. 

\begin{remark}
We note that the Lip-norms constructed in Theorem (\ref{Leibniz-lip-thm}) are strong Leibniz if $\Lip_\A$, $\Lip_\B$ are strong Leibniz, using the notations of (\ref{Leibniz-lip-thm}). Thus, the quantum propinquity provides strong Leibniz Lip-norm. However, Rieffel's proximity \cite{Rieffel10c} does not satisfy, as far as we know, the triangle inequality, so it is unclear whether the quantum propinquity dominates Rieffel's proximity.
\end{remark}

We conclude this section by comparing the quantum propinquity with the Gro\-mov-Haus\-dorff distance. It is proven in \cite{Rieffel00} that the Gro\-mov-Haus\-dorff distance dominates the quantum Gro\-mov-Haus\-dorff distance restricted to the compact quantum metric spaces given by Example (\ref{classical-ex}). Thus Theorem (\ref{propinquity-dominates-cor}) is not enough to compare our quantum propinquity with the Gro\-mov-Haus\-dorff distance. Yet, we have:

\begin{theorem}\label{gh-thm}
Let $(X,\mathsf{d}_X)$ and $(Y,\mathsf{d}_Y)$ be two compact metric spaces, and let $\mathsf{GH}$ be the Gro\-mov-Haus\-dorff distance \cite{Gromov}. Then:
\begin{equation*}
\propinquity((C(X),\Lip_X),(C(Y),\Lip_Y)) \leq \mathsf{GH}((X,\mathsf{d}_X),(Y,\mathsf{d}_Y))\text{,}
\end{equation*}
where $\Lip_X$ and $\Lip_Y$ are, respectively, the Lipschitz seminorms associated to $\mathsf{d}_X$ and $\mathsf{d}_Y$.
\end{theorem}

\begin{proof}
Let $\delta = \mathsf{GH}((X,\mathsf{d}_X),(Y,\mathsf{d}_Y))$. Let $\varepsilon > 0$. By definition of the Gro\-mov-Haus\-dorff distance, there exists a distance $\mathsf{d}$ on the disjoint union $X\coprod Y$ of $X$ and $Y$ which restricts to $\mathsf{d}_X$ on $X$ and $\mathsf{d}_Y$ on $Y$, and such that the Haus\-dorff distance for $\mathsf{d}$ between $X$ and $Y$ is less than $\delta+\varepsilon$.
\begin{equation*}
Z = \{ (x,y) \in X\times Y : \mathsf{d}(x,y) \leq \delta + 2\varepsilon \}\text{.}
\end{equation*}
We endow $Z$ with the restriction of the product topology on $X\times Y$. Thus, $Z$ is easily checked to be compact, and moreover, by our choice of $\delta$ and $\mathsf{d}$, the canonical surjection $X\times Y\twoheadrightarrow X$ and $X\times Y\twoheadrightarrow Y$ are still surjections when restricted to $Z$, which we denote, respectively, by $\rho_X$ and $\rho_Y$. We define $\pi_X : f\in C(X)\mapsto f\circ\rho_X$ and $\pi_Y : f\in C(Y) \mapsto f\circ\rho_Y$. Thus, by definition, $\gamma = (C(Z),1,\pi_X,\pi_Y)$ is a bridge, whose height is null by construction.

Let $f \in \sa{C(X)}$ with $\Lip_X(f)\leq 1$. Let $\tilde{f} : X\coprod Y \rightarrow \R$ be a $1$-Lipschitz extension of $f$ to $(X\coprod Y,\mathsf{d})$. Let $g$ be the restriction of $\tilde{f}$ to $Y$ and note that $\Lip_Y(g)\leq 1$. We have, for all $(x,y) \in Z$:
\begin{equation*}
| f(x) - g(y) | = |\tilde{f}(x) - \tilde{f}(y)| \leq \mathsf{d}(x,y)\leq \delta + 2\varepsilon \text{.}
\end{equation*}
Hence, $\bridgelength{\gamma}{\Lip_X,\Lip_Y} \leq \delta + 2\varepsilon$. As $\varepsilon>0$ was arbitrary, we have proven our theorem.
\end{proof}

\begin{remark}
Let $\mathcal{C}$ be a class of {\Lqcms s} which include all the {\Lqcms s} given by Example (\ref{classical-ex}). Then Theorem (\ref{gh-thm}) adapts immediately to $\propinquity_{\mathcal{C}}$ as well.
\end{remark}

Last, we note that the quantum tori and the fuzzy tori form a continuous family for the quantum propinquity, thus strengthening our result in \cite{Latremoliere05}. The techniques we developed in \cite{Latremoliere05} are the basis upon which this result is proven. In \cite[Remark 5.5]{li03}, it was shown that when restricted to the class of nuclear C*-algebras, the nuclear distance $\dist_{\mathrm{nu}}$ satisfies \cite[Theorem 5.2, Theorem 5.3]{li03} and thus, in turn, our work in \cite{Latremoliere05} applies to show that the quantum and fuzzy tori form a continuous family for the nuclear distance. 

However, the nuclear distance may not dominate our quantum propinquity. The unital nuclear distance is a modified version of Li's nuclear distance, offered by Kerr and Li in \cite{Kerr09}. As pointed out in \cite{Kerr09}, fuzzy and quantum tori form a continuous family in the sense described below in Theorem (\ref{qt-thm}) for the unital nuclear distance as well. The argument is similar to the one for the nuclear distance. Indeed, a deep result of Blanchard \cite{Blanchard97}, which builds upon a remarkable observation of Haagerup and R{\o}rdam \cite{Rordam95}, shows that continuous fields of nuclear C*-algebras over compact metric spaces can be sub-trivialized \cite{Blanchard97}.

Since the quantum propinquity is dominated by the unital nuclear distance \cite{Kerr09}, we obtain the following analogue of \cite[Theorem 3.13]{Latremoliere05}:

\begin{theorem}\label{qt-thm}
Let $H_\infty$ be a compact Abelian group endowed with a continuous length function $\ell$. Let $(H_n)_{n\in\N}$ be a sequence of closed subgroups of $H_\infty$ converging to $H_\infty$ for the Hausdorff distance induced by $\ell$ on the class of closed subsets of $H_\infty$. Let $\sigma_\infty$ be a skew bicharacter of the Pontryagin dual $\widehat{H_\infty}$. For each $n\in\N$, we let $\sigma_n$ be a skew bicharacter of $\widehat{H_n}$, which we implicitly identity with its unique lift as a skew bicharacter of $\widehat{H_\infty}$. If $(\sigma_n)_{n\in\N}$ converges pointwise to $\sigma_\infty$, then:
\begin{equation*}
\lim_{n\rightarrow\infty} \propinquity\left(\left(C^\ast\left(\widehat{H_n},\sigma_n\right),\Lip_n\right),\left(C^\ast\left(\widehat{H_\infty},\sigma_\infty\right),\Lip_\infty\right)\right) = 0\text{,}
\end{equation*}
where for all $n\in\N\cup\{\infty\}$ and $a\in C^\ast\left(\widehat{H_n},\sigma_n\right)$ we set:
\begin{equation*}
\Lip_n (a) = \sup\left\{\frac{\|a-\alpha_n^g(a)\|_{C^\ast\left(\widehat{H_n},\sigma_n\right)}}{\ell(g)} : g\in H_\infty\setminus\{1_{H_\infty}\} \right\}
\end{equation*}
with $1_{H_\infty}$ is the unit of $H_\infty$ and $\alpha_n$ is the dual action of $H_n$ on $C^\ast\left(\widehat{H_n},\sigma_n\right)$.
\end{theorem}



\bibliographystyle{amsplain}
\providecommand{\bysame}{\leavevmode\hbox to3em{\hrulefill}\thinspace}
\providecommand{\MR}{\relax\ifhmode\unskip\space\fi MR }
\providecommand{\MRhref}[2]{%
  \href{http://www.ams.org/mathscinet-getitem?mr=#1}{#2}
}
\providecommand{\href}[2]{#2}

\vfill
\end{document}